\documentclass[11pt,reqno]{amsart}
\usepackage[left=1.3in, right=1.2in, top=1in, bottom=1in, includefoot]{geometry}
\usepackage{anyfontsize}
\usepackage{graphicx}
\usepackage{mathrsfs} 
\usepackage{amssymb}   
\usepackage{amsthm}  
\usepackage{bbm}
\usepackage[shortlabels]{enumitem}
\usepackage{xcolor}
\usepackage{soul}

\newtheorem{thm}{Theorem}[section]

\newtheorem{prop}[thm]{Proposition}
\newtheorem{cor}[thm]{Corollary}
\newtheorem{Def}[thm]{Definition}
\newtheorem{lemma}[thm]{Lemma}
\newtheorem{remark}[thm]{Remark}

\newcommand{\ve}{{\varepsilon}}
\newcommand{\N}{{\mathbb{N}}}
\newcommand{\R}{{\mathbb{R}}}
\numberwithin{equation}{section}
\title{Long time well-posedness and full justification of a Whitham-Green-Naghdi system} 
\author[L. Emerald]{Louis Emerald}
\author[M. Oen Paulsen]{Martin Oen Paulsen}

\address{Department of Mathematics, Nazarbayev University, Nur-Sultan 010000, Kazakhstan}
\email{Louisemerald76@gmail.com}
\address{Department of Mathematics\\ University of Bergen\\ Postbox 7800\\ 5020 Bergen\\ Norway}
\email{Martin.Paulsen@UiB.no}

\date{\today}
\keywords{Fully dispersive Green-Naghdi system; Rigorous justification; bathymetry}
\subjclass[2010]{Primary: 35Q35; Secondary: 76B15}
\begin{document}

    \begin{abstract} 
        We establish the full justification of a \lq\lq Whitham-Green-Naghdi\rq\rq\:  system  modeling the propagation of surface gravity waves with bathymetry in the shallow water regime. It is an asymptotic model of the water waves equations with the same  dispersion relation.  The model under study is a nonlocal quasilinear symmetrizable hyperbolic system without surface tension.   We prove the consistency of the general water waves equations with our system at the order of precision $O(\mu^2 (\ve + \beta))$, where $\mu$ is the shallow water parameter, $\ve$ the nonlinearity parameter, and $\beta$ the topography parameter. Then we prove the long time well-posedness on a time scale $O(\frac{1}{\max\{\ve,\beta\}})$. Lastly, we show the convergence of the solutions of the Whitham-Green-Naghdi system to the ones of the water waves equations on the later time scale.
    \end{abstract}
    \maketitle
    \section{Introduction}
  
    In this article, we study a full dispersion Green-Naghdi system that describes strongly dispersive surface waves over a variable bottom. The system under consideration is described in terms of the unknowns $\zeta$, $v$, and $b$. Here $\zeta(t,x)\in \R$  denotes the surface elevation, $v(t,x) \in \R$ is related to the velocity field described by the full Euler equations, and $b$ is the elevation of the bathymetry. The system reads,
    \begin{align}\label{W-G-N}
    	\begin{cases}
    		\partial_t \zeta + \partial_x(hv) = 0
    		\\ 
    		(h  + \mu h \mathcal{T}[h,\beta  b]) \big{(}\partial_t v + \ve v \partial_x v\big{)} + h \partial_x \zeta +  \mu \varepsilon h(\mathcal{Q}[h,v] + \mathcal{Q}_b[h,b,v]) = 0,
    	\end{cases}
    \end{align}
    where $h = 1+ \ve \zeta - \beta b$ and
    \begin{align}\label{opT}
    	\mathcal{T} [h,\beta b] v
    	& = 
    	-\frac{1}{3h} \partial_x \mathrm{F}^{\frac{1}{2}} \big{(} h^3  \mathrm{F}^{\frac{1}{2}}\partial_x v\big{)} 
    	+
    	 \frac{1}{2h} \big{(} \partial_x \mathrm{F}^{\frac{1}{2}} (h^2 (\beta \partial_x b)v) - h^2(\beta \partial_x b)   \mathrm{F}^{\frac{1}{2}}  \partial_x  v\big{)} 
    	 \\
    	 & \hspace{0.5cm} \notag
    	 +
    	 (\beta \partial_x b)^2 v,
    \end{align}
    and
    \begin{align}\label{Q}
    	\mathcal{Q}[h,v] 
    	& =
    	\frac{2}{3h} \partial_x \mathrm{F}^{\frac{1}{2}}\big{(} h^3  (\mathrm{F}^{\frac{1}{2}}\partial_xv)^2\big{)} 
    	\\
    	\label{B}
    	\mathcal{Q}_b[h,\beta b,v]
    	 & =   
    	h (\mathrm{F}^{\frac{1}{2}} \partial_x  v)^2 (\beta \partial_x b) 
    	+
    	\frac{1}{2h} \partial_x \mathrm{F}^{\frac{1}{2}} (h^2 v^2 \beta \partial_x^2 b) 
    	+
    	v^2(\beta \partial_x^2 b)(\beta \partial_x b),
    \end{align}
    with $\mathrm{F}^{\frac{1}{2}}$ being a Fourier multiplier associated with the dispersion relation of the water waves system. Specifically, if we let $\hat{f}(\xi)$ be the Fourier transform of $f$, then the symbol is defined in frequency by
    \begin{equation}\label{F}
    	\widehat{\mathrm{F}^{\frac{1}{2}} f} (\xi)=  \sqrt{\frac{3}{\mu \xi^2}\Big{(} \frac{\sqrt{\mu}\xi}{\tanh(\sqrt{\mu}\xi)}-1\Big{)}} \hat{f}(\xi).
    \end{equation}
    The parameters $\mu, \ve$, and $\beta$ are defined by the comparison between characteristic quantities of the system under study. Among those are the characteristic water depth $H_0$, the characteristic wave amplitude $a_{s}$, the characteristic bathymetry amplitude $a_b$, and the characteristic wavelength $L$. From these comparisons appear three adimensional parameters of main importance:
    \begin{itemize}
        \item $\mu := \frac{H_0^{2}}{L^{2}}$ is the shallow water parameter,
        \item $\varepsilon  := \frac{a_{\mathrm{s}}}{H_0}$  is the nonlinearity parameter,
        \item $\beta:= \frac{a_b}{H_0}$  is the bathymetry parameter.
    \end{itemize}
    Replacing the Fourier multiplier $\mathrm{F}^{\frac{1}{2}}$ by identity in system \eqref{W-G-N} we retrieve the classical Green-Naghdi system. The later system is proved to be consistent with the water waves equations, in the sense of Definition 5.1 in \cite{WWP}, at the order of precision $O(\mu^2)$ for parameters $(\mu,\varepsilon, \beta)$ in the shallow water regime: 
    \begin{Def}
        Let $\mu_{\max} > 0$, then we define the shallow water regime to be 
            \begin{align*}
                \mathcal{A}_{\mathrm{SW}} := \{ (\mu,\varepsilon,\beta) : \mu \in (0,\mu_{\max}], \varepsilon \in [0,1], \beta \in [0,1] \}.
            \end{align*}
    \end{Def}
    Taking $\ve$ to be zero in \eqref{W-G-N}, we get the linearized water waves equations around the rest state with the following dispersion relation
    \begin{equation}\label{Disp relation WW}
        \omega_{WW}(\xi)^2 = \xi^2\frac{\tanh(\sqrt{\mu}\xi)}{\sqrt{\mu}\xi}. 
    \end{equation}
    This is why we say that system \eqref{W-G-N} is a full dispersion Green-Naghdi model. Moreover, it is proved in the present paper that the water waves equations are consistent, in the sense of Proposition \ref{Concictency of new model}, with system \eqref{W-G-N} at the order of precision $O(\mu^2(\varepsilon+\beta))$. The improved precision compared to the classical Green-Naghdi system allows for a change in the propagation of the waves. Such occurrences have been studied in the Dingemans experiments \cite{Dingemans94}. In these experiments, they investigated a long wave passing over a submerged obstacle. They observed that waves tend to steepen due to a compression effect from the bottom,  where high harmonics generated by topography-induced nonlinear interactions are freely released behind the obstacle. This last phenomenon makes it natural that one wants to improve the frequency dispersion of the classical shallow water models. Deriving such models has been the subject of active research. Here are some references in the case of the Boussinesq model \cite{GobbiKirby00, MadsenBingham02, ChazelBenoit09}. In the case of the Green-Naghdi model, one can consult \cite{WeiKirbyGrilliEtAl95} and \cite{ChazelLannes11}, where the authors compared the classical Green-Naghdi model with one-parameter and three-parameters Green-Naghdi models in one case of the Dingemans experiments for which the propagation and interaction of highly dispersive waves are under study. By tuning the parameters, they are able to describe the dispersion relation of the water waves equations for a larger set of frequencies. As an example, the dispersion relation of the three-parameter model is
    \begin{align}\label{Disp relation GN}
        \omega_{GN}(\xi)^2 = \xi^2 \frac{(1 + \mu \frac{\theta + \gamma}{3} \xi^2)(1 + \mu \frac{\alpha -1}{3}\xi^2)}{(1+ \mu \frac{\gamma}{3}\xi^2)(1 + \mu \frac{\alpha + \theta}{3}\xi^2)},
    \end{align}
    where the parameters $\alpha, \gamma$ and $\theta$ are chosen such that \eqref{Disp relation GN} approximates well the dispersion relation of the water waves equations, \eqref{Disp relation WW}, for higher frequencies. In particular, for $(\theta, \alpha ,\gamma) = (-1, 1,1)$ we obtain the original Green-Naghdi system. Moreover, in the case $(\theta, \alpha, \gamma) = (0.207, 1,  0.071 )$ it was demonstrated in \cite{ChazelLannes11}, that \eqref{Disp relation GN} is a better approximation of \eqref{Disp relation WW}  (see Figure \ref{fig: disp}). This improvement allowed the authors to describe strongly dispersive waves with uneven bathymetry accurately.
    \begin{figure}[h!]\label{fig: disp}
		\hspace{-0.9cm}
		\includegraphics[scale=0.38]{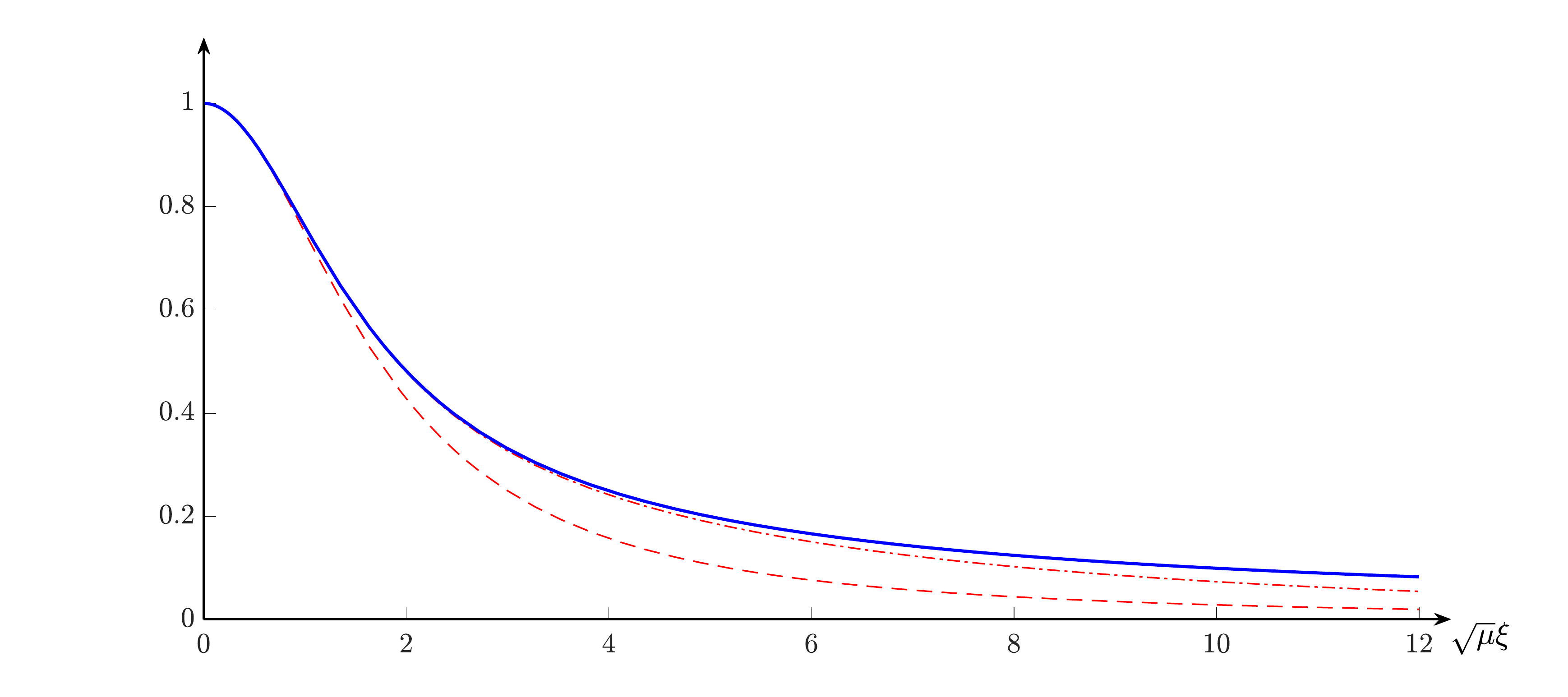}
		\centering
		\caption{\small The blue curve is a plot of $\omega_{WW}(\xi)/\xi^2$  (line). The red curves plots $\omega_{GN}(\xi)/\xi^2$ in the case $(\theta,\alpha,\gamma)= (-1,1,1)$ (dash) and $(\theta,\alpha,\gamma)= (0.207,1,0.071)$ (dash-dots).}
	\end{figure}
    In fact, in the case where high frequencies are dominant, the improved Green-Naghdi models tend to describe the propagation of the waves more correctly. However, in general, one can expect to have even higher frequency interactions for which one needs to keep the full dispersion relation of the water waves equations.

    The first full dispersion model, called the Whitham equations, was introduced by Whitham in \cite{Whitham67} to study breaking waves and Stokes waves of maximal amplitude. The existence of these phenomena for this model has been proved in the recent papers \cite{EhrnstromWahlen19, Hur17, SautWang22, TruongWahlen22}. The Whitham is a classical model in oceanography and can be seen as a modified version of the Kordeweg-de Vries equations with lower frequency dispersion. In addition, the existence of periodic waves was proved in \cite{EhrstromKalisch09}, and the existence of Benjamin-Feir instabilities was demonstrated in \cite{HurJohnson15, SanfordJodamaKalisch14}. See also the series of papers on the stability of traveling waves \cite{Arnesen16,EhrnstromGroves12,JohnsonDouglas20, StefanovWright20},

    The study of bidirectional full dispersion models for a flat bottom has also been the subject of active research. One class of such systems is the Whitham-Boussinesq ones. They are the full dispersion versions of the Boussinesq system, meaning they have the same dispersion relation as the water waves equations \eqref{Disp relation WW}. Like the Whitham equation, these type of systems features solitary waves \cite{DinvayNilsson21, NilssonWang19}, Benjamin-Feir instabilities \cite{HurPandey16,Pandey19, VargasMaganaPanayotaros16}, high-frequency instabilities of small-amplitude periodic traveling waves \cite{EhrnstromJohnson19}. See also some comparative studies between the Boussinesq and the Whitham-Boussinesq models \cite{Carter18,DinvayDutykhKalisch19}.

    The full dispersion Whitham-Green-Naghdi models are next order approximations of the water waves equations when compared to the Whitham-Boussinesq systems. These systems were recently derived in \cite{Emerald21} for a flat bottom and extended to include bathymetry in \cite{DucheneMMWW21}. See also \cite{DucheneIsrawiTalhouk16} where the authors derived a two-layers Whitham-Green-Naghdi system. There is still a lot of research left to be done on the study of qualitative properties of these systems, but we mention the work of Duchene et. al \cite{DucheneNilssonWahlen18}, which proved the existence of solitary waves where they consider both surface gravity waves and internal waves. 

    An important part of the study of the full dispersion systems is the full justification as an asymptotic model of the water waves equations in the shallow water regime. To be more precise, we say a model is fully justified if the following points are proven:
    \begin{itemize}
        \item The solutions of the water waves equations exist on the scale $O(\frac{1}{\max\{\ve,\beta\}})$.
        \item The solutions of the asymptotic model exist on the scale $O(\frac{1}{\max\{\ve,\beta\}})$.
        \item Solutions of the water waves equations solve the asymptotic model up to remainder terms of a specified order of precision in terms of the adimensional parameters $\mu,\varepsilon$, and $\beta$. This last point is called the consistency of the water waves equations with respect to the asymptotic model.
        \item By virtue of the previous points, 
        one has to show that the difference between the solutions of the water waves equations and the asymptotic model satisfies an error estimate depending polynomially on $\mu, \ve$ and $\beta$. 
    \end{itemize}
    If we can verify these four points, then we can compare solutions of the water waves equations with solutions of the asymptotic models up to times of order $O(\frac{1}{\max\{\ve,\beta\}})$. The first point is proved by Alvarez-Samaniego and Lannes in \cite{Alvarez-SamaniegoLannes08a}.

    The three remaining points are specific to the asymptotic model. For instance, in the case of the Whitham equation, the local well-posedness in the relevant time scale follow by classical arguments on hyperbolic systems. The consistency of the water waves equations with this model has been recently proved in \cite{Emerald21b} at the order of precision $O(\mu\varepsilon)$ in the unidirectional case, but the method supposes well-prepared initial conditions. In the bidirectional case, the author proved an order of precision $O(\mu\varepsilon + \varepsilon^2)$ and doesn't suppose well-prepared initial conditions. In conclusion, we have the full justification of the Whitham equation at the order of precision $O(\mu\varepsilon)$ in the unidirectional case under the restriction of well-prepared initial conditions. In the bidirectional case, the order of precision is $O(\mu\varepsilon + \varepsilon^2)$. 

    Regarding the Whitham-Boussinesq systems for flat bottoms, the consistency of the water waves equations with the later models has been proved in \cite{Emerald21} with an order of precision $O(\mu\varepsilon)$ in the shallow water regime. When nonflat bottoms are considered, it has been proved in \cite{DucheneMMWW21} to be consistent with the water waves with a precision $O(\mu(\ve + \beta))$. With respect to the second point of the justification, it has been proved for a large class of Whitham-Boussinesq systems with flat bottoms \cite{Paulsen22, Emerald22}, to be well-posed on the time scale $O(\frac{1}{\ve})$. Lastly, we also mention earlier results on the local-well posedness on a fixed time scale given in \cite{Dinvay19, Dinvay20, DinvaySelbergTesfahun19}. 

    For the Whitham-Green-Naghdi systems, it is proved in \cite{Emerald21} that for a flat bottom, the water waves equations are consistent with the later systems at the order of precision $O(\mu^2\varepsilon)$ in the shallow water regime. Moreover, in the case of uneven bathymetry, it has been proved in \cite{DucheneMMWW21} that the precision order is $O(\mu^2(\varepsilon + \beta))$. In \cite{DucheneIsrawiTalhouk16}, the authors proved the local well-posedness with a relevant time scale for a two-layer full dispersion Green-Naghdi model with surface tension. This system can be seen as a generalization of \eqref{W-G-N}. However, their method relies on adding surface tension, where the time of existence tends to zero as the surface tension parameter goes to zero. Moreover, this system has only been proved to be consistent with the water waves equations at the order of precision $O(\mu^2)$ even if, based on numerical experiments, the expected seems to be $O(\mu^2\varepsilon)$. 

    %
    %
    %
    In the present paper, we  prove the full justification of the Whitham-Green-Naghdi system without surface tension \eqref{W-G-N} as an asymptotic model of the water waves equations at the order of precision $O(\mu^2(\varepsilon+\beta))$.

    \subsection{Definition and notations}
	\begin{itemize}
		\item We  let $c$ denote a positive constant independent of $\mu, \ve, \beta$ that may change from line to line. Also, as a shorthand, we use the notation $a \lesssim b$ to mean $a \leq c\: b$.

        \item Let $s \in \R$ then the function $\lceil s \rceil$ returns the smallest integer greater than or equal to $s$.

		\item Let $L^2(\R)$ be the usual space of square integrable functions with norm $|f|_{L^2} = \sqrt{\int_{\R} |f(x)|^2 \: dx}$. Also, for any $f,g \in L^2(\R)$ we denote the scalar product by $\big{(} f,g \big{)}_{L^2} = \int_{\R} f(x) \overline{g(x)} \: dx$.
		

		\item Let $f:\mathbb{R} \to \mathbb{R}$ be a tempered distribution, let $\hat{f}$ or $\mathcal{F}f$ be its Fourier transform. Let $G:\mathbb{R}\to\mathbb{R}$ be a smooth function. Then the Fourier multiplier associated with $G(\xi)$ is denoted $\mathrm{G}$ and defined by the formula:
        \begin{align*}
            \widehat{\mathrm{G}f}(\xi) = G(\xi)\hat{f}(\xi).
        \end{align*}

		\item For any $s \in \mathbb{R}$ we call the multiplier $ \widehat{\mathrm{D}^{s}f} (\xi)= |\xi|^s \hat{f}(\xi)$ the Riesz potential of order $-s$. 

		\item For any $s \in \mathbb{R}$ we call the multiplier $\mathrm{J}^s = (1+\mathrm{D}^2)^{\frac{s}{2}} = \langle \mathrm{D} \rangle^s$ the Bessel potential of order $-s$.

            \item The Sobolev space $H^s(\mathbb{R})$  is equivalent to the weighted $L^2-$space with $|f|_{H^s} = |\mathrm{J}^s f|_{L^2}$. 
            
            \item For any $s\geq 1$ we will denote $\dot{H}^s(\mathbb{R})$ the Beppo Levi space  with $|f|_{\dot{H}^s} = |\mathrm{J}^{s-1}\partial_x f|_{L^2}$. 
    
        %
        %
        %
        %
        \item Let $k\in\mathbb{N}, l\in\mathbb{N}$ and $m \in \mathbb{N}$. A function $R$ is said to be of order $O(\mu^k(\varepsilon^l+ \beta^m))$, denoted $R=O(\mu^k(\varepsilon^l+\beta^m)),$ if divided by $\mu^k(\varepsilon^l+\beta^m)$ this function is uniformly bounded with respect to $(\mu,\varepsilon,\beta) \in \mathcal{A}_{\mathrm{SW}}$ in the Sobolev norms $|\cdot|_{H^{s}}$, $s \geq 0$.

		\item We say $f$ is a  Schwartz function $\mathscr{S}(\mathbb{R})$, if $f \in C^{\infty}(\mathbb{R})$ and satisfies for all $j,k  \in \mathbb{N}$,
		\begin{equation*}
			\sup \limits_{x} |x^{j} \partial_x^{k} f | < \infty.
		\end{equation*}
		
		\item If  $A$ and $B$ are two operators, then we denote the commutator between them to be $[A,B] = AB - BA$.	
	\end{itemize}

    \subsection{Main results}
    Throughout this paper, we will always make the following fundamental assumption. 
    \begin{Def}[Non-cavitation assumption]\label{nonCavitation}  Let $s >\frac{1}{2}$, $\ve \in (0,1)$ and $\beta\geq 0$. We say the initial surface elevation $\zeta_0  \in H^s(\mathbb{R})$ and  the bottom profile $b\in L^{\infty}(\R)$ satisfies the \lq\lq non-cavitation assumption\rq\rq\: if there exist $h_0\in(0,1)$ such that 
		\begin{equation}\label{NonCav}
			1+\ve\zeta_0(x) - \beta b(x) \geq h_0, \quad \text{for all} \: \: \: x\in \mathbb{R}.
		\end{equation}
	\end{Def}

    Next, before we state the main results, we define the energy space associated to \eqref{W-G-N}.
 	
	\begin{Def}\label{Function space}
		We define the complete function space  $Y^s_{\mu}(\mathbb{R}^d) = H^s(\R) \times X^s_{\mu}(\R)$, where $X^s_{\mu}(\R)$ is a subspace of $H^{s+\frac{1}{2}}(\R)$ equipped with the norm
            \begin{equation*}
                |v|_{X^s_{\mu}}^2 := |v|_{H^s}^2 + \sqrt{\mu} |\mathrm{D}^{\frac{1}{2}} v|_{H^{s}}^2,
            \end{equation*}
            and we make the definition
		\begin{equation*}
			|(\zeta, v)|_{Y^s_{\mu}}^2: = |\zeta|_{H^s}^2 + |v|_{X^s_{\mu}}^2.
		\end{equation*}
            
	\end{Def}

    The following Theorem is one of the main results of the paper and concerns the local well-posedness of \eqref{W-G-N} on the relevant time scale $O(\frac{1}{\max\{\ve,\beta\}})$ in the energy space.
 
    \begin{thm}[Well-posedness]	\label{W-P W-G-N} Let $s> \frac{3}{2}$ and $(\mu, \ve, \beta) \in \mathcal{A}_{\mathrm{SW}}$. Assume that $(\zeta_0,v_0) \in Y^s_{\mu}(\mathbb{R})$ satisfies the non-cavitation condition \eqref{NonCav} and $ b \in H^{s+2}(\R)$. Then there exists $T = c \big{(}|(\zeta_0, v_0)|_{Y^s_{\mu}}\big{)}^{-1}$ such that \eqref{W-G-N} admits a unique solution
		$$(\zeta, v) \in C([0,\frac{T}{\max\{\ve, \beta\}}] : Y^s_{\mu}(\mathbb{R})) \cap C^1([0, \frac{T}{\max\{\ve, \beta\}}] : Y^{s-1}_{\mu}(\mathbb{R})),$$
		that satisfies
		\begin{equation}\label{bound on solution thm}
			\sup\limits_{t \in [0,\frac{T}{\max\{\ve, \beta\}} ]} |(\zeta,v)|_{Y^s_{\mu}} \lesssim | (\zeta_0,v_0) |_{Y^s_{\mu}}.
		\end{equation}

		Furthermore, there exists a neighborhood  of $(\zeta_0, v_0)$ such that the flow map
		%
		%
		\begin{equation*}
			 : Y_{\mu}^s(\R) \: \rightarrow \: C([0,\tfrac{T}{2\max\{\ve, \beta\}}]; Y^s_{\mu}(\R)), \quad 
			(\zeta_0,v_0) \: \: \mapsto \: \: (\zeta, v),
		\end{equation*}
		is continuous.

    \end{thm}

    \begin{remark}
        For the sake of simplicity, we restrict our study to the one-dimensional setting. We comment on the possible extension to two dimensions at the end of Section 3. 
    \end{remark}

    For the next Theorem, we will state the full justification of \eqref{W-G-N} as a water waves model. To give the result, we first state the water waves equations:  
    \begin{align}\label{WW}
        \begin{cases}
            \partial_t \zeta - \frac{1}{\mu}\mathcal{G}^{\mu}[\varepsilon\zeta]\psi = 0 \\
            \partial_t \psi + \zeta + \frac{\varepsilon}{2}(\partial_x\psi)^2-\frac{\mu\varepsilon}{2}\frac{(\frac{1}{\mu}\mathcal{G}^{\mu}[\varepsilon\zeta]\psi + \varepsilon\partial_x\zeta\cdot\partial_x\psi)^2}{1+\varepsilon^2\mu(\partial_x\zeta)^2} = 0,
        \end{cases}
    \end{align}
    where $\mathcal{G}^{\mu}[\varepsilon\zeta]$ stands for the Dirichlet-Neumann operator and $\psi$ is the trace at the surface of the velocity potential $\Phi$, see \cite{WWP} for more information. 
    To compare solutions between the water waves equations and system \eqref{W-G-N}, we define the vertical average of the horizontal component of the velocity field through the formula
    \begin{align}\label{VbarDef}
        \overline{V} = \frac{1}{h}\int_{-1+\beta b}^{\varepsilon\zeta} \partial_x \Phi \: \mathrm{d}z, 
    \end{align}
    where $\Phi$ stands for the velocity potential in the water domain $\Omega_t := \{ (x,z)\in \mathbb{R}^2, -1+\beta b \leq z \leq \varepsilon\zeta \}$. It is the solution of the following elliptic problem
    \begin{align}\label{EllipticProblem}
        \begin{cases} 
            \partial_z^2\Phi + \mu \partial_x^2 \Phi = 0, \ \ \mathrm{in} \ \ \Omega_t \\
            \Phi|_{z=\varepsilon\zeta} = \psi, \ \ \partial_n \Phi|_{z=-1+\beta b} = 0,
        \end{cases}
    \end{align}
    where $\partial_n \Phi|_{z=-1+\beta b} = \partial_z \Phi - \mu \beta \partial_x b \partial_x \Phi$. We may now state the final result of this paper. 

    \begin{thm}[Full justification] \label{thm Justification} Let $s \in \N $ such that $s\geq 4$ and $(\mu, \ve, \beta) \in \mathcal{A}_{\mathrm{SW}}$. Then for $b\in H^{s+2}(\R)$ and any initial data $(\zeta_0, \psi_0) \in  H^{s}(\R)\times \dot{H}^{s}(\R)$ satisfying the non-cavitation assumption \eqref{NonCav},  there exist a unique classical solution of the water waves equations \eqref{WW} given by 
    $$\mathbf{U} = (\zeta,  \psi) \in C([0,\frac{\tilde{T}}{\max\{\ve, \beta\}}] : H^{s}(\R)\times \dot{H}^{s}(\R)),$$ 
    from which we define $\overline{V} \in C([0,\frac{\tilde{T}}{\max\{\ve, \beta\}}] : H^{s}(\R))$ through \eqref{EllipticProblem} and \eqref{VbarDef}.
    Moreover, if we let $v_0^{\text{\tiny   WGN} } = \overline{V}|_{t=0}$. Then $v_0^{\text{\tiny   WGN} } \in  X^s_{\mu}(\R)$ and there exist a unique classical solution, denoted by
    $$\mathbf{U}^{\text{\tiny WGN} } = (\zeta^{\text{\tiny WGN} }, v^{\text{\tiny   WGN} }) \in C([0,\frac{T}{\max\{\ve, \beta\}}] : Y^s_{\mu}(\mathbb{R})),$$
    of the Whitham-Green-Naghdi system \eqref{W-G-N}  sharing the same initial data
    $$\mathbf{U}^{\text{\tiny   WGN} }|_{t=0} =  (\zeta , \overline{V})|_{t=0}.$$
    Comparing the two solutions, we have that for $s\in \N$ large enough such that  for all  $0 \leq \max\{\ve, \beta\}t \leq \min\{\tilde{T},T\}$ there holds
    \begin{equation*}
        |\mathbf{U}  - \mathbf{U}^{\text{\tiny   WGN} }|_{L^{\infty}([0,t] \times \R) } \leq C(|\zeta_0|_{H^s},|\overline{V}|_{t=0}|_{H^s}, |b|_{H^{{s+2} }}) \mu^2(\ve + \beta)t,
    \end{equation*}
    with $\tilde{T}, T, C$ positive constants uniform with respect to $(\mu, \ve, \beta)\in \mathcal{A}_{SW}$.  
        
    \end{thm}

    \begin{remark}
        In the statement of the theorem, we simply let $s$ be large enough. The reason is due to the consistency result given by Theorem $10.5$ in \cite{DucheneMMWW21}, which links the water waves equations with a similar Whitham-Green-Naghdi system. However, it is possible to have a precise range of $s$ if one reproves this theorem and carefully tracks the \lq\lq loss of derivatives\rq\rq. See Section \ref{Consistency} for more on this point. 
    \end{remark}

    \subsection{Outline}

    In Section \ref{Preliminary results}, we state the technical estimates that will be used throughout the paper. In Subsection \ref{Classical estimates}, we state some classical estimates. In Subsection \ref{PropF} we study the properties of the Fourier multiplier $\mathrm{F}^{\frac{1}{2}}$. Lastly, in Subsection \eqref{PropT} we establish the properties related to the operator $\mathcal{T}[h,\beta b]$ defined by \eqref{opT}.

    In Section \ref{Consistency} we prove the consistency of the water waves equations with system \eqref{W-G-N} at the order of precision $O(\mu^2(\varepsilon+\beta))$ in the shallow water regime $\mathcal{A}_{\mathrm{SW}}$. The starting point of this proof is the full dispersion Green-Naghdi system derived \cite{DucheneMMWW21} where the precision with respect to the water waves equations \eqref{WW} is proved to be $O(\mu^2(\varepsilon+\beta))$.

    Sections,  \ref{AprioriEstimates} and \ref{Estimates on the diff} are about establishing the energy estimates with uniform bounds on the solutions. Then as a result of the energy estimates provided in the aforementioned sections, we are in the position to prove Theorem \ref{W-P W-G-N} in Section \ref{Proof - WP}. The proof relies on classical hyperbolic theory for quasilinear systems.

    In Section \ref{Justification}, we prove the full justification result of system \eqref{W-G-N} resulting from all previous sections.

    \section{Preliminary results}\label{Preliminary results}

    \subsection{Classical estimates}\label{Classical estimates}

    In this section, we state some classical results that will be used throughout the paper. 
    %
	First, recall the embedding results (see, for example, \cite{LinaresPonce15}).
	\begin{prop}[Sobolev embedding]
		Let $f \in \mathscr{S}(\R)$ and $s\in (0,\frac{1}{2})$. Then $H^s(\R) \hookrightarrow L^p(\R)$ with $p = \frac{2}{1-2s}$, and there holds
		\begin{equation}\label{Sobolev embedding s small}
			| f |_{L^p} \lesssim | \mathrm{D}^s f |_{L^2}.
		\end{equation}
		Moreover, In the case $s>\frac{1}{2}$, then $H^s(\R)$ is continuously embedded in $L^{\infty}(\R)$. 
		
	\end{prop}

	Next, we state the Leibniz rule for the Riesz potential.
	
	\begin{prop}[Fractional Leibniz rule \cite{KenigPonceVega93}]  Let $\sigma = \sigma_1 + \sigma_2 \in (0,1)$ with $\sigma_i\in[0,\sigma]$ and $p,p_1,p_2\in(1,\infty)$ satisfy $\frac{1}{p} = \frac{1}{p_1} + \frac{1}{p_2}$. Then, for $f,g \in \mathscr{S}(\R)$
		\begin{equation}\label{FracLeibnizRule}
			|\mathrm{D}^{\sigma}(fg) - f\mathrm{D}^{\sigma}g - g \mathrm{D}^{\sigma}f |_{L^p} \lesssim | \mathrm{D}^{\sigma_1} f |_{L^{p_1}} |\mathrm{D}^{\sigma_2}g|_{L^{p_2}}.
		\end{equation}
		Moreover, the case $\sigma_2 = 0$, $p_2 = \infty$ is also allowed. 
	\end{prop}

    \begin{cor} Let $r \in (\frac{1}{2},1)$,  $f\in H^{r}(\R)$ and $g\in H^{\frac{1}{2}}(\R)$. Then
    \begin{equation}\label{Holder + sobolev}
        |fg|_{L^2} \lesssim |\mathrm{D}^{r-\frac{1}{2}}f|_{L^2} |g|_{H^{\frac{1}{2}}}   
    \end{equation}
    and 
    \begin{equation}\label{D_half - fracLeib - Sob}
            |\mathrm{D}^{\frac{1}{2}}(fg)|_{L^2} \lesssim |f|_{H^{r}}|g|_{H^{\frac{1}{2}}}.
        \end{equation}
    \end{cor}
    \begin{proof}
        To prove \eqref{Holder + sobolev}, we first let $\nu  \in (0,\frac{1}{2})$ to be fixed later.  Then combine Hölder's inequality with the conjugate pair $\frac{1}{p_1} + \frac{1}{p_2} = \nu + \frac{1-2\nu}{2}=\frac{1}{2}$ and \eqref{Sobolev embedding s small} to get that
        \begin{align*}
            |fg|_{L^2} 
             \lesssim
            |f|_{L^{\frac{1}{\nu}}} |g|_{L^{\frac{2}{1-2\nu}} } 
             \lesssim 
            |\mathrm{D}^{\frac{1-2\nu}{2}}f|_{L^2}|\mathrm{D}^{\nu}g|_{L^2}.
        \end{align*}
        However, for any $r \in (\frac{1}{2},1)$ we observe that we may choose $\nu$ such that $\frac{1-2\nu}{2} = r-\frac{1}{2}$, and the proof follows.

        Next, we prove \eqref{D_half - fracLeib - Sob}.  We will use Hölder's inequality, \eqref{FracLeibnizRule} with $(\sigma_1,\sigma_2) = (\frac{1}{2},0)$, and $\frac{1}{2} = \nu + \frac{1- 2\nu}{2}$ with $\nu \in (0,\frac{1}{2})$ as above to deduce
        \begin{align*}
             |\mathrm{D}^{\frac{1}{2}} (f g) |_{L^2} 
                &
                \leq
                |\mathrm{D}^{\frac{1}{2}} (f g) - f\mathrm{D}^{\frac{1}{2}}g- g\mathrm{D}^{\frac{1}{2}}f |_{L^2} 
                +
                |f\mathrm{D}^{\frac{1}{2}}g|_{L^2} 
                +
                |g\mathrm{D}^{\frac{1}{2}}f  |_{L^2}
                \\
                &
                \lesssim
               |\mathrm{D}^{\frac{1}{2}} f |_{L^{\frac{1}{\nu}}} | g|_{L^{\frac{2}{1-2\nu}}}
                +
                |f|_{L^{\infty}} |\mathrm{D}^{\frac{1}{2}}g |_{L^2}
                \\
                & \lesssim 
                |f|_{H^r} |g|_{H^{\frac{1}{2}}},
        \end{align*}
        where we used \eqref{Sobolev embedding s small} in the last line with $\frac{1-2\nu}{2} = r -\frac{1}{2}$, and the Sobolev embedding $H^{r}(\R) \hookrightarrow L^{\infty}(\R) $.  
        
    \end{proof}

    \begin{Def}\label{definition order of a Fourier multiplier}
    Let $d=1,2$ We say that a Fourier multiplier $\mathrm{G}(\mathrm{D})$ is of order $s$ $(s\in\mathbb{R})$ and write $\mathrm{G} \in \mathcal{S}^s$ if $\xi \in \mathbb{R}^d \mapsto G(\xi) \in \mathbb{C}$ is smooth and satisfies
    \begin{align*}
        \forall \xi \in \mathbb{R}^d, \forall\beta\in\mathbb{N}^d, \ \ \sup_{\xi\in\mathbb{R}^d} \langle\xi\rangle^{|\beta|-s}|\partial^{\beta}G(\xi)| < \infty.
    \end{align*}
    We also introduce the seminorm
    \begin{align*}
        \mathcal{N}^s(\mathrm{G}) = \sup_{\beta\in\mathbb{N}^d,|\beta|\leq2+d+\lceil\frac{d}{2}\rceil} \sup_{\xi\in\mathbb{R}^d} \langle\xi\rangle^{|\beta|-s}|\partial^{\beta}G(\xi)|.
    \end{align*}
    \end{Def}

    \begin{prop}
        Let $d = 1,2$, $t_0 > d/2$, $s \geq 0$ and $\mathrm{G} \in \mathcal{S}^s$. If $f\in H^s\cap H^{t_0+1}(\mathbb{R}^d)$ then, for all $g \in H^{s-1}(\mathbb{R}^d)$,
        \begin{align}\label{Commutator estimates}
            |[\mathrm{G},f]g|_2 \leq \mathcal{N}^s(\mathrm{G})|f|_{H^{\max\{t_0+1,s\}}}|g|_{H^{s-1}}. 
        \end{align}
    \end{prop}
    \begin{proof}
        See Appendix B.2 in \cite{WWP} for the proof of this proposition.
    \end{proof}
	Next, we will need the following results to run  the Bona-Smith argument (provided in the classical paper \cite{BonaSmith75}) on the multiplier $\chi_{\delta}(\mathrm{D})$ defined by:
\begin{Def}\label{regularisation}
	Let $\chi \in \mathscr{S}(\mathbb{R})$ be a real valued function such that $\chi(0) = 1$ and $\int_{\mathbb{R}} \chi (\xi) \mathrm{d} \xi = 1$. Then for $\delta>0$ we define the
	regularisation operators $\mathrm{\chi}_{\delta}(\mathrm{D})$ in frequency by
	\begin{align*}
		\forall f \in L^2(\mathbb{R}), \quad \forall \xi \in \mathbb{R}, \quad \widehat{\mathrm{\chi}_{\delta} f}(\xi) : = \chi(\delta \xi) \hat{f}(\xi).
	\end{align*}
\end{Def}
\noindent
We give the version of the regularisation estimates as presented in \cite{LinaresPonce15} (Proposition $9.1$).
\begin{prop}\label{Rate of Decay in norm} Let $s>0$, $\delta>0$ and $f \in \mathscr{S}(\mathbb{R})$. Then
	%
	%
	%
	\begin{equation}\label{reg 3}
		|\chi_{\delta} (\mathrm{D}) f|_{H^{s+\alpha}} \lesssim \delta^{-\alpha}| f|_{H^s}, \quad \forall \alpha \geq 0,
	\end{equation}
	and
	\begin{equation}\label{reg 2}
		|\chi_{\delta} (\mathrm{D}) f - f|_{H^{s-\beta}} \lesssim \delta^{\beta} |f|_{H^s}, \quad \forall \beta \in [0,s].
	\end{equation}
	Moreover, there holds
	%
	%
	%
	\begin{equation}\label{reg 4}
		|\chi_{\delta}(\mathrm{D})  f - f|_{H^{s-\beta}} \underset{\delta \rightarrow 0}{=}o(\delta^{\beta} ), \: \: \: \quad \forall \beta \in [0,s].
	\end{equation}
\end{prop}
    
	Lastly, we need an interpolation inequality. In particular, for any $s,r,t \in \R$ such that $r < t$ and $s\in(r,t)$, then for $\theta \in (0,1)$ given by $ \theta + (1-\theta) = \frac{t-s}{t-r} + \frac{s-r}{t-r} $, we have by Plancherel's identity and Hölder's inequality that
    \begin{align}\label{Interpolation}
        |f|_{H^s}
        \leq 
         |f|_{H^{r}}^{\theta} |f|_{H^{t}}^{(1-\theta)}.
    \end{align} \\

    \subsection{Properties of $\mathrm{F}$}\label{PropF} In this section, we prove estimates concerning the dispersive properties of the equation. 

        \begin{prop} Let $s\in \R$ and $f \in \mathscr{S}(\R)$, then there exist $c>0$ such that
            \begin{equation}\label{Equivalence of norms}
               c^{-1}|f|_{X^s_{\mu}}^2 
 \leq |f|_{H^s}^2 + \mu |\mathrm{F}^{\frac{1}{2}}\partial_x f|_{H^s}^2 \leq c|f|_{X^s_{\mu}}^2,
            \end{equation}
        \begin{equation}\label{F -1/2 }
            |\mathrm{F}^{-\frac{1}{2}}f|_{H^s}\leq c |f|_{X_{\mu}^s},
        \end{equation}

        \begin{equation}\label{Kill a derivative}
            \sqrt{\mu} |\mathrm{F}^1f|_{H^s} \leq c |f|_{H^{s-1}},
        \end{equation}

		 \begin{equation}\label{Kill half derivative}
			\mu^{\frac{1}{4}} |\mathrm{F}^{\frac{1}{2}}f|_{H^s} \leq c |f|_{H^{s-\frac{1}{2}}}.
		\end{equation}
        \end{prop}

        \begin{proof}
            The behaviour at low frequency of the three Fourier multipliers $\mathrm{F}^{\frac{1}{2}}, \mathrm{F}^{-\frac{1}{2}}$ and $\mathrm{F}^1$ at low frequency is
            \begin{align*}
                F^{\frac{1}{2}}(\xi), F^{-\frac{1}{2}}(\xi), F^1(\xi) \underset{0}{\sim} 1. 
            \end{align*}
            At high frequency, their respective behavior is
            \begin{align*}
                F^{\frac{1}{2}}(\xi) \underset{\infty}{\sim} \dfrac{1}{\mu^{\frac{1}{4}}\sqrt{|\xi|}}, \ \ F^{-\frac{1}{2}}(\xi) \underset{\infty}{\sim} \mu^{\frac{1}{4}}\sqrt{|\xi|}, \ \ F^1(\xi) \underset{\infty}{\sim} \dfrac{1}{\sqrt{\mu}|\xi|}.
            \end{align*}
            This gives us \eqref{F -1/2 }, \eqref{Kill a derivative}, \eqref{Kill half derivative}, and the right-hand side inequality of \eqref{Equivalence of norms}. It only remains to prove the left-hand side inequality of \eqref{Equivalence of norms}:
            \begin{align*}
                | f |_{X^{s}_{\mu}}^2 = | f |_{H^s}^2 + \sqrt{\mu}|\mathrm{D}^{\frac{1}{2}}f |_{H^s}^2 
            \end{align*}
            Now,  let $\mathcal{F}(\mathbbm{1}_{\{\sqrt{\mu}|\mathrm{D}|\leq 1 \}}f) (\xi) = \mathbbm{1}_{\{\sqrt{\mu}|\xi|\leq 1\}}\hat{f}(\xi)$ where $\mathbbm{1}_{\{\sqrt{\mu}|\xi|\leq 1\}}$ is the usual indicator function supported on the frequencies $\sqrt{\mu}|\xi| \leq 1$. Then we get that
            \begin{align*}
                \sqrt{\mu}|\mathrm{D}^{\frac{1}{2}}f |_{H^s}^2 &= \sqrt{\mu}| \mathbbm{1}_{\{\sqrt{\mu}|\mathrm{D}|\leq 1 \}} \mathrm{D}^{\frac{1}{2}}\mathrm{J}^s f|_{L^2}^2 + \sqrt{\mu}| \mathbbm{1}_{\{\sqrt{\mu}|\mathrm{D}|>1 \}} \mathrm{D}^{\frac{1}{2}}\mathrm{J}^s f|_{L^2}^2\\
                &
                \lesssim 
                \sqrt{\mu}| \mathbbm{1}_{\{\sqrt{\mu}|\mathrm{D}|\leq 1 \}} \mathrm{F}^{\frac{1}{2}}\mathrm{D}^{\frac{1}{2}}\mathrm{J}^s f|_{L^2}^2 
                +
                \mu^{\frac{3}{2}} | \mathbbm{1}_{\{\sqrt{\mu}|\mathrm{D}|>1 \}} \mathrm{F}^1 \mathrm{D}^{\frac{1}{2}}\mathrm{J}^s f|_{L^2}^2\\
                &\lesssim | f|_{H^s}^2 + \mu | \mathrm{F}^{\frac{1}{2}} f |_{H^s}^2.
            \end{align*}
        \end{proof}

        \begin{prop}  Let $f,g \in \mathscr{S}(\R)$ and $t_0>\frac{1}{2}$. Then for $s\geq \frac{1}{2} $ there holds,
        \begin{equation}\label{Commutator: J_sF_half}
            |[\mathrm{J}^s\mathrm{F}^{\frac{1}{2}}, f]g|_{L^2} \lesssim |f|_{H^{\max\{t_0+1,s-\frac{1}{2}\}}} |\mathrm{F}^{\frac{1}{2}}g|_{H^{s-1}}.
        \end{equation}
    	In the case $s\geq 1$, there holds
    	\begin{equation}\label{Commutator: J Fdx}
    		|[\mathrm{J}^s, \mathrm{F}^{\frac{1}{2}}\partial_x (f \cdot)] \partial_x g |_{L^2} \lesssim |\mathrm{F}^{\frac{1}{2}} \partial_x f |_{H^s} |\partial_xg|_{H^{t_0}} 
            +
            |\mathrm{F}^{\frac{1}{2}} \partial_xg |_{H^s} |\partial_xf|_{H^{t_0}}.
    	\end{equation}
        Moreover, in the case $s=0$ we have that
        %
        %
        %
    	%
    	%
    	%
        \begin{equation}\label{Commutator: F_half}
            |[\mathrm{F}^{\frac{1}{2}},f] g |_{L^2} \lesssim |f|_{X^{t_0+1}_{\mu}} |\mathrm{F}^{\frac{1}{2}} g|_{H^{-1}}.
        \end{equation}

\
        \end{prop}

        \begin{proof} 
            To prove \eqref{Commutator: J_sF_half} we note that the Fourier multiplier $\mathrm{J}^s \mathrm{F}^{\frac{1}{2}}$ is of order $s-\frac{1}{2}$ in the sense of Definition \ref{definition order of a Fourier multiplier}. Moreover,  we observe that
            \begin{align*}
                \mathcal{N}^{s-\frac{1}{2}}(\mathrm{J}^s \mathrm{F}^{\frac{1}{2}}) \lesssim \mu^{-\frac{1}{4}}.
            \end{align*}
            Thanks to the commutator estimates of Proposition \ref{Commutator estimates} and estimate \eqref{Equivalence of norms}, we have
            \begin{align*}
                |[\mathrm{J}^s\mathrm{F}^{\frac{1}{2}}, f]g|_{L^2} \lesssim \mu^{-\frac{1}{4}}|f|_{H^{\max\{t_0+1,s-\frac{1}{2}\}}} | g |_{H^{s-\frac{3}{2}}} \lesssim |f|_{H^{\max\{t_0+1,s-\frac{1}{2}\}}} |\mathrm{F}^{\frac{1}{2}}g|_{H^{s-1}},
            \end{align*}
            and proves estimate \eqref{Commutator: J_sF_half}.\\

            For the proof of \eqref{Commutator: J Fdx}, we start by estimating the bilinear form: 
            $$\mathrm{a}(\mathrm{D})(f,g) = [\mathrm{J}^s, \mathrm{F}^{\frac{1}{2}}\partial_x (f \cdot)] \partial_x g,$$ 
            given by
            \begin{align*}
            	|\hat{a}(\xi)(f,g)|  \leq  \int_{\R} F^{\frac{1}{2}}(\xi) |\xi|\big{|} \langle \xi \rangle^s - \langle \rho \rangle^s \big{|} |\hat{f}(\xi-\rho)| \: |\widehat{\partial_x g}(\rho)| \: d\rho.
            \end{align*}
            First, if $|\xi| \leq |\rho|$ we can use the mean value theorem to deduce that
            \begin{align*}
            	|\hat{a}(\xi)(f,g)|  \lesssim  \int_{\R} F^{\frac{1}{2}}(\rho) |\rho|  \langle \rho \rangle^{s-1} |\xi - \rho| |\hat{f}(\xi-\rho)| \: |\widehat{\partial_x g}(\rho)| \: d\rho,
            \end{align*}
            since  both $\omega \mapsto \langle \omega \rangle^{s-1}$ and $\omega \mapsto F^{\frac{1}{2}}(\omega) |\omega|$ are increasing functions. By extension, we make a change of variable $\gamma = \xi - \rho$,  apply Minkowski integral inequality and Cauchy-Schwarz to find the estimate
            \begin{align*}
            	|[\mathrm{J}^s, \mathrm{F}^{\frac{1}{2}}\partial_x (f \cdot)] \partial_x g |_{L^2}
            	&
            	\lesssim 
            	\Big{|}
            	\int_{\R} F^{\frac{1}{2}}(\cdot - \gamma ) |\cdot - \gamma|  \langle \cdot - \gamma \rangle^{s-1} |\gamma| |\hat{f}(\gamma)| \: |\widehat{\partial_x g}(\cdot - \gamma)| \: d\gamma
            	\Big{|}_{L^2_{\xi}}
            	\\
            	& 
            	\lesssim 
            	| \mathrm{F}^{\frac{1}{2}}  \partial_x g
            	|_{H^s}
            	\int_{\R}  |\gamma| |\hat{f}(\gamma)|  \: d\gamma
            	\\
            	& 
            	\lesssim 
            	|\mathrm{F}^{\frac{1}{2}} \partial_x g |_{H^s} |\partial_xf|_{H^{t_0}}.
            \end{align*}
            On the other hand, when $|\rho| \leq |\xi|$ then we can argue similarly to find that
            \begin{align*}
            	|\hat{a}(\xi)(f,g)|  
            	& \lesssim  
            	\int_{\R} F^{\frac{1}{2}}(\xi - \rho) |\xi - \rho |  \langle \xi - \rho \rangle^{s-1} |\xi - \rho| |\hat{f}(\xi-\rho)| \: |\widehat{\partial_x g}(\rho)| \: d\rho
            	\\
            	& 
            	\hspace{0.5cm}
            	+
            	\int_{\R} F^{\frac{1}{2}}(\rho) |\rho|  \langle \rho \rangle^{s-1} |\xi - \rho| |\hat{f}(\xi-\rho)| \: |\widehat{\partial_x g}(\rho)| \: d\rho,
            \end{align*}
            and as using the estimate above we conclude in this case that
            \begin{align*}
            	|[\mathrm{J}^s, \mathrm{F}^{\frac{1}{2}}\partial_x (f \cdot)] \partial_x g |_{L^2}
            	&
            	\lesssim 
            	\Big{|}
            	\int_{\R} F^{\frac{1}{2}}(\cdot - \rho ) |\cdot - \rho|  \langle \cdot - \rho \rangle^{s-1} |\cdot - \rho| |\hat{f}(\cdot - \rho)| \: |\widehat{\partial_x g}(\rho)| \: d\rho
            	\Big{|}_{L^2_{\xi}}
            	\\
            	& 
            	\hspace{0.5cm}
            	+
            	|\mathrm{F}^{\frac{1}{2}} \partial_x g |_{H^s} |\partial_xf|_{H^{t_0}}
            	\\
            	& 
            	\lesssim 
            	|\mathrm{F}^{\frac{1}{2}} \partial_x f |_{H^s} |\partial_xg|_{H^{t_0}}.
            	+
            	|\mathrm{F}^{\frac{1}{2}} \partial_x g |_{H^s} |\partial_xf|_{H^{t_0}}.
            \end{align*}
            Adding the two cases completes the proof.\\


            Next, we prove \eqref{Commutator: F_half} by estimating the bilinear form:
            \begin{align*}
                |\hat{a}(\xi)(f,g)| \lesssim  \int_{\R} \big{|}F^{\frac{1}{2}}(\xi) - F^{\frac{1}{2}}(\rho) \big{|} |\hat{f}(\xi-\rho) |  | \hat{g}(\rho) | \: d\rho.
            \end{align*}
            Clearly, it is enough to prove that
            \begin{align}\label{Claim 1}
                k(\xi, \rho) : = \big{|}F^{\frac{1}{2}}(\xi) - F^{\frac{1}{2}}(\rho) \big{|} F^{-\frac{1}{2}}(\rho) \langle \rho \rangle \lesssim 1 + |\xi-\rho|F^{-\frac{1}{2}}(\xi-\rho).
            \end{align}
            Indeed, assuming the claim \eqref{Claim 1} and using Plancherel, Minkowski integral inequality, the Cauchy-Schwarz inequality and \eqref{F -1/2 } we obtain the desired estimate
            \begin{align*}
                | [\mathrm{F}^{\frac{1}{2}},f]g|_{L^2} 
                & \leq
                \Big{|} 
                \int_{\R}
                k(\xi,\rho) \: |\hat{f}(\xi-\rho)| \: F^{\frac{1}{2}}(\rho)\langle \rho \rangle^{-1}|\hat{g}(\rho)| \: d\rho
                \Big{|}_{L^2_{\xi}}
                \\ 
                & \lesssim
                \Big{|} 
                \int_{\R}\big{(}1 + |\gamma|F^{-\frac{1}{2}}(\gamma)\big{)}\: |\hat{f}(\gamma)| \: F^{\frac{1}{2}}(\xi-\gamma)\langle \xi-\gamma \rangle^{-1}|\hat{g}(\xi-\gamma)| \: d\gamma
                \Big{|}_{L^2_{\xi}}
                \\
                & 
                \lesssim |f|_{X^{t_0+1}_{\mu}}|\mathrm{F}^{\frac{1}{2}}g|_{H^{-1}}.
            \end{align*}
            Now, to prove the claim \eqref{Claim 1}, we consider three cases. First, in the case $|\rho|\leq 1$ it follows directly that
            \begin{align*}
                k(\xi, \rho) \lesssim 1, 
            \end{align*}
            since $\xi \mapsto F^{\frac{1}{2}}(\xi)$ and $\rho \mapsto F^{-\frac{1}{2}}(\rho) \langle \rho \rangle$ is bounded. Next, consider the case $|\rho|>1$ and $|\xi| \leq |\rho|$. Then we note that since $\xi \mapsto F^{\frac{1}{2}}(\xi)$ is decreasing for $\xi>0$, we have the estimate
            \begin{align*}
                \frac{F(\rho)}{F(\xi)} \leq 1,
            \end{align*}
            and moreover since  $\xi \mapsto \big{(}\frac{\sqrt{\mu}\xi}{\tanh(\sqrt{\mu}\xi)}-1\big{)}$ is increasing for $\xi>0$, we get  that 
            \begin{align}\label{observation 1}
                \frac{|\xi |}{ |\rho |}
                \leq 
                \bigg{(}\frac{|\xi|}{|\rho|}\bigg{)}^{\frac{1}{2}}
                \leq 
                \bigg{(}\frac{F(\rho)}{F(\xi)} \bigg{)}^{\frac{1}{2}} 
                \leq 1.
            \end{align}
            Thus, we obtain the bound 
            \begin{align*}
                k(\xi, \rho) 
                & =
                \bigg{(} 
                1 -\bigg{(}\frac{F(\rho)}{F(\xi)} \bigg{)}^{\frac{1}{2}}
                \bigg{)}
                \bigg{(}\frac{F(\xi)}{F(\rho)} \bigg{)}^{\frac{1}{2}}
                \langle \rho \rangle 
                \\
                & 
                \lesssim (|\rho| - |\xi|)
                \bigg{(}\frac{F(\xi)}{F(\rho)} \bigg{)}^{\frac{1}{2}}.
            \end{align*}
            Finally, to conclude this case, we make the observation that if $|\xi| \sim  |\rho|$, then 
            \begin{equation*}
                \bigg{(}\frac{F(\xi)}{F(\rho)} \bigg{)}^{\frac{1}{2}} \lesssim 1.
            \end{equation*}
            Otherwise, we obtain
            \begin{align*}
                \bigg{(}\frac{F(\xi)}{F(\rho)} \bigg{)}^{\frac{1}{2}}
                \lesssim 1+\mu^{\frac{1}{4}} |\rho|^{\frac{1}{2}} \lesssim 1+\mu^{\frac{1}{4}} |\rho-\xi|^{\frac{1}{2}}.
            \end{align*}
            Gathering these estimates allows us to conclude that
         \begin{align*}
                k(\xi, \rho) 
                \lesssim (|\rho| - |\xi|)
                \bigg{(}\frac{F(\xi)}{F(\rho)} \bigg{)}^{\frac{1}{2}} \lesssim 1 + |\xi-\rho| F^{-\frac{1}{2}}(\xi-\rho).
            \end{align*}
            On the other hand,  the case $|\xi| > |\rho| >1$ follows directly by changing the role of $\xi$ and $\rho$ in \eqref{observation 1}. Indeed, we obtain that
            \begin{align*}
                k(\xi, \rho) 
                & =
                \bigg{(} 
                1 -\bigg{(}\frac{F(\xi)}{F(\rho)} \bigg{)}^{\frac{1}{2}}
                \bigg{)}
                \langle \rho \rangle 
                \\
                & 
                \lesssim (|\xi| - |\rho|),
            \end{align*}
            and the proof of \eqref{Commutator: F_half} is complete.

        \end{proof}

        \begin{prop}\label{estimate F1/2}
            Let $s\geq 0$, and let $f \in H^{s+2}(\mathbb{R})$, then we have the following estimation on the Fourier multiplier $\mathrm{F}^{\frac{1}{2}}$
            \begin{align*}
                |(\mathrm{F}^{\frac{1}{2}} -1)f|_{H^s} \lesssim \mu |f|_{H^{s+2}}.
            \end{align*}
        \end{prop}
        \begin{proof}
            First, remark that it is enough to prove the result only when $s=0$. The function defining the Fourier multiplier $\mathrm{F}^{\frac{1}{2}}$ is a smooth function on $(0,+\infty)$, continuous in $0$ with $F^{\frac{1}{2}}(0) = 1$ and its first derivative is zero. Moreover, its second derivative is bounded in $[0,+\infty)$, so that from Plancherel identity and the Taylor-Lagrange formula, we get
            \begin{align*}
                |(F^{\frac{1}{2}}(\sqrt{\mu}|\xi|)-1)\hat{f}|_2 \leq \mu| |\xi|^2\hat{f}|_{L^2}.
            \end{align*}
            In the end, we have the estimate
            \begin{align*}
                |(\mathrm{F}^{\frac{1}{2}}-1)f|_{L^2} \leq \mu|f|_{H^2}.
            \end{align*}
        \end{proof}

    \subsection{Properties of $\mathscr{T}[h,\beta b]$}\label{PropT} 
    In this section, we study an elliptic operator associated with $\mathcal{T}[h,\beta b]$ given by \eqref{opT}. The main result is given in the following proposition where the main reference is \cite{Samer11}.
    
        \begin{prop}\label{Inverse of T}
            Let $(\mu, \ve, \beta)\in \mathcal{A}_{\mathrm{SW}}$, $s\geq 0$, $\zeta \in H^{\max\{1,s\}}(\R)$, $b \in H^{s+2}(\R)$ and let $h = 1 +\ve\zeta - \beta b$ satisfy the non-cavitation condition \eqref{NonCav}. Define the application
            \begin{equation}\label{T[h]}
			\mathscr{T}[h,\beta b] 
			\:
			:
			\:
			\begin{cases}
				H^1(\R) & \rightarrow L^2(\R) \\
				v &  \mapsto hv + \mu h \mathcal{T}[h,\beta b] v
			\end{cases} 
		\end{equation}
            Then we have the following properties:
            \begin{itemize}
                \item [1.] The operator \eqref{T[h]}  is well-defined and  for $v\in H^1(\R)$  there holds,
                    \begin{equation}
                        | \mathscr{T}[h,\beta   b] v |_{L^2} \lesssim  | v |_{X^{\frac{1}{2}}_{\mu}}.
                    \end{equation}
                \item [2.] The operator \eqref{T[h]} is one-to-one and onto.\\

                \item [3.] For $s\geq 0$ and $f \in H^s(\R)$ there holds,
                    \begin{equation}\label{Inverse est T}
                        | \mathscr{T}^{-1}[h,\beta  b]f|_{X^s_{\mu}}\lesssim  |f|_{H^{s}}.
                    \end{equation}

                \item [4.] For $s \geq 1$ and $f \in H^{s-1}(\R)$ there  holds, 
                    \begin{equation}\label{Inverse est F T}
                         \sqrt{\mu} |\mathrm{F}^{\frac{1}{2}}\mathscr{T}^{-1}[h,\beta   b]f|_{H^s} \lesssim  |f|_{H^{s-1}} .
                    \end{equation}
            \end{itemize}        
        \end{prop}

        \begin{proof} We give the proof in four steps. \\

            \noindent
            \underline{Step 1:} The application \eqref{T[h]} is well-defined. Indeed, by assumption and Sobolev embedding $H^{\frac{1}{2}^+}(\R) \hookrightarrow L^{\infty}(\R)$ we have that $h \in L^{\infty}(\R)$. Therefore, by \eqref{Equivalence of norms} we get that
            \begin{align*}
                |\mathscr{T}[h,\beta  b] v |_{L^2} 
                & \lesssim
                |hv|_{L^2} + \mu |\partial_x \mathrm{F}^{\frac{1}{2}}(h^3\mathrm{F}^{\frac{1}{2}}\partial_x v)|_{L^2} 
                +
                \mu |\partial_x \mathrm{F}^{\frac{1}{2}} (h^2 (\beta \partial_x b)v)|_{L^2}
                \\ 
                & 
                \hspace{0.5cm}
                +
                \mu |h^2(\beta \partial_x b) \mathrm{F}^{\frac{1}{2}} \partial_x v|_{L^2}
                +
                \mu|h(\beta \partial_x b)^2 v|_{L^2}
                \\
                & 
                \lesssim 
                |h|_{L^{\infty}}|v|_{L^2} 
                +
                \mu^{\frac{3}{4}}|\mathrm{D}^{\frac{1}{2}}(h^3 \mathrm{F}^{\frac{1}{2}} \partial_x v)|_{L^2}
                +
                \mu^{\frac{3}{4}} | \mathrm{D}^{\frac{1}{2}} (h^2 (\beta \partial_x b)v)|_{L^2}
                \\ 
                & 
                \hspace{0.5cm}
                +
                \mu |h^2(\beta \partial_x b) \mathrm{F}^{\frac{1}{2}} \partial_x v|_{L^2}
                +
                \mu|h|_{L^{\infty}}|(\beta \partial_x b)^2 v|_{L^2}
                \\
                & = :
                A_1 + A_2 + A_3 + A_4 + A_5.
            \end{align*}
            To conclude, we note that $(h-1)\in H^1(\R)$ and together with Hölder's inequality, the Sobolev embedding, and \eqref{Equivalence of norms} we estimate $A_1 + A_4 + A_5$:
            \begin{align*}
                A_1 + A_4 + A_5 \leq c(|h-1|_{H^1}, |h^2-1|_{H^1}, \beta |\partial_x b|_{L^{\infty}}) |v|_{X^{\frac{1}{2}}_{\mu}}.
            \end{align*}
            The remaining terms are treated similarly, after an application of \eqref{D_half - fracLeib - Sob}, and yield the desired estimate
            \begin{align*}
                |\mathscr{T}[h,\beta  b] v |_{L^2} \lesssim |v |_{X^{\frac{1}{2}}_{\mu}}.
            \end{align*}

            \noindent
            \underline{Step 2.} The application \eqref{T[h]} is one-to-one and onto. Equivalently, we prove that there exist a unique solution $v \in H^1(\R)$ to the equation
	\begin{equation}\label{Tu=f, strong form}
		\mathscr{T}[h,\beta  b] v = f,
	\end{equation}
	for $f\in L^2(\R)$. To construct a solution, we first consider the variational formulation of \eqref{Tu=f, strong form} that is given by
	\begin{equation}\label{variational formulation}
		a(v,\varphi ) = L(\varphi),
	\end{equation}
	for any $\varphi \in C^{\infty}_{c}(\R)$ and with 
        \begin{equation*}
        \begin{cases}
            a(v,\varphi) := \big{(} v,h \varphi \big{)}_{L^2} 
            +
            \big{(}  v, \mu h\mathcal{T}[h,\beta  b]\varphi \big{)}_{L^2}
            \\
            L(\varphi ) :=\big{(} f,\varphi \big{)}_{L^2}.
        \end{cases}
        \end{equation*}
        Then, through a direct application of the Lax-Milgram lemma, we prove there exists a unique variational solution $v \in  \overline{C_c^{\infty}(\R)}^{|\cdot |_{H^{ \frac{1}{2} }}}= H^{\frac{1}{2}}(\R)$. Indeed, we observe that the application $(u,v) \mapsto a(u,v)$ is continuous on $H^{\frac{1}{2}}(\R) \times H^{\frac{1}{2}}(\R)$:
        \begin{align*}
            |a(v,\varphi)| \leq c(|h-1|_{H^1},\beta|\partial_x b|_{L^{\infty}}) |v|_{X^{0}_{\mu}}|\varphi|_{X^{0}_\mu},
        \end{align*}
        by integration by parts, Hölder's inequality and \eqref{Equivalence of norms}. Moreover, the coercivity estimate is deduced by first making the observation:
        \begin{align*}
            a(v,v) 
            & =
            \big{(} v, h v \big{)}_{L^2}
            +
            \mu 
            \Big{(}h
            \big{(}\frac{h}{\sqrt{3}}\mathrm{F}^{\frac{1}{2}}\partial_xv - \frac{\sqrt{3}}{2}(\beta \partial_x b)v\big{)},\frac{h}{\sqrt{3}}\mathrm{F}^{ \frac{1}{2} }\partial_x v - \frac{\sqrt{3}}{2} (\beta \partial_x b) v \Big{)}_{L^2}
            \\
            & 
            \hspace{0.5cm}
            +
            \frac{\mu}{4}
            \big{(} h(\beta \partial_x b) v,(\beta \partial_x b)v \big{)}_{L^2}
            \\
            & \geq 
            h_0 |v|_{L^2}^2
            +
            \mu h_0
            \Big{|} 
            \frac{h}{\sqrt{3}}\mathrm{F}^{\frac{1}{2}}\partial_xv - \frac{\sqrt{3}}{2}(\beta \partial_x b)v
            \Big{|}_{L^2}^2
            +
            \frac{\mu\beta^2}{4}|\sqrt{h}(\partial_xb)v|_{L^2}^2
            \\
            & =: 
            I.
        \end{align*}
        Now, let $\nu>0$ be chosen later and make the decomposition  $I = (1-\nu) I + \nu I$. Then the first term can be bounded below by
        \begin{align*}
            (1-\nu) I 
            \geq 
            (1-\nu ) h_0 |v|_{L^2}^2
            +
            \frac{(1-\nu)\mu\beta^2}{4}|\sqrt{h}(\partial_xb)v|_{L^2}^2.
        \end{align*}
        On the other hand, the remaining part is estimated by Cauchy-Schwarz and  Young's inequality:
        \begin{align*}
            \nu I 
            & \geq 
            \frac{\nu \mu h_0^3}{3} |\mathrm{F}^{\frac{1}{2}} \partial_x v |_{L^2}^2 - \mu \nu h_0 \beta |\sqrt{h}|_{L^{\infty}} | \sqrt{h} (\partial_x b) v |_{L^2} |\mathrm{F}^{\frac{1}{2}}\partial_x v|_{L^2}
            \\
            & \geq  
            \frac{\nu \mu h_0^3}{3} |\mathrm{F}^{\frac{1}{2}} \partial_x v |_{L^2}^2 - \mu \nu h_0 \Big(\dfrac{h_0^2}{6}|\mathrm{F}^{\frac{1}{2}}\partial_x v|_{L^2}^2 + \dfrac{3}{2h_0^2}|\sqrt{h}|_{L^{\infty}}^2\beta^2 | \sqrt{h} (\partial_x b) v |_{L^2}\Big).
        \end{align*}
        So that
        \begin{align*}
            I \geq (1-\nu)h_0|v|_{L^2}^2 + \dfrac{\nu \mu h_0^3}{6} |\mathrm{F}^{\frac{1}{2}}\partial_x v|_{L^2}^2 + \mu\beta^2 \Big(\dfrac{1}{4} - \nu\Big(\dfrac{1}{4} + \dfrac{3| \sqrt{h}|_{L^{\infty}}^2}{2h_0} \Big)\Big) | \sqrt{h} (\partial_x b) v |_{L^2}.
        \end{align*}
        Thus, to conclude, simply choose $\nu$ small enough, from which we deduce the desired estimate
        \begin{align}\label{Coereicivity}
            a(v,v) 
           \geq
           c
           |v|_{X^{0}_{\mu}}^2.
        \end{align}
        Lastly, the application $\varphi\mapsto L(\varphi)$ is continuous on $\varphi \in H^{\frac{1}{2}}(\R) $ by Cauchy-Schwarz. Consequently, we have a unique variational solution $v\in H^{\frac{1}{2}}(\R)$ satisfying \eqref{variational formulation} for any $\varphi \in H^{\frac{1}{2}}(\R)$. Let us show that this solution is in $H^1(\R)$, so that it also satisfies \eqref{Tu=f, strong form}.\\

        Let $0<\delta\leq 1 $ and take $\chi_{\delta}(D)$ as in Definition  \ref{regularisation} and define a sequence of smooth functions given by $v_{\delta} := \chi_{\delta}v \in  \cap_{s >0}  H^{s}(\R)$. Then using $\mathrm{J}^1 (\chi_{\delta})^2 v\in H^{\frac{1}{2}}(\R)$ as a test function, we get
        \begin{align*}
            a(\mathrm{J}^{\frac{1}{2}}v_{\delta},\mathrm{J}^{\frac{1}{2}}v_{\delta}) 
            &=
            a\Big(v,\mathrm{J}^1(\chi_{\delta})^2 v\Big) 
            -
            \big([\mathrm{J}^{\frac{1}{2}}\chi_{\delta},h]v,\mathrm{J}^{\frac{1}{2}}v_{\delta} \big)_{L^2} 
            -
            \frac{\mu}{3} \big([\mathrm{J}^{\frac{1}{2}}\chi_{\delta},h^3]\mathrm{F}^{\frac{1}{2}}\partial_x v, \mathrm{J}^{\frac{1}{2}}\mathrm{F}^{\frac{1}{2}}\partial_x v_{\delta} \big)_{L^2} 
            \\
            & \hspace{0.5cm} 
            +
            \frac{\mu}{2}\big([\mathrm{J}^{\frac{1}{2}}\chi_{\delta},h^2(\beta\partial_x b)]\mathrm{F}^{\frac{1}{2}}\partial_x v, \mathrm{J}^{\frac{1}{2}}v_{\delta} \big)_{L^2} 
            +
            \frac{\mu}{2}\big([\mathrm{J}^{\frac{1}{2}}\chi_{\delta},h^2(\beta\partial_x b)] v, \mathrm{J}^{\frac{1}{2}}\mathrm{F}^{\frac{1}{2}}\partial_x v_{\delta} \big)_{L^2}
            \\
            & \hspace{0.5cm} 
            -
            \mu\big([\mathrm{J}^{\frac{1}{2}}\chi_{\delta},h(\beta \partial_x b)^2]v, \mathrm{J}^{\frac{1}{2}}v_{\delta} \big)_{L^2}.
        \end{align*}
        Then using \eqref{variational formulation} and \eqref{Coereicivity}, we get
        \begin{align*}
            c|v_{\delta} |_{X^{\frac{1}{2}}_{\mu}}^2 
            &\leq 
            a(\mathrm{J}^{\frac{1}{2}}v_{\delta},\mathrm{J}^{\frac{1}{2}}v_{\delta}) 
            \\
            &=
            |\big(f,\mathrm{J}^1 (\chi_{\delta})^2 v\big)_{L^2} 
            -
            \big([\mathrm{J}^{\frac{1}{2}}\chi_{\delta},h]v,\mathrm{J}^{\frac{1}{2}}v_{\delta} \big)_{L^2} 
            -
            \frac{\mu}{3} \big([\mathrm{J}^{\frac{1}{2}}\chi_{\delta},h^3]\mathrm{F}^{\frac{1}{2}}\partial_x v, \mathrm{J}^{\frac{1}{2}}\mathrm{F}^{\frac{1}{2}}\partial_x v_{\delta} \big)_{L^2} 
            \\
            & \hspace{0.5cm} 
            +
            \frac{\mu}{2}\big([\mathrm{J}^{\frac{1}{2}}\chi_{\delta},h^2(\beta\partial_x b)]\mathrm{F}^{\frac{1}{2}}\partial_x v, \mathrm{J}^{\frac{1}{2}}v_{\delta} \big)_{L^2} 
            +
            \frac{\mu}{2}\big([\mathrm{J}^{\frac{1}{2}}\chi_{\delta},h^2(\beta\partial_x b)] v, \mathrm{J}^{\frac{1}{2}}\mathrm{F}^{\frac{1}{2}}\partial_x v_{\delta} \big)_{L^2}
            \\
            & \hspace{0.5cm} 
            - \mu\big([\mathrm{J}^{\frac{1}{2}}\chi_{\delta},h(\beta \partial_x b)^2]v, \mathrm{J}^{\frac{1}{2}}v_{\delta} \big)_{L^2}|.
        \end{align*}
        Now remark that $\mathrm{J}^{\frac{1}{2}}\chi_{\delta}(D)$ is a Fourier multiplier of order $\frac{1}{2}$ in the sense of Definition \ref{definition order of a Fourier multiplier}, and that $\mathcal{N}^{\frac{1}{2}}(\mathrm{J}^{\frac{1}{2}}\chi_{\delta}(D)) \lesssim 1$ uniformly in $\delta$. Hence, from Cauchy-Schwarz inequality, \eqref{Equivalence of norms} and the commutator estimates of Proposition \ref{Commutator estimates}, we get
        \begin{align*}
            c |v_{\delta} |_{X^{\frac{1}{2}}_{\mu}}^2 \leq c(|f |_{L^2} + |v |_{L^2}) |v_{\delta}|_{H^1}.
        \end{align*}
        We, therefore, deduce the estimate
        \begin{align}\label{Base case for step 4}
            \sqrt{\mu}c |v_{\delta}|_{H^1} 
             \lesssim       |f|_{L^2} + |v|_{L^2}.
        \end{align}
        The family $\{v_{\delta}\}_{0<\delta\leq 1}$ is uniformly bounded in $H^1(\R)$. Hence, since $H^1(\R)$ is a reflexive Banach space, there exists $V \in H^1(\R)$ and a subsequence $\{v_{\delta_n}\}_{0<\delta_n\leq 1}$ with $\delta_n \to 0$ such that $v_{\delta_n} \rightharpoonup V$. By uniqueness of the limit in $L^2(\R)$, we deduce that $v = V \in H^1(\R)$.\\
        To conclude, we may now use \eqref{variational formulation} and integration by parts to find that
	\begin{equation*}
		\big{(} \mathscr{T}[h,\beta  b]v,\varphi \big{)}_{L^2} = \big{(} f, \varphi \big{)}_{L^2},
	\end{equation*}
	for any $\varphi \in C^{\infty}_{c}(\R)$. Hence, we conclude that the variational solution also provides a unique solution of \eqref{Tu=f, strong form}.\\
 
        \noindent
        \underline{Step 3.} The estimate \eqref{Inverse est T} holds. To prove the claim, we first consider $v$, the solution of \eqref{Tu=f, strong form}. From the coercivity estimate \eqref{Coereicivity} and Cauchy-Schwarz inequality, we have
        \begin{align*}
            |v|_{X^0_{\mu}}^2 \lesssim a(v,v) = L(v) \leq |f|_{L^2}|v|_{L^2}\leq |f|_{L^2}|v|_{X^0_{\mu}},
        \end{align*}
        so that
        \begin{align}\label{Base case}
            |v|_{X^0_{\mu}} \lesssim c|f|_{L^2}.
        \end{align}
        Next, we apply $\mathrm{J}^s$ to \eqref{Tu=f, strong form} and observe that $\mathrm{J}^s v$ is a distributional solution of the equation
        \begin{align}\label{differentiated equation}
            \mathscr{T}[h, \beta b]v_s &= \mathrm{J}^s f - [\mathrm{J}^s,h]v + \frac{\mu}{3} \partial_x \mathrm{F}^{\frac{1}{2}}([\mathrm{J}^s,h^3]\mathrm{F}^{\frac{1}{2}}\partial_x v) - \frac{\mu}{2}\partial_x \mathrm{F}^{\frac{1}{2}}([\mathrm{J}^s,h^2(\beta\partial_x b)]v) \\
            & \hspace{0.5cm} + \frac{\mu}{2}[\mathrm{J}^s,h^2(\beta\partial_x b)]\mathrm{F}^{\frac{1}{2}}\partial_x v - [\mathrm{J}^s,h(\beta\partial_x b)^2]v
        \end{align}
        Moreover, from the coercivity estimate \eqref{Coereicivity}, the variational solution of \eqref{differentiated equation}, $v_s$, satisfies
        \begin{align*}
            c|v_s |_{X^0_{\mu}}^2 \leq a(v_s,v_s) &= \Big(\mathrm{J}^s f - [\mathrm{J}^s,h]v + \frac{\mu}{3} \partial_x \mathrm{F}^{\frac{1}{2}}([\mathrm{J}^s,h^3]\mathrm{F}^{\frac{1}{2}}\partial_x v) - \frac{\mu}{2}\partial_x \mathrm{F}^{\frac{1}{2}}([\mathrm{J}^s,h^2(\beta\partial_x b)]v) \\
            & \hspace{0.5cm} + \frac{\mu}{2}[\mathrm{J}^s,h^2(\beta\partial_x b)]\mathrm{F}^{\frac{1}{2}}\partial_x v - [\mathrm{J}^s,h(\beta\partial_x b)^2]v, v_s \Big)_{L^2}\\
            &= \big(\mathrm{J}^s f, v_s)_{L^2} - \big([\mathrm{J}^s,h]v,v_s\big)_{L^2} - \frac{\mu}{3} \big( [\mathrm{J}^s,h^3]\mathrm{F}^{\frac{1}{2}}\partial_x v, \mathrm{F}^{\frac{1}{2}}\partial_x v_s\big)_{L^2} \\
            &\hspace{0.5cm} + \frac{\mu}{2}\big( [\mathrm{J}^s,h^2(\beta\partial_x b)]v, \mathrm{F}^{\frac{1}{2}} \partial_x v_s\big)_{L^2} + \frac{\mu}{2} \big([\mathrm{J}^s,h^2(\beta\partial_x b)]\mathrm{F}^{\frac{1}{2}}\partial_x v, v_s\big)_{L^2} \\
            &\hspace{0.5cm} - \big([\mathrm{J}^s,h(\beta\partial_x b)^2]v, v_s \big)_{L^2}.
        \end{align*}
        Then using Cauchy-Schwarz inequality and the commutator estimates of Proposition \ref{Commutator estimates}, we get
        \begin{align*}
            |v_s|^2_{X^0_{\mu}} 
            &  \lesssim c
            (|f|_{H^s} + |v|_{X^{s-1}_{\mu}} )|v_s|_{X^s_{\mu}}.
        \end{align*}
        To conclude, we first consider $s\in \mathbb{N}$ and simply argue by induction using \eqref{Base case} as a base case noting that the distributional and variational solutions must coincide, i.e. $v_s = \mathrm{J}^s v$. Then use the interpolation inequality \eqref{Interpolation} to obtain \eqref{Inverse est T} for any $s$ real number $\geq 0$.\\

        \noindent
        \underline{Step 4.} The estimate \eqref{Inverse est F T} holds. Arguing as above, we apply $\mathrm{F}^{\frac{1}{2}}\mathrm{J}^s$ to \eqref{Tu=f, strong form} and get
        \begin{align*}
            |\mathrm{F}^{\frac{1}{2}} v |_{X^s_{\mu}} &\lesssim a(\mathrm{F}^{\frac{1}{2}} \mathrm{J}^s v, \mathrm{F}^{\frac{1}{2}} \mathrm{J}^s v) \\
            &= \big(\mathrm{F}^1 \mathrm{J}^s f, \mathrm{J}^s v)_{L^2} - \big([\mathrm{J}^s\mathrm{F}^{\frac{1}{2}},h]v,\mathrm{F}^{\frac{1}{2}} \mathrm{J}^s v\big)_{L^2} - \frac{\mu}{3} \big( [\mathrm{J}^s\mathrm{F}^{\frac{1}{2}},h^3]\mathrm{F}^{\frac{1}{2}}\partial_x v, \mathrm{F}^1\partial_x \mathrm{J}^s v\big)_{L^2} \\
            &\hspace{0.5cm} + \frac{\mu}{2}\big( [\mathrm{J}^s\mathrm{F}^{\frac{1}{2}},h^2(\beta\partial_x b)]v, \mathrm{F}^1 \partial_x \mathrm{J}^s v\big)_{L^2} + \frac{\mu}{2} \big([\mathrm{J}^s\mathrm{F}^{\frac{1}{2}},h^2(\beta\partial_x b)]F^{\frac{1}{2}}\partial_x v, \mathrm{F}^{\frac{1}{2}}\mathrm{J}^sv\big)_{L^2} \\
            &\hspace{0.5cm} - \big([\mathrm{J}^s\mathrm{F}^{\frac{1}{2}},h(\beta\partial_x b)^2]v, \mathrm{F}^{\frac{1}{2}} \mathrm{J}^s v \big)_{L^2}.
        \end{align*}
        Now using Cauchy-Schwarz inequality, the commutator estimates \eqref{Commutator: J_sF_half}  and \eqref{Kill a derivative}, we get
        \begin{align*}
            |\mathrm{F}^{\frac{1}{2}} v |_{X^s_{\mu}}^2 &\lesssim (\frac{1}{\sqrt{\mu}}|f|_{H^{s-1}} + |v|_{H^{s-1}})|v|_{H^s}.
        \end{align*}
        %
        %
        %
        %
        %
        %
        %
        %
        %
        %
        %
        %
        Moreover, for all $s\in \R$ there holds,
        \begin{align*}
            |v|_{H^s}
            & \lesssim 
             |\mathbbm{1}_{\{\sqrt{\mu}|D|\leq 1\}} v|_{H^s} 
            +
            |\mathbbm{1}_{\{\sqrt{\mu}|D|> 1\}} v|_{H^s}
            \\
            & 
            \lesssim 
             |\mathbbm{1}_{\{\sqrt{\mu}|D|\leq 1\}} \mathrm{F}^{\frac{1}{2}}v|_{H^s}
            +
            \sqrt{\mu}|\mathbbm{1}_{\{\sqrt{\mu}|D|> 1\}} \mathrm{F}^1\partial_x  v|_{H^s}
            \\
            & 
            \lesssim
            |\mathrm{F}^{\frac{1}{2}}v|_{H^s} + \sqrt{\mu}|\mathrm{F}^1\partial_x v|_{H^s} \\
            &
            \lesssim  |\mathrm{F}^{\frac{1}{2}} v |_{X^s_{\mu}}.
        \end{align*}
        Thus, by gathering these estimates we get
        \begin{align*}
            \sqrt{\mu}|\mathrm{F}^{\frac{1}{2}}v|_{X^s_{\mu}} 
            \lesssim c\big( |f|_{H^{s-1}} + \sqrt{\mu}|\mathrm{F}^{\frac{1}{2}}v|_{X^{s-1}_{\mu}}\big),
        \end{align*}
        and allows us to argue by induction for $s\in \mathbb{N}\backslash\{0\}$, where the base case reads
        \begin{align*}
            |\mathrm{F}^{\frac{1}{2}} v |_{X^0_{\mu}} \lesssim |v |_{X^0_{\mu}} \lesssim |f |_{L^2}
        \end{align*}
        %
        %
        %
        Then use \eqref{Interpolation} to conclude the proof. In the end, we have the estimate
        \begin{align*}
            \sqrt{\mu} |\mathrm{F}^{\frac{1}{2}} v |_{H^s} \leq \sqrt{\mu}|\mathrm{F}^{\frac{1}{2}} v |_{X^s_{\mu}} \leq c | f |_{H^{s-1}}.
        \end{align*}
        \end{proof}

    \section{Consistency between \eqref{W-G-N} and \eqref{WW}} \label{Consistency}
 
    To derive system \eqref{W-G-N}, we start from the full dispersion Green-Naghdi model derived in \cite{DucheneMMWW21} for which we know the order of precision with respect to the water waves equations \eqref{WW}.

    \begin{prop}[Theorem $10.5$ in \cite{DucheneMMWW21}]\label{Consistency model Vincent} 
        There exists $n \in \mathbb{N}$ and $T >0$ such that for all $s\geq0$ and $( \mu,\varepsilon, \beta)\in\mathcal{A}_{\mathrm{SW}}$, with $b\in H^{s+n}(\R)$ and for every solution $(\zeta,\psi) \in C([0,\frac{T}{ \varepsilon }];H^{s+n}(\mathbb{R})\times \dot{H}^{s+n}(\mathbb{R}))$ to the water waves equations \eqref{WW} one has
        \begin{equation}\label{FDGN}
		       \begin{cases}
			    \partial_t \zeta + \partial_x(h\overline{V}) = 0
			    \\ 
			    \partial_t (\overline{V} + \mu \mathcal{T}[h,\beta \partial_x b] \overline{V}  ) + \partial_x \zeta + \ve \overline{V} \partial_x \overline{V} + \mu\ve  \partial_x\mathcal{R}[h,\beta \partial_x b,\overline{V}] = \mu^2(\varepsilon + \beta)R,
		  \end{cases}
	    \end{equation}
         where $\overline{V}$ is defined through \eqref{EllipticProblem} and \eqref{VbarDef}, 
        and
         \begin{align*}
            \mathcal{T}[h,\beta b] \overline{V} 
            & =
            -
            \frac{1}{3h}\partial_x \mathrm{F}^{\frac{1}{2}}(h^3 \mathrm{F}^{\frac{1}{2}}\partial_x \overline{V}) 
            +
            \frac{1}{2h}
            \Big{(}
                \partial_x \mathrm{F}^{\frac{1}{2}}(h^2 (\beta \partial_x b) \overline{V})
                -
                h^2(\beta \partial_x b)  \mathrm{F}^{\frac{1}{2}} \partial_x \overline{V}
            \Big{)}
            \\
            & \hspace{10cm}
            +
            (\beta \partial_xb)^2 \overline{V},
            \\
            \mathcal{R}[h,\beta b,\overline{V}]
            & = -  \frac{\overline{V}}{3h}\partial_x \mathrm{F}^{\frac{1}{2}}\big( h^3 \mathrm{F}^{\frac{1}{2}}\partial_x \overline{V} \big)
            -
            \frac{1}{2} h^2 \big(\mathrm{F}^{\frac{1}{2}}\partial_x \overline{V} \big)^2 \\
            & \hspace{2cm} 
            +
            \frac{1}{2} \Big(\frac{\overline{V}}{h}\partial_x \mathrm{F}^{\frac{1}{2}}\big(h^2 (\beta \partial_x b) \overline{V} \big) + h (\beta \partial_x b) \overline{V} \mathrm{F}^{\frac{1}{2}}\partial_x \overline{V} + (\beta \partial_x b)^2 \overline{V}^2 \Big),
         \end{align*}
         and where $|R|_{H^s} \leq C(\frac{1}{h_0},
         \mu_{\max},|\zeta|_{H^{s+n)}},|\partial_x \psi|_{H^{s+n}}, |b|_{H^{s+n}})$. \\

         Furthermore, we say that the water waves equations are consistent with the system  \eqref{FDGN} at the order of precision $O(\mu^2(\varepsilon+\beta))$ in the shallow water regime.
    \end{prop}

    \begin{prop}\label{Concictency of new model}
        The water waves equations are consistent with the system
        \begin{align*}
        	\begin{cases}
        		\partial_t \zeta + \partial_x(h \overline{V}) = 0
        		\\ 
        		(h  + \mu h \mathcal{T}[h,\beta \partial_x b]) \big{(}\partial_t \overline{V} + \ve \overline{V} \partial_x \overline{V}\big{)} + h \partial_x \zeta +  \mu \varepsilon h(\mathcal{Q}[h,\overline{V}] + \mathcal{Q}_b[h,b,\overline{V}]) = 0,
        	\end{cases}
        \end{align*}
        at the order of precision $O(\mu^2(\varepsilon + \beta))$, where
          \begin{align}
       	\mathcal{Q}[h,\overline{V}] 
       	& =
       	\frac{2}{3h} \partial_x\mathrm{F}^{\frac{1}{2}}\big{(} h^3  (\mathrm{F}^{\frac{1}{2}}\partial_x \overline{V})^2\big{)} 
       	\\
       	\mathcal{Q}_b[h,\beta b,\overline{V}]
       & =   
       h ( \mathrm{F}^{\frac{1}{2}} \partial_x \overline{V})^2 (\beta \partial_x b) 
       +
       \frac{1}{2h}\partial_x \mathrm{F}^{\frac{1}{2}}(h^2 \overline{V}^2 \beta \partial_x^2 b) 
       +
       \overline{V}^2(\beta \partial_x^2b)(\beta \partial_x b).
       \end{align}
    \end{prop}    
    
    \begin{proof} 
        Let us first remark that we only have to work on the second equation of system \eqref{FDGN} and that the first equation can also be written
        \begin{align}\label{first equation FDGN}
            \partial_t h = -\varepsilon\partial_x(h\overline{V}).
        \end{align}
        Then multiplying the second equation of \eqref{FDGN} by $h$ we can write 
        \begin{align*}
            h\partial_t\big(\overline{V} + \mu \mathcal{T}[h,\beta   b]\overline{V}\big) = (h + \mu h\mathcal{T}[h,\beta  b])\partial_t \overline{V} + h[\partial_t, \mu \mathcal{T}[h,\beta b]]\overline{V}.
        \end{align*}
        Now, using \eqref{first equation FDGN} we observe that the following terms are of order $\mu\ve$:
        \begin{align*}
            h[\partial_t, \mu \mathcal{T}[h,\beta b]]\overline{V} +  \mu \varepsilon h \mathcal{R}[h,v],
        \end{align*}
        and so we can use Proposition \ref{estimate F1/2} to trade the multiplier $\mathrm{F}^{\frac{1}{2}}$ with identity and terms of order $\mu^2 \ve$. Thus, following the derivation presented in \cite{Samer11} we obtain that
        \begin{align*}
        		(h  + \mu h \mathcal{T}[h,\beta \partial_x b]) \big{(}\partial_t \overline{V} + \ve \overline{V} \partial_x \overline{V}\big{)} + h \partial_x \zeta +  \mu \varepsilon h(\tilde{\mathcal{Q}}[h,\overline{V}] + \tilde{\mathcal{Q}}_b[h,b,\overline{V}]) = O(\mu^2 \ve)
        \end{align*}
        where
        \begin{align*}
       	\tilde{\mathcal{Q}}[h,\overline{V}] 
       	& =
       	\frac{2}{3h} \partial_x\big{(} h^3  (\partial_x \overline{V})^2\big{)} 
       	\\
       	\tilde{\mathcal{Q}}_b[h,\beta b,\overline{V}]
       & =   
       h (  \partial_x \overline{V})^2 (\beta \partial_x b) 
       +
       \frac{1}{2h}\partial_x (h^2 \overline{V}^2 \beta \partial_x^2 b) 
       +
       \overline{V}^2(\beta \partial_x^2 b)(\beta \partial_x b).
       \end{align*}
        To conclude, we simply apply Proposition \ref{estimate F1/2} once more to see that
        \begin{equation*}
            \tilde{\mathcal{Q}}[h,\overline{V}]  = \mathcal{Q}[h,\overline{V}]  + O(\mu)
        \end{equation*}
        and
        \begin{equation*}
            \tilde{\mathcal{Q}}_b[h, \beta b, \overline{V}]  =  \mathcal{Q}_{b}[h,\beta b, \overline{V}]  + O(\mu).
        \end{equation*}

    \end{proof} 
    \begin{remark}
        If we consider the two-dimensional case where we let $X = (x_1,x_2)$ and $\overline{V},R \in \R^2$, then system \eqref{FDGN} reads
        \begin{equation}\label{FDGN remark}
		       \begin{cases}
			    \partial_t \zeta + \nabla_X \cdot (h\overline{V}) = 0
			    \\ 
			    \partial_t (\overline{V} + \mu \mathcal{T}[h,\beta  b] \overline{V}  ) + \nabla_X \zeta + \frac{\ve}{2}\nabla_X |\overline{V}|^2 + \mu\ve  \nabla_X\mathcal{R}[h,\beta \partial_x b,\overline{V}] = \mu^2(\varepsilon + \beta)R.
		  \end{cases}
	\end{equation}
        In this case, one can exploit the observation that the quantity 
        \begin{equation}\label{U}
            U = \overline{V} + \mu \mathcal{T}[h, \beta b]\overline{V}
        \end{equation}
        approximates the gradient of the velocity potential at the free surface. Consequently, for regular solutions, one can impose the condition $\mathrm{curl} \: U|_{t=0} = 0$ and using the second equation in \eqref{FDGN remark}, we can deduce that $\mathrm{curl} \: U = 0$ whenever the solution is defined. However, this observation does not carry over to \eqref{W-G-N} since the two systems are not equivalent. On the other hand, if $\mathrm{F} = \mathrm{Id}$, then the two systems are equivalent, and one may exploit this insight to deal with the two-dimensional case.
    \end{remark}

    \begin{remark}
        The estimates in Section $2$ can be extended to two dimensions where we note that $F^{\frac{1}{2}}(\xi)$ is a radial function. Also, in light of the previous remark, it could be possible to work on system \eqref{FDGN remark} directly where we estimate the variables $\zeta$, $U$ and with $\overline{V} = \overline{V}[h, \beta b, U]$ uniquely defined by \eqref{U} (see \cite{DucheneIsrawi18} for similar observations). However, doing this change of unknowns would change the mathematical structure of the equations. So that it is not obvious that we can close the energy method in that case. 
    \end{remark}

    \section{A priori estimates}\label{AprioriEstimates}
 
        In this section, we establish \textit{a priori} bounds on the solutions of \eqref{W-G-N}. To this end, we let $\mathbf{U}  = (\zeta, v)$ and for simplicity we introduce the notation
        \begin{align*}
        	\mathscr{T}=  \mathscr{T}[h, \beta b] , \quad \mathcal{Q} = \mathcal{Q}[h, v], \quad \mathcal{Q}_b = \mathcal{Q}_b[h, \beta b, v],
        \end{align*} 
        allowing us to write \eqref{W-G-N} on the more compact form:
        \begin{equation}\label{Matrix W-G-N}
           S(\mathbf{U})(\partial_t \mathbf{U} + M_1(\mathbf{U})\partial_x \mathbf{U})+ M_2(\mathbf{U}) \partial_x\mathbf{U} 
           + 
            Q(\mathbf{U}) + Q_b(\mathbf{U})
           = \mathbf{0},
        \end{equation}
        with
        \begin{equation*}
            S(\mathbf{U}) 
            := 
            \begin{pmatrix}
                1 & 0
                \\
                0 & \mathscr{T} (\cdot)
            \end{pmatrix}
            ,
            \quad 
            M_1(\mathbf{U}) 
            : =
            \begin{pmatrix}
            	0&0
            	\\
            	0&\ve v
            \end{pmatrix}
        	,
        	\quad            
             M_2(\mathbf{U})
            : = 
            \begin{pmatrix}
                \ve v & h
                \\
                h  &   0
            \end{pmatrix},
        \end{equation*}
        and where the quadratic terms are  
        \begin{equation}
        		Q(\mathbf{U}) =
        	\begin{pmatrix}
        		0
        		\\
        		\mu \ve h \mathcal{Q}   
        	\end{pmatrix}
        	, 
        	\quad 
        	Q_b(\mathbf{U}) =
        	\begin{pmatrix}
        		-(\beta \partial_x b)v
        		\\
        		\mu \ve h  \mathcal{Q}_b 
        	\end{pmatrix},
        \end{equation}
        with $\mathcal{Q} $ as defined by  \eqref{Q} and $ \mathcal{Q}_b $ defined by  \eqref{B}.
        We may now give the energy and the energy estimate of \eqref{Matrix W-G-N}. In particular, we make the definition:
	\begin{equation}\label{Energy Es}
		E_s(\mathbf{U})  = \big{(}  \mathrm{J}^s\mathbf{U},  S(\mathbf{U}) \mathrm{J}^s \mathbf{U} \big{)}_{L^2},
	\end{equation}
        allowing us to state the following result.
	
	\begin{prop}\label{Four: Energy estimate s} Let $s> \frac{3}{2}$, $(\mu, \ve, \beta)\in \mathcal{A}_{SW}$, and $(\zeta, v) \in C([0,T]; Y^s_{\mu}(\mathbb{R}) )$ be a solution to \eqref{Matrix W-G-N} on a time interval $[0,T]$ for some $T>0 $.  Moreover, assume $   b \in H^{s+2}(\R) $ and there exist $h_0 \in (0,1)$  such that 
		\begin{equation}\label{non cav energy}
			h_0 - 1 + \beta b \leq \ve \zeta(x,t), \quad   \forall (x,t) \in \mathbb{R} \times [0,T],
		\end{equation}
            and suppose that
            \begin{equation}\label{Bound on solution}
                N(s) : = \ve \sup\limits_{t\in [0,T]} |(\zeta(t, \cdot), v(t, \cdot))|_{Y^s_{\mu}}  + \beta|b|_{H^{s+2}}\leq N^{\star},
            \end{equation}
            for some $N^{\star}\in \R^{+}$. Then, for the energy given by  \eqref{Energy Es},  there holds, 
		\begin{equation}\label{Energy estimate 3/2}
			\frac{d}{dt} E_s(\mathbf{U}) \lesssim  N(s)E_s(\mathbf{U}),
		\end{equation}
            and
		\begin{equation}\label{Equivalence energy vs norm}
			|(\zeta, v)|^2_{Y^s_{\mu}} \lesssim E_s(\mathbf{U}) \lesssim  |(\zeta, v)|^2_{Y^s_{\mu}},
		\end{equation}
		for all $0<t<T$.
	\end{prop}

        \begin{proof}
            We first prove \eqref{Equivalence energy vs norm}. We note that the energy is similar to the bilinear form defined in \eqref{variational formulation}. Thus, the estimate is a direct consequence of Step 2. in the proof of Proposition \ref{Inverse of T} and \eqref{non cav energy}. \\
 
            Next, we prove \eqref{Energy estimate 3/2}. Using   \eqref{Matrix W-G-N}, the self-adjointness of $S(\mathbf{U})$ and the invertibility provided by Proposition \ref{Inverse of T} under assumption \eqref{non cav energy}, we obtain that
        \begin{align*}
            \frac{1}{2} \frac{d}{dt} E_s(\mathbf{U}) 
            & =
            \frac{1}{2}\big{(}  \mathrm{J}^s\mathbf{U},  (\partial_tS(\mathbf{U})) \mathrm{J}^s \mathbf{U} \big{)}_{L^2} 
            +
            \big{(}  \mathrm{J}^s\partial_t \mathbf{U},  S(\mathbf{U}) \mathrm{J}^s \mathbf{U} \big{)}_{L^2} 
            \\
            & =
            \frac{1}{2}
            \big{(}  \mathrm{J}^s\mathbf{U},  (\partial_tS(\mathbf{U})) \mathrm{J}^s \mathbf{U} \big{)}_{L^2} 
            -
            \big{(}   \mathrm{J}^s M_1(\mathbf{U})\partial_x \mathbf{U}, S(\mathbf{U})\mathrm{J}^s \mathbf{U} \big{)}_{L^2} 
            \\
            &
           \hspace{0.5cm} 
            -
           \big{(}  M_2(\mathbf{U})\partial_x \mathrm{J}^s \mathbf{U}, \mathrm{J}^s \mathbf{U} \big{)}_{L^2} 
            -
            \big{(}  [\mathrm{J}^s, (S^{-1}M_2)(\mathbf{U})]\partial_x \mathbf{U},  S(\mathbf{U}) \mathrm{J}^s \mathbf{U} \big{)}_{L^2}
            \\
            & \hspace{0.5cm}
            -
            \big{(}  \mathrm{J}^s (S^{-1}Q)(\mathbf{U}),  S(\mathbf{U}) \mathrm{J}^s \mathbf{U} \big{)}_{L^2}
            -
             \big{(}  \mathrm{J}^s (S^{-1}Q_b)(\mathbf{U}),  S(\mathbf{U}) \mathrm{J}^s \mathbf{U} \big{)}_{L^2}  
            \\
            & = :
            I + II + III + IV + V + VI.
        \end{align*}

        \noindent
        \underline{Control of $I$}. We first use the equation for $\partial_t \zeta$ in \eqref{Matrix W-G-N}, together with the Sobolev embedding $H^{s-1}(\mathbb{R}) \hookrightarrow L^{\infty}(\R)$ with $s-1>\frac{1}{2}$ and the algebra property to deduce the estimate:
        \begin{align}
            |\partial_t h|_{L^{\infty}} \notag
            & \leq
            \ve |\partial_x(hv)|_{L^{\infty}} 
            \\
            &
            \lesssim
            \ve (1+\ve |\zeta|_{H^{s}} +   \beta|b|_{H^s})|v|_{H^s}. 
            \label{L infty est dt zeta}
        \end{align}
        Therefore, by definition \eqref{T[h]} of $\mathscr{T}[h,\beta b]$, using integration by parts, Hölder's inequality, \eqref{Bound on solution}, Sobolev embedding, and \eqref{Equivalence energy vs norm} we obtain the bound
        \begin{align*}
            | I |
            & 
            \leq
            \frac{1}{2} 
            |\big{(}
            \mathrm{J}^s v, (\partial_t h) \mathrm{J}^s v
            \big{)}_{L^2}|
            + 
            \frac{ \mu}{6}
            |\big{(}
            \mathrm{F}^{\frac{1}{2}}\partial_x \mathrm{J}^s v , (\partial_t (h^3)) \mathrm{F}^{\frac{1}{2}}\partial_x \mathrm{J}^s v
            \big{)}_{L^2}|
            \\
            & 
            \hspace{0.5cm}
            +
            \frac{ \mu}{2}
            |\big{(}
            \mathrm{F}^{\frac{1}{2}}\partial_x \mathrm{J}^s v , (\partial_t (h^2))  \beta (\partial_x b)\mathrm{J}^s v
            \big{)}_{L^2}|
            +
            \mu
            |\big{(}
             \mathrm{J}^s v , (\partial_t h)  (\beta \partial_x b)^2\mathrm{J}^s v
            \big{)}_{L^2}|
            \\
            & 
            \lesssim
            N(s)E_s(\mathbf{U}),
        \end{align*}
        for $s>\frac{3}{2}$.\\

		\noindent
		\underline{Control of $II$.} By definition of $\mathscr{T}[h,\beta b]$ we must deal with the terms:
		\begin{align*}
			II
			& =
			-\ve \big{(} \mathrm{J}^s(v\partial_xv), h\mathrm{J}^sv \big{)}_{L^2} 
			+
			\frac{\mu\ve}{3} \big{(} \mathrm{J}^s(v\partial_xv), \partial_x\mathrm{F}^{\frac{1}{2}}(h^3 \partial_x\mathrm{F}^{\frac{1}{2}}\mathrm{J}^sv) \big{)}_{L^2}
                \\
                & 
                \hspace{0.5cm}
                -
			\frac{\mu\ve\beta}{2} \big{(} \mathrm{J}^s(v\partial_xv), \partial_x\mathrm{F}^{\frac{1}{2}}(h^2(\partial_x b) \mathrm{J}^sv) \big{)}_{L^2}
                 +
			\frac{\mu\ve\beta}{2} \big{(} \mathrm{J}^s(v\partial_xv), h^2 (\partial_x b)  \partial_x\mathrm{F}^{\frac{1}{2}}\mathrm{J}^sv \big{)}_{L^2}
			\\
                & 
                \hspace{0.5cm}
                -
                \frac{\mu\ve\beta}{2} \big{(} \mathrm{J}^s(v\partial_xv), h (\beta \partial_x b)^2\mathrm{J}^sv \big{)}_{L^2}
                \\
			 & = :
			II_1 + II_2 + II_3 + II_4 + II_5.
		\end{align*}
		Using integration by parts, we may decompose $II_1$ into two pieces 
		\begin{align*}
		II_1& =
		-\ve \big{(} hv\mathrm{J}^s\partial_xv, \mathrm{J}^sv \big{)}_{L^2} 
			-
			 \ve \big{(} [\mathrm{J}^s,v]\partial_xv, h\mathrm{J}^sv \big{)}_{L^2} 
			 \\
			 &
			  = 
			  \frac{\ve}{2} \big{(} (\partial_x(h v))\mathrm{J}^sv, \mathrm{J}^sv \big{)}_{L^2} 
			  -
			  \ve \big{(} [\mathrm{J}^s,v]\partial_xv, h\mathrm{J}^sv \big{)}_{L^2}. 
		\end{align*}
		Then by Hölder's inequality, Sobolev embedding, and the  commutator estimate \eqref{Commutator estimates}, we obtain the estimate:
		\begin{align*}
			|II_1| \lesssim \ve(1+|h-1|_{H^s}) |v|_{H^s}^3.
		\end{align*}
            We also note that $II_5$ can be estimated in the same way, and we obtain easily that
            \begin{align*}
                |II_5| \lesssim  \ve(1+|h-1|_{H^s}) |b|_{H^{s+1}}^2|v|_{H^s}^3.
            \end{align*}
		For $II_2$, we also use integration by parts to make the observation:
		\begin{align*}
			II_2
			& = 
			-\frac{\mu\ve}{3} \big{(} [\mathrm{J}^s,\mathrm{F}^{\frac{1}{2}}\partial_x (v\cdot)]\partial_xv, h^3 \partial_x\mathrm{F}^{\frac{1}{2}}\mathrm{J}^sv\big{)}_{L^2}
			-
			\frac{\mu\ve}{3} \big{(} \mathrm{F}^{\frac{1}{2}}\partial_x (v\mathrm{J}^s\partial_xv), h^3 \partial_x\mathrm{F}^{\frac{1}{2}}\mathrm{J}^sv\big{)}_{L^2}
			\\
			&
			= : II_2^1 + II_2^2.
		\end{align*}
		Then we treat $II_2^1$ with Hölder's inequality, Sobolev embedding, and \eqref{Commutator: J Fdx} to get
		\begin{align*}
			|II_2^1| 
			& \lesssim (1+|h^3-1|_{H^s})  |v|_{H^s}|v|_{X^s_{\mu}}^2.
		\end{align*}
		On the other hand, we need to decompose $II_2^2$ further and carefully distribute the $\mu$:
		\begin{align*}
			II_2^2 
			& =
				-\frac{\mu\ve}{3} \big{(}v \mathrm{F}^{\frac{1}{2}}\mathrm{J}^s\partial_x^2v, h^3 \partial_x\mathrm{F}^{\frac{1}{2}}\mathrm{J}^sv\big{)}_{L^2} 
				-
				\frac{\mu\ve}{3} \big{(} [\mathrm{F}^{\frac{1}{2}},v]\mathrm{J}^s\partial_x^2v, h^3 \partial_x\mathrm{F}^{\frac{1}{2}}\mathrm{J}^sv\big{)}_{L^2} 
				\\
				& 
				\hspace{0.5cm} 
				-
				\frac{\mu\ve}{3} \big{(} \mathrm{F}^{\frac{1}{2}}\big{(}(\partial_x v)\mathrm{J}^s\partial_xv\big{)}, h^3 \partial_x\mathrm{F}^{\frac{1}{2}}\mathrm{J}^sv\big{)}_{L^2}  
				\\
				& 
				=:
				II_2^{2,1} + II_2^{2,2} + II_2^{2,3}.
 		\end{align*}
 		For $II_2^{2,1}$, we simply integrate by parts and argue as we did for $II_1$ to obtain
		\begin{align*}
			|II_2^{2,1}|
			& \lesssim
			\mu\ve |\partial_x(h^3v)|_{H^{s-1}} |\mathrm{F}^{\frac{1}{2}}\mathrm{J}^s\partial_xv|_{H^s}^2
			\\
			& \lesssim 
			\ve (1+ |h^3-1|_{H^s} )|v|_{H^s} |v|_{X_{\mu}^s}^2.
		\end{align*}
		For $II_2^{2,2}$, we use Hölder's inequality, Sobolev embedding, and \eqref{Commutator: F_half} to directly obtain that
        \begin{align*}
        	|II_2^{2,2} | 
        	& 
        	\lesssim \mu  \ve (1+ |h^3-1|_{H^s})|v|_{X^s_{\mu}}  |\mathrm{F}^{\frac{1}{2}} \partial_x^2  v|_{H^{s-1}} |\mathrm{F}^{\frac{1}{2}}\partial_xv|_{H^s}
        	\\
        	& \lesssim   \ve (1+ |h^3-1|_{H^s})|v|_{X^s_{\mu}}^3.
        \end{align*}
        For $II_2^{2,3}$, we also need to be careful in the distribution of $\mu$. In fact, we need to use Plancherel, then Cauchy-Schwarz and \eqref{F -1/2 } to get
        \begin{align*}
	       | II_2^{2,3}|
	        &= 
	        \frac{\mu\ve}{3} |\big{(}\mathrm{F}^{\frac{1}{2}}\mathrm{J}^s\partial_xv,  \mathrm{F}^{-\frac{1}{2}}\Big{(}(\partial_x v)\mathrm{F}^{\frac{1}{2}}\big{(}h^3 \partial_x\mathrm{F}^{\frac{1}{2}}\mathrm{J}^sv\big{)}\Big{)}\big{)}_{L^2}  |
	        \\
	        &
	        \lesssim \ve  |v|_{X^s_{\mu}}
	        \big{(}
	        \sqrt{\mu}| (\partial_x v)\mathrm{F}^{\frac{1}{2}}\big{(}h^3 \partial_x\mathrm{F}^{\frac{1}{2}}\mathrm{J}^sv\big{)}|_{L^2} + \mu^{\frac{3}{4}}|\mathrm{D}^{\frac{1}{2}} \Big{(}(\partial_x v)\mathrm{F}^{\frac{1}{2}}\big{(}h^3 \partial_x\mathrm{F}^{\frac{1}{2}}\mathrm{J}^sv\big{)}\Big{)}|_{L^2}
	         \big{)}
	         \\
	         & =: \ve  |v|_{X^s_{\mu}} (A+B).
        \end{align*}
    	Then estimate $A$ by Hölder's inequality, the Sobolev embedding, and the boundedness of $\mathrm{F}^{\frac{1}{2}}$ on $L^2(\R)$ to get
    	\begin{align*}
    		A \lesssim |v|_{H^s}(1+ |h^3-1|_{H^s})|v|_{X^s_{\mu}},
    	\end{align*}
   		while for $B$, we also use \eqref{D_half - fracLeib - Sob} and \eqref{Kill half derivative} to get	
   		\begin{align*}
   			B 
   			& \lesssim \mu^{\frac{3}{4}} |v|_{H^s} |\mathrm{F}^{\frac{1}{2}}\big{(}h^3 \partial_x\mathrm{F}^{\frac{1}{2}}\mathrm{J}^sv\big{)}|_{H^{\frac{1}{2}}}
   			\\
   			& 
   			\lesssim 
   			\sqrt{\mu}
   			|v|_{H^s} |h^3 \partial_x\mathrm{F}^{\frac{1}{2}}\mathrm{J}^sv|_{L^2}
   			\\
   			&
   			\lesssim |v|_{H^s} (1+|h^3-1|_{H^s}) |v|_{X^s_{\mu}}.
   		\end{align*}
            Next, we use integration by parts to decompose $II_3$ into several pieces:
            \begin{align*}
                II_3 
                & = 
                -
			\frac{\mu\ve\beta}{2} \big{(} [\mathrm{J}^s,v]\partial_xv, \partial_x\mathrm{F}^{\frac{1}{2}}(h^2(\partial_x b) \mathrm{J}^sv) \big{)}_{L^2}
                +
			\frac{\mu\ve\beta}{2} \big{(} \mathrm{F}^{\frac{1}{2}}((\partial_xv)\mathrm{J}^s\partial_xv), h^2(\partial_x b) \mathrm{J}^sv \big{)}_{L^2}
                \\
                & 
                \hspace{0.5cm}
                +
                \frac{\mu\ve\beta}{2} \big{(} \mathrm{F}^{\frac{1}{2}}(v\mathrm{J}^s\partial_x^2v), h^2(\partial_x b) \mathrm{J}^sv \big{)}_{L^2}
                \\
                & = :
                II_3^1 + II_3^2 + II_3^3.
            \end{align*}
            Then for $II_3^1$, we apply \eqref{Commutator estimates}, \eqref{Kill half derivative}, and \eqref{D_half - fracLeib - Sob} to obtain that
            \begin{align*}
                |II_3^1| &  \lesssim \ve |v|_{H^s}^2 \mu^{\frac{3}{4}}|\mathrm{D}^{\frac{1}{2}}(h^2(\partial_x b) \mathrm{J}^sv)|_{L^2}
                \\ 
                & 
                \lesssim \ve (1 + |h^2-1|_{H^s})|b|_{H^{s+1}}|v|_{X^s_{\mu}}^3.
            \end{align*}
            For $II_3^2$, we argue as for $II_2^{2,3}$ to get that
            \begin{align*}
                |II_3^2| 
                & \lesssim 
                \ve \mu | \big{(}F^{\frac{1}{2}}\mathrm{J}^s\partial_xv, F^{-\frac{1}{2}}( (\partial_xv)\mathrm{F}^{\frac{1}{2}}(h^2(\partial_x b) \mathrm{J}^sv)) \big{)}_{L^2}|
                \\
                & 
                \lesssim 
                \ve |v|_{X^s_{\mu}} \sqrt{\mu}(|(\partial_xv)\mathrm{F}^{\frac{1}{2}}(h^2(\partial_x b) \mathrm{J}^sv)|_{L^2} + \mu^{\frac{1}{4}}|\mathrm{D}^{\frac{1}{2}}((\partial_xv)\mathrm{F}^{\frac{1}{2}}(h^2(\partial_x b) \mathrm{J}^sv))|_{L^2})
                \\
                & 
                \lesssim \ve |v|_{X^s_{\mu}}(1+|h^2-1|_{H^s}) |b|_{H^{s+1}}|v|_{H^s}^2.
            \end{align*}
            For $II_3^3$, we first make the decomposition
            \begin{align*}
                II_3^2 
                & = 
                \frac{\mu\ve\beta}{2} \big{(} [\mathrm{F}^{\frac{1}{2}},v]\mathrm{J}^s\partial_x^2v, h^2(\partial_x b) \mathrm{J}^sv \big{)}_{L^2} 
                +
                \frac{\mu\ve\beta}{2} \big{(} v\mathrm{F}^{\frac{1}{2}}\mathrm{J}^s\partial_x^2v, h^2(\partial_x b) \mathrm{J}^sv \big{)}_{L^2} 
                \\
                & = : 
                II_3^{2,1} + II_{3}^{2,2}.
            \end{align*}
            For $II_3^{2,1}$, we employ Hölder's inequality, Sobolev embedding, and \eqref{Commutator: F_half} to get that
            \begin{align*}
                |II_3^{2,1}| \lesssim  \mu \ve |v|_{X^s_{\mu}} |\mathrm{F}^{\frac{1}{2}}\partial_x v|_{H^{s}}|b|_{H^{s+1}}(1+|h^2-1|_{H^s})|v|_{H^s}.
            \end{align*}
            Lastly, for $II_3^{2,2}$, we use integration by parts to make the observation that
            \begin{align*}
                II_{3}^{2,2} 
                & = 
                -
                \frac{\mu\ve\beta}{2} \big{(} \mathrm{F}^{\frac{1}{2}}\mathrm{J}^s\partial_xv, (vh^2(\partial_x b)) v\mathrm{J}^s\partial_x v \big{)}_{L^2}
                -
                \frac{\mu\ve\beta}{2} \big{(} \mathrm{F}^{\frac{1}{2}}\mathrm{J}^s\partial_xv, (\partial_x (vh^2(\partial_x b)) )\mathrm{J}^sv \big{)}_{L^2}
                \\
                & =
                \frac{\mu\ve\beta}{2} \big{(} \mathrm{F}^{\frac{1}{2}}\mathrm{J}^s\partial_xv, vh^2(\partial_x b) [\mathrm{J}^s,v]\partial_x v \big{)}_{L^2} - II_4
                 -
                \frac{\mu\ve\beta}{2} \big{(} \mathrm{F}^{\frac{1}{2}}\mathrm{J}^s\partial_xv, (\partial_x (vh^2(\partial_x b)) )\mathrm{J}^sv \big{)}_{L^2} 
                .
            \end{align*}
            Then we may use Hölder's inequality, \eqref{Commutator estimates}, \eqref{Bound on solution}, \eqref{Equivalence energy vs norm}, and Sobolev embedding to get that
            \begin{align*}
                |II_3^{2,2} + II_4| \lesssim N(s) E_s(\mathbf{U}).
            \end{align*}
	   	Gathering all these estimates, using \eqref{Bound on solution} and \eqref{Equivalence energy vs norm},  allows us to conclude that
	   	\begin{equation*}
	   		|II| \lesssim N(s) E_s(\mathbf{U}).
	   	\end{equation*}

        \noindent
        \underline{Control of $III$.}
        Then by definition, we must estimate the terms:
        \begin{align*}
            III
            & =
            -\ve \big{(} v  \partial_x \mathrm{J}^s\zeta , \mathrm{J}^s \zeta \big{)}_{L^2} 
            -
             \big{(} h \partial_x \mathrm{J}^s v, \mathrm{J}^s \zeta  \big{)}_{L^2}
            -
             \big{(} h \partial_x \mathrm{J}^s \zeta, \mathrm{J}^s v  \big{)}_{L^2}
        \end{align*}
        For the estimate on these terms, we integrate by parts and apply Hölder's inequality, Sobolev embedding, and \eqref{Equivalence energy vs norm} to deduce
        \begin{align*}
            |III| 
            & 
            \leq 
            \frac{\ve}{2}|\big{(} (\partial_xv)  \mathrm{J}^s\zeta , \mathrm{J}^s \zeta \big{)}_{L^2}|
            +
             |\big{(} (\partial_xh)  \mathrm{J}^s v, \mathrm{J}^s \zeta  \big{)}_{L^2}|
            \\ 
            & 
            \lesssim 
            N(s)E_s(\mathbf{U}).
        \end{align*}\\

        \noindent
        \underline{Control of $IV$.} We decompose each term in $IV$ and estimate them separately. In particular, we must estimate the following terms,
        \begin{align*}
            IV 
            & =
            -\ve \big{(} [\mathrm{J}^s, v] \partial_x \zeta, \mathrm{J}^s \zeta \big{)}_{L^2}
            -
            \big{(} [\mathrm{J}^s, h] \partial_x v, \mathrm{J}^s \zeta \big{)}_{L^2}
            -
            \big{(} [\mathrm{J}^s, \mathscr{T}^{-1}(h \boldsymbol{\cdot})] \partial_x \zeta, \mathscr{T} \mathrm{J}^s v \big{)}_{L^2}
            \\
            & 
            =: IV_1 + IV_2 +  IV_3.
        \end{align*}
        The first two terms are easily controlled by Cauchy-Schwarz and \eqref{Commutator estimates}:
        \begin{align*}
            |IV_1| + |IV_2| \lesssim \ve  |v|_{L^{\infty}} |\zeta|^2_{H^s} + (\ve |\zeta|_{L^{\infty}} + \beta |b|_{L^{\infty}})|\zeta|_{H^s}|v|_{H^s}.
        \end{align*}
        Then use Sobolev embedding and \eqref{Equivalence energy vs norm} to conclude. However, need to decompose the remaining term further.  To do so, we make the observation that
        \begin{align}\label{Formula for T[Js,invT]}
            \mathscr{T}[\mathrm{J}^s,\mathscr{T}^{-1}(h\boldsymbol{\cdot})]f 
                &
                =
                -[\mathrm{J}^s,h]\mathscr{T}^{-1}(hf)+\frac{\mu}{3}\partial_x \mathrm{F}^{\frac{1}{2}}[\mathrm{J}^s,h^3]\partial_x\mathrm{F}^{\frac{1}{2}}\big{(}\mathscr{T}^{-1}(hf)\big{)}
                \\ 
                &\notag
                \hspace{0.5cm}
                - \frac{\mu}{2} \partial_x \mathrm{F}^{\frac{1}{2}}[\mathrm{J}^s,h^2 \beta (\partial_xb)]\mathscr{T}^{-1}(hf) 
                +
                \frac{\mu}{2} [\mathrm{J}^s,h^2 \beta (\partial_xb)]\partial_x \mathrm{F}^{\frac{1}{2}}\mathscr{T}^{-1}(hf)
                \\
                &\notag
                \hspace{0.5cm}
                +
                \mu [\mathrm{J}^s, h(\beta \partial_x b)^2]\mathscr{T}^{-1}(hf) +[\mathrm{J}^s,h]f .
        \end{align}
        Then by this identity, the self-adjointness of $\mathscr{T}[h,\beta b]$, and integration by parts, we may decompose $IV_3$ into six pieces:
        \begin{align*}
            IV_3
            & =
             \big{(}[\mathrm{J}^s,h]\mathscr{T}^{-1}(h \partial_x \zeta), \mathrm{J}^s v \big{)}_{L^2}
            +
             \frac{\mu}{3}\big{(}[\mathrm{J}^s,h^3]\partial_x\mathrm{F}^{\frac{1}{2}}\big{(}\mathscr{T}^{-1}(h\partial_x \zeta)\big{)}, \partial_x \mathrm{F}^{\frac{1}{2}}\mathrm{J}^s v \big{)}_{L^2}
            \\ 
            &
            \hspace{0.5cm}
            -
            \frac{\mu}{2}
            \big{(}
            [\mathrm{J}^s,h^2 \beta (\partial_xb)]\mathscr{T}^{-1}(h\partial_x\zeta) ,
            \partial_x \mathrm{F}^{\frac{1}{2}}
            \mathrm{J}^s v 
            \big{)}_{L^2}
            -
            \frac{\mu}{2}
            \big{(}
            [\mathrm{J}^s,h^2 \beta (\partial_xb)]\partial_x \mathrm{F}^{\frac{1}{2}}\mathscr{T}^{-1}(h\partial_x\zeta) ,
            \mathrm{J}^s v 
            \big{)}_{L^2}
            \\
            &
            \hspace{0.5cm}
            -
            \mu 
            \big{(}
            [\mathrm{J}^s, h(\beta \partial_x b)^2]\mathscr{T}^{-1}(h\partial_x \zeta), \mathrm{J}^sv 
            \big{)}_{L^2}
            -
            \big{(}[\mathrm{J}^s,h]\partial_x \zeta, \mathrm{J}^s v \big{)}_{L^2} 
            \\
            & = :
            IV_3^1 + IV_3^2 + IV_3^3 + IV_3^4 + IV_3^5 + IV_3^6.
        \end{align*}
        For $IV_3^1$, use Cauchy-Schwarz inequality, \eqref{Commutator estimates}, Sobolev embedding, \eqref{Inverse est T}, \eqref{Bound on solution}, and the algebra property of $H^{s-1}(\R)$ for $s-1>\frac{1}{2}$ to get the bound
        \begin{align*}
            |IV^1_3| 
            & \lesssim
            (\ve |\partial_x \zeta|_{L^{\infty}} + \beta |\partial_x b|_{L^{\infty}})|h \partial_x \zeta |_{H^{s-1}} |v|_{H^s} 
            \\
            & 
            \lesssim \ve |\zeta|_{H^s}^2|v|_{H^s} + \beta|\partial_x b|_{L^{\infty}}|\zeta|_{H^s}|v|_{H^s}.
        \end{align*}
        Similarly, when estimating $IV_3^2$ we also use \eqref{Equivalence of norms} and the inverse estimate \eqref{Inverse est F T} to deduce
        \begin{align*}
            |IV_3^2|
            & \lesssim 
            \ve \mu |\zeta|_{H^s}|\mathrm{F}^{\frac{1}{2}}\mathscr{T}^{-1} (h \partial_x \zeta )|_{H^s} |\partial_x \mathrm{F}^{\frac{1}{2}}\mathrm{J}^sv|_{H^s} 
            \\
            & 
            \lesssim 
            \ve |\zeta|_{H^s}^2|v|_{X^s_{\mu}}
            + \beta|\partial_x b|_{L^{\infty}}|\zeta|_{H^s}|v|_{X_{\mu}^s}.
        \end{align*}
        Next, we see that $IV_3^3 + IV_3^4 + IV_3^5$ offers no other difficulties. In fact, applying the same estimates as above, with \eqref{Bound on solution}, yields 
        \begin{align*}
            |IV_3^3| + |IV_3^4| + |IV_3^5| \lesssim  (1+\ve|\zeta|_{H^s} )\beta|\partial_x b |_{L^{\infty}}|\zeta|_{H^s} |v|_{X^s_{\mu}}.
        \end{align*}
        Lastly, $IV_3^6$ is controlled by Cauchy-Schwarz inequality, \eqref{Commutator estimates} and Sobolev emebedding:
        \begin{align*}
            |IV_3^6|\lesssim \ve |\zeta|_{H^s}^2 |v|_{H^s}.
        \end{align*} \\
    
    \noindent
    \underline{Control of $V$.} We need to make a careful decomposition of the following term
    	\begin{align*}
    		\mathscr{T}\big{(}\mathrm{J}^s\mathscr{T}^{-1}(h \mathcal{Q})\big{)} 
    		=
    		\frac{2 }{3}	\mathscr{T}
    		\Big{(}
    		\mathrm{J}^s\mathscr{T}^{-1}
    		\big{(}
    		\partial_x \mathrm{F}^{\frac{1}{2}}(h^3 (\mathrm{F}^{\frac{1}{2}} \partial_xv)^2)
    		\big{)} 
    		\Big{)}. 
    	\end{align*}
    	To do so, we use the identity 
    	\begin{align*}
    		\mathscr{T}\Big{(}\mathrm{J}^s\mathscr{T}^{-1}\big{(}\mathrm{F}^{\frac{1}{2}}\partial_x(f g)\big{)}\Big{)}  
    		 & =
    		 -[\mathrm{J}^s,\mathscr{T}]\mathscr{T}^{-1}(\mathrm{F}^{\frac{1}{2}}\partial_x(f g ))
    		 +
    		 [\mathrm{J}^s,\mathrm{F}^{\frac{1}{2}}\partial_x(f\cdot)]g
    		 \\
    		 & 
    		 \hspace{0.5cm}
    		 +
    		 \mathrm{F}^{\frac{1}{2}}\partial_x(f\mathrm{J}^s g),
    	\end{align*}
    	then use integration by parts to make the decomposition
    	\begin{align*}
    		V 
    		& = 
    		\frac{2\mu \ve}{3} \Big{[}
    		\big{(}[\mathrm{J}^s,\mathscr{T}]\mathscr{T}^{-1}(h \mathcal{Q}), \mathrm{J}^s v \big{)}_{L^2} 
    		+
    		\big{(}[\mathrm{J}^s,h^3 ] \big{(}(\mathrm{F}^{\frac{1}{2}} \partial_xv)^2 \big{)}, \mathrm{F}^{\frac{1}{2}}\partial_x\mathrm{J}^s v \big{)}_{L^2}
    		\\
    		& 
    		\hspace{0.5cm}
    		+
    		\big{(}h^3\mathrm{J}^s \big{(}(\mathrm{F}^{\frac{1}{2}} \partial_xv)^2 \big{)}, \mathrm{F}^{\frac{1}{2}}\partial_x\mathrm{J}^s v \big{)}_{L^2}
    		\Big{]}
    		\\
    		& = 
    		V_1 + V_2 + V_3.
    	\end{align*}
        We treat $V_1$ first, where we must control the following terms:
        \begin{align*}
           	V_1
            	& = \mu\ve
            	\big{(}[\mathrm{J}^s,h]\mathscr{T}^{-1}(h \mathcal{Q}), \mathrm{J}^s v \big{)}_{L^2}
            	+
            	\frac{\mu^2 \ve }{3}\big{(}[\mathrm{J}^s,h^3]\partial_x\mathrm{F}^{\frac{1}{2}}\big{(}\mathscr{T}^{-1}(h \mathcal{Q} )\big{)}, \partial_x \mathrm{F}^{\frac{1}{2}}\mathrm{J}^s v \big{)}_{L^2}
            	\\ 
            	&
            	\hspace{0.5cm}
            	-
            	\frac{\mu^2 \ve}{2}
            	\big{(}
            	[\mathrm{J}^s,h^2 \beta (\partial_xb)]\mathscr{T}^{-1}(h \mathcal{Q}) ,
            	\partial_x \mathrm{F}^{\frac{1}{2}}
            	\mathrm{J}^s v 
            	\big{)}_{L^2}
            	-
            	\frac{\mu^2\ve}{2}
            	\big{(}
            	[\mathrm{J}^s,h^2 \beta (\partial_xb)]\partial_x \mathrm{F}^{\frac{1}{2}}\mathscr{T}^{-1}(h \mathcal{Q}) ,
            	\mathrm{J}^s v 
            	\big{)}_{L^2}
            	\\
            	&
            	\hspace{0.5cm}
            	-
            	\mu^2\ve
            	\big{(}
            	[\mathrm{J}^s, h(\beta \partial_x b)^2]\mathscr{T}^{-1}(h\partial_x  \mathcal{Q}), \mathrm{J}^sv 
            	\big{)}_{L^2}
            	\\
            	& = :
            	V_1^1+V_1^2+V_1^3+V_1^4+V_1^5.
        \end{align*}
        To estimate the first term, $V_1$, we simply argue as above. Indeed, by \eqref{Commutator estimates}, the Sobolev embedding, and using that $X^{s-1}(\R) \subset H^{s-1}(\R)$ with \eqref{Inverse est T} yields
        \begin{align*}
            |V_1^1|
            &\lesssim  \mu\ve
            | [\mathrm{J}^s,h]\mathscr{T}^{-1}(h\mathcal{Q})|_{L^2}  |v|_{H^s}
            \\
            & 
            \lesssim \mu \ve
            |h\mathcal{Q}|_{H^{s-1}} |\zeta|_{H^s} |v|_{H^s}.
        \end{align*}
        Then to estimate $|h\mathcal{Q}|_{H^{s-1}}$, we 
        first observe by the interpolation inequality \eqref{Interpolation} and Young's inequality that
    	\begin{align*}
    		\sqrt{\mu}|\mathrm{F}^{\frac{1}{2}} \partial_x v|_{H^{s-{\frac{1}{2}}}}^2 
    		& \lesssim 	|\mathrm{F}^{\frac{1}{2}} \partial_x v|_{H^{s-{1}}}\sqrt{\mu}|\mathrm{F}^{\frac{1}{2}} \partial_x v|_{H^{s}}
    		\\
    		& 
    		\lesssim 
    			| v|^2_{H^{s}}
                 +
                \mu |\mathrm{F}^{\frac{1}{2}} \partial_x v|_{H^{s}}^2.
    	\end{align*}
        Thus, we may estimate $|h\mathcal{Q}|_{H^{s-1}}$
        by using \eqref{Inverse est T}, the algebra property of $H^{s-{\frac{1}{2}}}(\R)$ for $s-\frac{1}{2}>1$ and combined with \eqref{Commutator: J_sF_half} and \eqref{Kill half derivative}:
        \begin{align*}
        	\mu |h\mathcal{Q}|_{H^{s-1}} 
        	& = 
       		\frac{\mu }{3} |\partial_x\mathrm{F}^{\frac{1}{2}}\big{(}h^3((\mathrm{F}^{\frac{1}{2}}\partial_x v)^2)) |_{H^{s-1}}
       		\\
        	& \lesssim 
            \mu  |\partial_x\mathrm{F}^{\frac{1}{2}}((\mathrm{F}^{\frac{1}{2}}\partial_x v)^2)  |_{H^{s-1}}
            +
            \mu  |\partial_x\mathrm{F}^{\frac{1}{2}}\big{(}(h^3-1)\mathrm{F}^{\frac{1}{2}}((\partial_x v)^2)\big{)} |_{H^{s-1}}
            \\
            & \lesssim
            \sqrt{\mu} |\mathrm{F}^{\frac{1}{2}}\partial_xv|^2_{H^{s-\frac{1}{2}}}
            +
             \mu |[\mathrm{J}^{s}\mathrm{F}^{\frac{1}{2}},h^3] ((\mathrm{F}^{\frac{1}{2}}\partial_x v)^2)|_{L^2}
            +
            \mu 
            | (h^3-1)\mathrm{F}^{\frac{1}{2}}((\mathrm{F}^{\frac{1}{2}}\partial_x v)^2)|_{H^s} 
            \\ 
            &  
            \lesssim N(s)|v|_{H^s}^2,
        \end{align*}
        and using \eqref{Bound on solution}, we deduce that
        \begin{align*}
            |V_1^1|
            & \lesssim N(s) |v|_{H^s}^2.
        \end{align*}
        Next, we consider $V_1^2$ and observe that we can have a similar bound. Indeed, using \eqref{Bound on solution},  \eqref{Commutator estimates}, and \eqref{Commutator: J_sF_half}  we observe that
        \begin{align*}
            |V_1^2|
            & \lesssim
              \mu^2\ve |[\mathrm{J}^s,h^3] \partial_x\mathrm{F}^{\frac{1}{2}}\mathscr{T}^{-1}(h\mathcal{Q})|_{L^2}  |\mathrm{F}^{\frac{1}{2}} \partial_x v |_{H^s}
              \\
              & \lesssim
              \mu^{\frac{3}{2}} \ve  |\mathrm{F}^{\frac{1}{2}}\mathscr{T}^{-1}(h\mathcal{Q})|_{H^s} |v|_{X^s_{\mu}}
              \\
              & \lesssim   \mu\ve  |h\mathcal{Q}|_{H^{s-1}} |v|_{X^s_{\mu}},
        \end{align*}
        and we use the previous estimates to obtain that
        \begin{align*}
           | V_1^2| \lesssim N(s)|v|_{X^s_{\mu}}^2.
        \end{align*}
        Moreover, we note that it is straightforward to estimate $|V_1^3|+...+|V_1^5|$  arguing as we did for $V_1^1$ and $V_1^2$. Thus, gathering all these estimates and using \eqref{Equivalence energy vs norm} yields,
        \begin{align*}
            |V_1|\lesssim N(s)E_s(\mathbf{U}).
        \end{align*}
		Next, we estimate $V_2$ using Hölder's inequality, \eqref{Bound on solution} and \eqref{Commutator estimates} to obtain
		\begin{align*}
			|V_2|&  \lesssim
			 \mu \ve 	(1+|h^3-1|_{H^s})|\mathrm{F}^{\frac{1}{2}}v|_{H^s}^2|\partial_x \mathrm{F}^{\frac{1}{2}}v|_{H^s}
			 \\
			  & \lesssim  N(s)|v|^2_{X^s_{\mu}}.
		\end{align*}
		Lastly, for $V_3$, we inject a commutator and use Hölder's inequality, Sobolev embedding, \eqref{Bound on solution}  and \eqref{Commutator estimates} to get
		\begin{align*}
			|V_3| 
			& \lesssim
			 \mu \ve |\big{(}h^3[\mathrm{J}^s ,(\mathrm{F}^{\frac{1}{2}} \partial_xv)](\mathrm{F}^{\frac{1}{2}} \partial_xv) , \mathrm{F}^{\frac{1}{2}}\partial_x\mathrm{J}^s v \big{)}_{L^2}|
			 +\mu \ve
			  |\big{(}h^3(\mathrm{F}^{\frac{1}{2}} \partial_x v)\mathrm{J}^s \mathrm{F}^{\frac{1}{2}} \partial_xv, \mathrm{F}^{\frac{1}{2}}\partial_x\mathrm{J}^s v \big{)}_{L^2} |
			  \\
			  &
			  \lesssim \mu \ve |\mathrm{F}^{\frac{1}{2}} \partial_xv|_{H^s}|\mathrm{F}^{\frac{1}{2}} \partial_xv|_{H^{s-1}}  |\mathrm{F}^{\frac{1}{2}}\partial_x\mathrm{J}^s v  |_{L^2}
			  \\
			  &  \lesssim N(s)|v|_{X^s_{\mu}}^2.
		\end{align*}\\

        \noindent
        \underline{Control of $VI$}. To complete the proof we need to estimate the remaining part:
        \begin{align*}
            VI 
            & =
            -
            \ve \mu 
            \big{(}
            [\mathrm{J}^s, \mathscr{T}] \mathscr{T}^{-1} \mathcal{Q}_b, \mathrm{J}^sv
            \big{)}_{L^2}
            -
            \ve \mu
            \big{(}
            \mathrm{J}^s\mathcal{Q}_b, \mathrm{J}^sv
            \big{)}_{L^2}
            \\
            & 
            \hspace{0.5cm}
            -\big{(}\mathrm{J}^s((\beta \partial_x b)v),\mathrm{J}^s\zeta\big{)}_{L^2} 
            \\
            & = : VI_1 + VI_2 + VI_3.
        \end{align*}
        The estimate in $VI$ is similar to the one of $V$, where we now have to deal with the following terms
        \begin{align*}
             VI_1 
            & = 
            -\mu\ve 
            \big{(} [\mathrm{J}^s,h]\mathscr{T}^{-1}\mathcal{Q}_b,  \mathrm{J}^s v \big{)}_{L^2}
            +
            \frac{\mu^2 \ve }{3}
            \big{(} \mathrm{F}^{\frac{1}{2}} \partial_x\big{(}[\mathrm{J}^s,h^3]\mathrm{F}^{\frac{1}{2}} \partial_x(\mathscr{T}^{-1}\mathcal{Q}_b\big{)},  \mathrm{J}^s v \big{)}_{L^2}
            \\
            & 
            \hspace{0.5cm}
            -
            \frac{ \mu^2 \ve }{2}
            \big{(} \mathrm{F}^{\frac{1}{2}} \partial_x\big{(}[\mathrm{J}^s,h^2\beta (\partial_x b)]\mathscr{T}^{-1}\mathcal{Q}_b\big{)},  \mathrm{J}^s v \big{)}_{L^2}
            \\
            & 
            \hspace{0.5cm}
            +
            \frac{\mu^2 \ve }{2}
            \big{(}[\mathrm{J}^s,h^2\beta (\partial_x b)]\mathrm{F}^{\frac{1}{2}} \partial_x(\mathscr{T}^{-1}\mathcal{Q}_b),  \mathrm{J}^s v \big{)}_{L^2}
            \\ 
            & 
            \hspace{0.5cm}
            -
             \mu^2 \ve  
            \big{(} [\mathrm{J}^s,h(\beta\partial_x b)^2]\mathscr{T}^{-1}\mathcal{Q}_b,  \mathrm{J}^s v \big{)}_{L^2}
            \\
            & =:
            VI_1^1 + VI_1^2 + VI_1^3 + VI_1^4 + VI_1^5.    
        \end{align*}
        Each term is treated similarly. For instance, take $VI_1^2$, which is the term with the least margin. Arguing as above, we use Cauchy-Schwarz inequality, \eqref{Commutator estimates}, \eqref{Inverse est F T} and \eqref{Bound on solution}  to deduce that
        \begin{align*}
            |VI_1^2| \lesssim \ve \mu|\mathcal{Q}_b|_{H^{s-1}} |v|_{X^s_{\mu}},
        \end{align*}
        where use the algebra property of $H^{s-1}(\R)$ for $s>\frac{3}{2}$ to get: 
        \begin{align*}
           \mu  |\mathcal{Q}_b|_{H^{s-1}}
            & 
            \lesssim 
                \mu |  h^2 (\partial_x \mathrm{F}^{\frac{1}{2}}  v)^2 (\beta \partial_x b) |_{H^{s-1}}
                +     
                \mu |\partial_x \mathrm{F}^{\frac{1}{2}}(h^2 v^2 \beta \partial_x^2 b) |_{H^{s-1}}
                +
                \mu |  h v^2(\beta \partial_x^2 b)(\beta \partial_x b)|_{H^{s-1}}
                \\
                &
                \lesssim |v|_{H^s}^2. 
        \end{align*}
        Using similar estimates for the remaining terms, it is easy to deduce that
        \begin{equation*}
            |VI_1| \lesssim \ve |v|_{X^s_{\mu}}^3.
        \end{equation*}
        For $IV_2$,  we use integration by parts to make the decomposition:
        \begin{align*}
            VI_2
            & =
             -
            \ve \mu
            \big{(}
            \mathrm{J}^s(h^2 (\partial_x \mathrm{F}^{\frac{1}{2}}  v)^2 (\beta \partial_x b) ), \mathrm{J}^sv
            \big{)}_{L^2}
            +
            \frac{\ve \mu}{2}
            \big{(}
            \mathrm{J}^s(h^2 v^2 \beta \partial_x^2 b)  ,\partial_x \mathrm{F}^{\frac{1}{2}} \mathrm{J}^sv
            \big{)}_{L^2}
            \\
            & \hspace{0.5cm}
            -
            \ve \mu
            \big{(}
            \mathrm{J}^s(hv^2(\beta \partial_x^2 b)(\beta \partial_x b)), \mathrm{J}^sv
            \big{)}_{L^2}.
        \end{align*}
        Each term is estimated by Hölder's inequality, Sobolev embedding, the algebra property of $H^s(\R)$, and \eqref{Equivalence energy vs norm}, leaving us with the estimate 
        \begin{align*}
            |VI_2|  
            \lesssim N(s)E_s(\mathbf{U}).
        \end{align*}
        Lastly, $VI_3$ is estimated using the same estimates and gives
        \begin{equation*}
            |VI_3| \lesssim \beta|\partial_x b |_{H^s}|v|_{H^s}|\zeta|_{H^s}.
        \end{equation*}
        Consequently, we have the estimate
        \begin{equation*}
            |VI| 
            \lesssim 
            N(s)E_s(\mathbf{U}),
        \end{equation*}
        and thus completes the proof of Proposition \ref{Four: Energy estimate s}.
        
        \end{proof}

        \begin{remark}\label{Quasilinear}
            Under the provision of Proposition \ref{Four: Energy estimate s}, using the algebra property of $H^{s-1}(\R)$ for $s>\frac{3}{2}$, \eqref{Inverse est T}, suitable commutator estimates one can easily obtain that
            \begin{equation}\label{N1}
                | (M_1+S^{-1}M_2)(\mathbf{U})\partial_x\mathbf{U}|_{H^{s-1}} \lesssim |\mathbf{U}|_{H^{s}},
            \end{equation}
            and 
            \begin{equation}\label{N2}
                |(S^{-1}(Q +Q_b)(\mathbf{U}) |_{H^{s-1}} \lesssim |\mathbf{U}|_{H^{s}}.
            \end{equation}
        \end{remark}

    \section{Estimates on the difference of two solutions} \label{Estimates on the diff}
 
    We will now estimate the difference between two solutions of \eqref{W-G-N} given by $\bold{U}_1= (\zeta_1, v_1 )^T$ and $\bold{U}_2 =\ve (\zeta_2, v_2)^T$. For   convenience,  we define  $(\eta,w) = (\zeta_1 - \zeta_2, v_1-v_2)$. Then $\bold{W} = (\eta,w)^T$ solves 
	\begin{equation}\label{lin W}
		\partial_t \bold{W} +  (M_1 + S^{-1}M_2)(\bold{U}_1)\partial_x \bold{W}= \bold{F}, 
	\end{equation}
	with $S,M_1,M_2,Q,Q_b$ defined as  in \eqref{Matrix W-G-N}       and
        \begin{align*}
            \bold{F} 
            & =
            -  \Big[(M_1+(S^{-1}M_2))(\bold{U}_1) - (M_1+S^{-1}M_2)(\bold{U}_2)  \Big] \partial_x\bold{U}_2
            \\
            & 
            \hspace{0.5cm}
       		- \Big{[}(S^{-1}(Q+Q_b))(\bold{U}_1) - (S^{-1}(Q+Q_b))(\bold{U}_2)  \Big]  
      		\\
        	& = :\bold{F}_1 + \bold{F}_2. 
        \end{align*}
	The energy associated to \eqref{lin W} is given in terms of the symmetrizer $S(\bold{U}_1)$ and reads
	\begin{equation}\label{tilde Energy s}
		\tilde{E}_{s}(\bold{W}) 
		: =
		\big(\mathrm{J}^s \bold{W}, S(\bold{U}_1)\mathrm{J}^s \bold{W}\big)_{L^2}. 
	\end{equation}
	
	The main result of this section reads:
	
	\begin{prop}\label{Compactness L2}

            Let $s> \frac{3}{2}$, $(\mu, \ve, \beta)\in \mathcal{A}_{SW}$, and $(\zeta_1,v_1),(\zeta_2,v_2) \in C([0,T]; Y^s_{\mu}(\mathbb{R}) )$ be a solution to \eqref{Matrix W-G-N} on a time interval $[0,T]$ for some $T>0 $.  Moreover, assume $  b \in H^{s+2}(\R) $ and there exist $h_0 \in (0,1)$  such that  \begin{equation*}\label{cond. sol. diff}
			h_0 - 1 + \beta b \leq \ve \zeta_i(x,t), \quad   \forall (x,t) \in \mathbb{R} \times [0,T],
		\end{equation*}
            for $i=1,2$, and suppose also that
            \begin{equation}\label{cond. sol. diff}
                N(s) : = \ve \sup\limits_{t\in [0,T]} |(\zeta_i(t, \cdot), v_i(t, \cdot))|_{Y^s_{\mu}}  + \beta|b|_{H^{s+2}}\leq N^{\star},
            \end{equation}
            for some $N^{\star}\in \R^{+}$. Define the difference to be $\bold{W}=(\eta,w) = (\zeta_1 - \zeta_2, v_1 - v_2)$. Then, for the energy defined by \eqref{tilde Energy s}, there holds
		\begin{equation}\label{Energy 1}
			\frac{d}{dt} \tilde{E}_0(\bold{W}) \lesssim N(s) |(\eta,w)|_{Y^{0}_{\mu}}^2,
		\end{equation}
		and
		\begin{equation}\label{equiv 1}
			|(\eta,w)|^2_{Y^0_{\mu}} \lesssim \tilde{E}_0(\bold{W}) \lesssim |(\eta,w)|^2_{Y^0_{\mu}}.
		\end{equation}
		
		Furthermore, we have the following estimate at the  $Y^s_{\mu}-$ level:
		\begin{equation}\label{Energy 2}
			\frac{d}{dt} \tilde{E}_s(\bold{W}) \lesssim  |\big(\mathrm{J}^s \bold{F}, S(\bold{U}_1) \mathrm{J}^s \bold{W}\big)_{L^2}| +  N(s)  |(\eta,w)|_{Y^s_{\mu}}^2,
		\end{equation}
		and
		\begin{equation}\label{equiv 2}
			|(\eta,w)|^2_{Y^s_{\mu}} \lesssim  \tilde{E}_s(\bold{W}) \lesssim |(\eta,w)|^2_{Y^s_{\mu}}.
		\end{equation}
	\end{prop}

 \begin{proof}
     We note that \eqref{equiv 1}, \eqref{Energy 2} and \eqref{equiv 2}  follow by the same arguments as in the proof of Proposition \ref{Four: Energy estimate s} and is therefore omitted. 

    To prove \eqref{Energy 1}, we use \eqref{lin W},  the self-adjointness $S(\mathbf{U})$, and Proposition \ref{Inverse of T} to obtain 
    \begin{align*}
        \frac{1}{2} \frac{d}{dt} \tilde{E}_0(\mathbf{W})
        & 
        =
        \frac{1}{2} 
        \big{(}
            \mathbf{W}, (\partial_t S(\mathbf{U}_1))\mathbf{W}
        \big{)}_{L^2}
         -
         \big{(}
        M_1(\mathbf{U}_1) \partial_x \mathbf{W}, S(\mathbf{U}_1)\mathbf{W}
        \big{)}_{L^2}  
        \\
        & 
        \hspace{0.5cm}
        -
         \big{(}
        M_2(\mathbf{U}_1) \partial_x \mathbf{W},  \mathbf{W}
        \big{)}_{L^2} 
        +
        \big{(} \mathbf{F}_1, S(\mathbf{U}_1) \mathbf{W} \big{)}_{L^2}
        +
         \big{(} \mathbf{F}_2, S(\mathbf{U}_1) \mathbf{W} \big{)}_{L^2}
        \\
        &
        =: \mathcal{I} + \mathcal{II} + \mathcal{III} + \mathcal{IV} + \mathcal{V}.
    \end{align*}

     \noindent
     \underline{Control of $\mathcal{I}$.} The estimate of $\mathcal{I}$ is a direct consequence of Hölder's inequality, \eqref{cond. sol. diff}, \eqref{L infty est dt zeta}, and \eqref{cond. sol. diff}:
     \begin{align*}
         |\mathcal{I}|
         \lesssim 
         N(s) |w|_{X^0_{\mu}}^2, 
     \end{align*}
     for $s>\frac{3}{2}$.\\

	\noindent
	\underline{Control of $\mathcal{II}$.}  By definition of $\mathcal{II}$, after performing an integration by parts, yields
	\begin{align*}
		\mathcal{II}
		& =
		\frac{\ve}{2} \big{(}( \partial_x(v_1h_1))w, w \big{)}_{L^2} 
		+
		\frac{\mu\ve}{3} \big{(} \partial_x\mathrm{F}^{\frac{1}{2}}(v_1\partial_xw), h_1^3 \partial_x\mathrm{F}^{\frac{1}{2}}w \big{)}_{L^2}
		\\
            & 
            \hspace{0.5cm}
                -
			\frac{\mu\ve\beta}{2} \big{(} v_1\partial_xw, \partial_x\mathrm{F}^{\frac{1}{2}}(h^2_1(\partial_x b) w) \big{)}_{L^2}
                 +
			\frac{\mu\ve\beta}{2} \big{(}  v_1\partial_xw, h^2_1 (\partial_x b)  \partial_x\mathrm{F}^{\frac{1}{2}}w \big{)}_{L^2}
			\\
                & 
                \hspace{0.5cm}
                +
                \frac{\mu\ve\beta}{2} \big{(} v_1\partial_xw, h_1 (\beta \partial_x b)^2w \big{)}_{L^2}
                \\		
                & = :
		      \mathcal{II}_1 + \mathcal{II}_2       
                +
                \mathcal{II}_3 + \mathcal{II}_4
                +
                \mathcal{II}_5.
	\end{align*}
	For $\mathcal{II}_1$ and $\mathcal{II}_5$, we simply use Hölders inequality and Sobolev embedding to obtain
	\begin{equation*}
		|\mathcal{II}_1| + |\mathcal{II}_5| \lesssim \ve (1+|h_1-1|_{H^s} + (1+|h_1-1|_{H^s})|b|^2_{H^{s+1}})|v_1|_{H^s}|w|_{L^2}^2.
	\end{equation*}
	For $\mathcal{II}_2$, we observe that is similar to $II_2^2$ in the proof of Proposition \ref{Four: Energy estimate s} where $w$ plays the role of $\mathrm{J}^sv$. Then reapplying the same estimates yields:
	\begin{equation*}
		|\mathcal{II}_2| \lesssim \ve (1+|h_1^3-1|_{H^s})|w|_{X^0_{\mu}}^3.
	\end{equation*}
        For $\mathcal{II}_3$, we integrate by parts to make the decomposition
        \begin{align*}
            \mathcal{II}_3
            & =
            \frac{\mu\ve\beta}{2} \big{(} \mathrm{F}^{\frac{1}{2}}((\partial_x v_1)\partial_xw), h^2_1(\partial_x b) w \big{)}_{L^2}
           +
           \frac{\mu\ve\beta}{2} \big{(} \mathrm{F}^{\frac{1}{2}}(v_1\partial_x^2w), h^2_1(\partial_x b) w \big{)}_{L^2}
           \\
           & = :
           \mathcal{II}_3^1 + \mathcal{II}_3^2.
        \end{align*}
        Here $\mathcal{II}_3^1$ is similar to $II_3^2$ in the proof of Proposition \ref{Four: Energy estimate s} and applying the estimates yields,
        \begin{align*}
            |\mathcal{II}_3^1|\lesssim \ve (1+ |h_1^2-1|_{H^s})|b|_{H^{s+1}}|v_1|_{H^s}|w|_{L^2}|w|_{X^0_{\mu}}.
        \end{align*}
        On the other hand, $\mathcal{II}_3^2$ is similar to $II_3^3$ and we observe that
        \begin{align*}
            \mathcal{II}_3^2
            & = 
            \frac{\mu\ve\beta}{2} \big{(} [\mathrm{F}^{\frac{1}{2}},v_1]\partial_x^2w, h^2_1(\partial_x b) w \big{)}_{L^2}
            -
            \frac{\mu\ve\beta}{2} \big{(} v_1\mathrm{F}^{\frac{1}{2}}\partial_xw, h^2_1(\partial_x b) \partial_x w \big{)}_{L^2}
            \\
            &
            \hspace{0.5cm}
            -
            \frac{\mu\ve\beta}{2} \big{(} v_1\mathrm{F}^{\frac{1}{2}}\partial_xw, \big{(}\partial_x(h^2_1(\partial_x b))\big{)}w \big{)}_{L^2}
            \\
            & =
            \mathcal{II}_3^{2,1} + \mathcal{II}_3^{2,2} + \mathcal{II}_3^{2,3}.
        \end{align*}
        Then we observe that $\mathcal{II}_3^{2,2} = - \mathcal{II}_4$, while for $\mathcal{II}_3^{2,1}$ and $\mathcal{II}_3^{2,3}$ we apply Hölder's inequality, \eqref{Commutator: F_half}, and Sobolev embedding to obtain the bound
        \begin{align*}
            |\mathcal{II}_3^{2,1}| + |\mathcal{II}_3^{2,3}| \lesssim \ve |v_1|_{H^s}(1+|h_1^2-1|_{H^s})|b|_{H^{s+1}}|w|_{L^2}|w|_{X^0_{\mu}}.
        \end{align*}
        Gathering these estimates and using \eqref{cond. sol. diff}  yields
        \begin{align*}
            |\mathcal{II}| \lesssim N(s)|w|_{X^0_{\mu}}^2.
        \end{align*}\\
        
	\noindent
	\underline{Control of $\mathcal{III}$.} By  definition of $\mathcal{III}$ we must estimate the terms:
	\begin{align*}
		\mathcal{III}
		& =
		-
		\ve \big{(}
		v_1 \partial_x \eta , \eta 
		\big{)}_{L^2}
		-
		\big{(}h_1 \partial_x w, \eta 
		\big{)}_{L^2}
		-
		\big{(}
		h_1 \partial_x \eta, w
		\big{)}_{L^2}
		\\
		& 
		= :
		\mathcal{III}_1 + \mathcal{III}_2 + \mathcal{III}_3.
	\end{align*}
	Starting with $\mathcal{III}_1$, we simply integrate by parts and use Hölder's inequality and Sobolev embedding to deduce
	\begin{align*}
		|\mathcal{III}_1| \lesssim \ve |v_1|_{H^s} |\eta|_{L^2}^2.
	\end{align*}
	Similarly, for $\mathcal{III}_2 + \mathcal{III}_3$ we use integration by parts, the Sobolev embedding, and \eqref{cond. sol. diff} to get that
	%
	%
	\begin{align*}
		|\mathcal{III}_2 + \mathcal{III}_3|& \lesssim  |\partial_xh_1|_{L^{\infty}}|w|_{L^2}|\eta|_{L^2}
		\\
		& \lesssim N(s)|w|_{L^2}|\eta|_{L^2}
		.
	\end{align*}
        In conclusion, we obtain the bound
        \begin{align*}
		|\mathcal{III}|\lesssim N(s)|(\eta,w)|_{Y^0_{\mu}}^2.
	\end{align*}\\

	\noindent
	\underline{Control of $\mathcal{IV}$.} First define the notation
	\begin{equation*}
		 \mathscr{T}_i = \mathscr{T}[h_i, \beta b],
	\end{equation*}
	for $i=1,2$ and consider the terms
	\begin{align*}
		\mathcal{IV}
		& = 
		-\ve \big{(} w \partial_x\zeta_2, \eta\big{)}_{L^2}
		-
		\ve \big{(} \eta \partial_xv_2, \eta\big{)}_{L^2}
		-
		 \big{(}(\mathscr{T}_1^{-1}(h_1\cdot) - \mathscr{T}_2^{-1}(h_2\cdot))\partial_x\zeta_2 , \mathscr{T}_1 w\big{)}_{L^2}
            \\
            & 
            \hspace{0.5cm}
            -
		\ve \big{(}w \partial_xv_2, \mathscr{T}_1 w\big{)}_{L^2}
		\\
		& = :
		\mathcal{IV}_1 +  \mathcal{IV}_2 +  \mathcal{IV}_3 + \mathcal{IV}_4.
	\end{align*}
	For the first two terms, we use Hölder's inequality and the Sobolev embedding to deduce the bound:
	\begin{align*}
		|\mathcal{IV}_1| + |\mathcal{IV}_2| \leq \ve | \zeta_2 |_{H^s} |w|_{L^2} |\eta |_{L^2} + \ve |v_2 |_{H^s} |\eta |_{L^2}^2,
	\end{align*}
	for $s>\frac{3}{2}$.  Next, we make the observation
	\begin{align}\label{id on diff op T}
		\mathscr{T}_1(\mathscr{T}_1^{-1} f_1 - \mathscr{T}_2^{-1}f_2)
		= (f_1 -f_2) - (\mathscr{T}_1-\mathscr{T}_2)\mathscr{T}_2^{-1}f_2.
	\end{align}
	Using \eqref{id on diff op T} and invertability of $\mathscr{T}_i$  we observe that
	\begin{align*}
		\mathcal{IV}_3 
		& =
		-
		\ve \big{(}\eta \partial_x \zeta_2, w\big{)}_{L^2} 
		+
		\frac{\mu \ve }{3}
		\big{(}\mathrm{F}^{\frac{1}{2}}\partial_x (\eta(h_1^2 + h_1h_2 + h_2^2)\partial_x \mathrm{F}^{\frac{1}{2}}\mathscr{T}_2^{-1}(h_2 \partial_x\zeta_2)), w\big{)}_{L^2}
		\\
		&
		\hspace{0.5cm}
		-
		\frac{\mu \ve }{2} 
		\big{(}
		\partial_x \mathrm{F}^{\frac{1}{2}} (\eta (h_1 + h_2) (\beta \partial_x b) \mathscr{T}_2^{-1}(h_2 \partial_x \zeta_2)), w
		\big{)}_{L^2}
		\\
		&
		\hspace{0.5cm}
		+
		\frac{\mu \ve }{2} 
		\big{(}
		\eta (h_1 + h_2) (\beta \partial_x b) \partial_x \mathrm{F}^{\frac{1}{2}}\mathscr{T}_2^{-1}(h_2 \partial_x \zeta_2), w
		\big{)}_{L^2}
		\\
		& 
		\hspace{0.5cm}
		-
		\mu \ve 
		\big{(}
		\eta  (\beta \partial_x b)^2\mathscr{T}_2^{-1}( h_2 \partial_x \zeta_2), w
		\big{)}_{L^2}
		\\
		& =:
		\mathcal{IV}_3^{1} + \mathcal{IV}_3^{2} + \mathcal{IV}_3^{3} + \mathcal{IV}_3^{4} + \mathcal{IV}_3^{5},
	\end{align*}
	where $\mathcal{IV}_3^{1} = \mathcal{IV}_1$ which is already treated. While for the second term, we use integration by parts, Hölder's inequality, Sobolev embedding, \eqref{cond. sol. diff}, and \eqref{Inverse est F T} to obtain
	\begin{align*}
		|\mathcal{IV}_3^{2}|
		&
		\leq 
		\ve | \eta |_{L^2} |(h_1^2 + h_1h_2 + h_2^2)|_{L^{\infty}} | \partial_x \mathrm{F}^{\frac{1}{2}}\mathscr{T}_2^{-1}(h_2 \partial_x\zeta_2))|_{L^{\infty}} |w|_{X^0_{\mu}}
		\\
		& 
		\lesssim 
		N(s) | \eta |_{L^2}|w|_{X^0_{\mu}}
	\end{align*}
	for $s> \frac{3}{2}$.  For $II_3^{3}$ we apply the same estimates together with \eqref{Inverse est T} to deduce
	\begin{align*}
		|\mathcal{IV}_3^{3}|
		& \lesssim 
		\ve |(h_1 + h_2) |_{L^{\infty}} |\beta \partial_x b |_{L^{\infty}} |h_2|_{L^{\infty}} | \partial_x \zeta_2 |_{L^{\infty}} |\eta|_{L^2} |w |_{X^0_{\mu}}
		\\
		& \lesssim
		N(s) | \eta |_{L^2} |w|_{X^0_{\mu}}.
	\end{align*}
	\noindent
	Next, we see that $\mathcal{IV}_3^{4}$ is estimated similarly to $\mathcal{IV}_3^{2}$ and we get that
	\begin{align*}
		|\mathcal{IV}_3^{4}|
		& 
		\lesssim N(s)|\eta|_{L^2} |w|_{X^0_{\mu}}.
	\end{align*}
	The  part $\mathcal{IV}_3^{5}$ is easily treated with Hölder's inequality and Sobolev embedding. Thus, gathering these estimates and applying \eqref{equiv 1} yields,
	\begin{align*}
		|\mathcal{IV}_3| \lesssim N(s) |(\eta,w)|_{Y^0_{\mu}}^2.
	\end{align*}
        Lastly, we deal with $\mathcal{IV}_4$:
        \begin{align*}
            \mathcal{IV}_4
            & =
            -
            \ve \big{(}w \partial_xv_2, h_1 w\big{)}_{L^2}
            +
            \frac{\mu \ve }{3}
            \big{(}w \partial_xv_2, \partial_x \mathrm{F}^{\frac{1}{2}}(h_1^3 \mathrm{F}^{\frac{1}{2}} \partial_x  w)\big{)}_{L^2}
            \\
            & 
            \hspace{0.5cm}
            -
            \frac{\mu \ve }{2}
            \big{(}w \partial_xv_2, \partial_x \mathrm{F}^{\frac{1}{2}}(h_1^2 (\beta \partial_x b) w)\big{)}_{L^2}
            +
            \frac{\mu \ve }{2}
            \big{(}w \partial_xv_2, h_1^2 (\beta \partial_x b) \partial_x \mathrm{F}^{\frac{1}{2}}w\big{)}_{L^2}
            \\
            &
            \hspace{0.5cm}
            -
            \mu \ve
            \big{(}w \partial_xv_2, h_1(\beta \partial_x b)^2w\big{)}_{L^2}
            \\
            & = :
            \mathcal{IV}_4^1 + \mathcal{IV}_4^2 + \mathcal{IV}_4^3 + \mathcal{IV}_4^4. 
        \end{align*}
        Each term is treated similarly, and we only give the details for $\mathcal{IV}_4^2$ since it is the term with the least margin. In particular, using integration by parts,  Hölder's inequality, Sobolev embedding, \eqref{cond. sol. diff}, 
        \begin{align*}
            |\mathcal{IV}_4^2|
            & \lesssim \ve \sqrt{\mu} (1+ |h_1^3 - 1|_{H^s})|\partial_x \mathrm{F}^{\frac{1}{2}}(w \partial_x v_2)|_{L^2} |w|_{X^{0}_{\mu}}
            \\
            & 
            \lesssim  N(s)\mu^{\frac{1}{4}}|w \partial_x v_2|_{H^{\frac{1}{2}}}|w|_{X^{0}_{\mu}}.
        \end{align*}
        Then we use Hölder's inequality, Sobolev embedding, and \eqref{D_half - fracLeib - Sob} to deduce that
        \begin{align*}
            \mu^{\frac{1}{4}}|w \partial_x v_2|_{H^{\frac{1}{2}}}
            & \lesssim
            \mu^{\frac{1}{4}}(|\partial_x v_2|_{L^{\infty}}|w|_{L^2} +|D^{\frac{1}{2}}(w\partial_x v_2)|_{L^2} )
            \\
            & 
            \lesssim |v_2|_{H^s}|w|_{L^2} + \mu^{\frac{1}{4}}|v_2|_{H^{r+1}}|w|_{H^{\frac{1}{2}}},
        \end{align*}
	for any $r>\frac{1}{2}$. Now choose $r$ such that $s>r+1>\frac{3}{2}$ allowing us to conclude that
        \begin{align*}
            \mu^{\frac{1}{4}}|w \partial_x v_2|_{H^{\frac{1}{2}}} \lesssim |v_2|_{H^s}|w|_{X^0_{\mu}},
        \end{align*}
        and from which we obtain:
        \begin{align*}
            |\mathcal{IV}_4^2| \lesssim N(s)|w|_{X^0_{\mu}}^2.
        \end{align*}
        To summarize this part, we can use \eqref{cond. sol. diff} to obtain the estimate
        \begin{align*}
            |\mathcal{IV}| \lesssim  N(s)|(\eta,w)|_{Y^0_{\mu}}^2.
        \end{align*}\\

	\noindent
	\underline{Control of $\mathcal{V}$.} Define the notation 
	\begin{equation*}
		\mathcal{Q}_i = \mathcal{Q}[h_i, v_i], \quad \mathcal{Q}_{b,i}= \mathcal{Q}_b[h_i, \beta b, v_i],
	\end{equation*}
	with $i = 1,2$, and using the identity \eqref{id on diff op T}, then  we obtain the following terms:
	\begin{align*}
		\mathcal{V} 
		& =
		\beta \big{(} \partial_xb w, \eta \big{)}_{L^2} 
		-
		\mu \ve \big{(} h_1\mathcal{Q}_1 - h_2 \mathcal{Q}_2  , w \big{)}_{L^2 } 
		+
		\mu \ve \big{(}  (\mathscr{T}_1-\mathscr{T}_2)\mathscr{T}_2^{-1}(h_2\mathcal{Q}_2), w \big{)}_{L^2}
		\\
		& 
		\hspace{0.5cm}
		-
		\mu \ve \big{(} h_1\mathcal{Q}_{b,1} - h_2 \mathcal{Q}_{b,2}  , w \big{)}_{L^2 } 
            +
		\mu \ve \big{(}  (\mathscr{T}_1-\mathscr{T}_2)\mathscr{T}_2^{-1}(h_2\mathcal{Q}_{b,2}), w \big{)}_{L^2}
		\\
		& 
		= :
		\mathcal{V}_1 + \mathcal{V}_2 + \mathcal{V}_3 + \mathcal{V}_4 + \mathcal{V}_5.   
	\end{align*}
	The estimate of $\mathcal{V} _1$ follows directly by Hölder's inequality:
	\begin{equation*}
		|\mathcal{V}_1| \lesssim \beta |\partial_xb|_{L^{\infty}}|w|_{L^2}|\eta|_{L^2}.
	\end{equation*}
	For $\mathcal{V}_2$, we use the definition of $Q_i$ and then integration by parts to make the following decomposition
	\begin{align*}
		\mathcal{V}_2 
		& = 
		\frac{2\mu \ve }{3}
		\big{(} \big{(}h_1^3(\partial_x\mathrm{F}^{\frac{1}{2}} v_1)^2 - h_2^3(\partial_x\mathrm{F}^{\frac{1}{2}} v_2)^2\big{)} , \partial_x \mathrm{F}^{\frac{1}{2}}w \big{)}_{L^2 }  
		\\
		& = 
		\frac{2\mu \ve }{3}
		\big{(} \big{(}\eta(h_1+h_2) \big{)}(\partial_x\mathrm{F}^{\frac{1}{2}} v_1)^2 , \partial_x \mathrm{F}^{\frac{1}{2}}w \big{)}_{L^2 } 
		\\
		& 
		\hspace{0.5cm} 
		+
		\frac{2\mu \ve }{3}
		\big{(} h_2^3\big{(}\partial_x\mathrm{F}^{\frac{1}{2}} (v_1+v_2)\big{)}(\partial_x\mathrm{F}^{\frac{1}{2}} w,)  , \partial_x \mathrm{F}^{\frac{1}{2}}w \big{)}_{L^2 }.
	\end{align*} 
	Now, estimate each term by Hölder's inequality and Sobolev embedding to obtain that
	\begin{align*}
		|\mathcal{V}_2 |
		& 
		\lesssim \mu \ve\big{(} |h_1+h_2|_{L^{\infty}} |\partial_x\mathrm{F}^{\frac{1}{2}} v_1 |_{L^{\infty}}^2 |\eta|_{L^2} |\partial_x \mathrm{F}^{\frac{1}{2}}w |_{L^2}
		+
		|h_2^3|_{L^{\infty}}|\partial_x\mathrm{F}^{\frac{1}{2}} (v_1+v_2)|_{L^{\infty}} |\partial_x \mathrm{F}^{\frac{1}{2}}w|^2_{L^2} \big{)}
		\\
		& 
		\lesssim 
		\ve \max_{i = 1,2}    
		\big{(}
		(1+|h_i-1|_{{H^s}} )| v_1|_{H^s}^2
		 |\eta|_{L^2} |w |_{X^0_{\mu}}
		+
		(1+|h_2^3-1|_{{H^s}} )| v_i |_{H^s}
		 |w|^2_{X^0_{\mu}} \big{)}.
	\end{align*}
	Then conclude this estimate by applying \eqref{cond. sol. diff}: 
	\begin{equation*}
		|\mathcal{V}_2 | \lesssim N(s) |(\eta,w)|_{Y^0_{\mu}}^2.
	\end{equation*}
	For $\mathcal{V}_3$, we use the same decomposition as for $\mathcal{IV}_3$ and find that
	\begin{align*}
		\mathcal{V}_3
		& = 
		\frac{\mu^2 \ve  }{3}
		\big{(}\mathrm{F}^{\frac{1}{2}}\partial_x \big{(}\eta(h_1^2 + h_1h_2 + h_2^2)\partial_x \mathrm{F}^{\frac{1}{2}}\mathscr{T}_2^{-1}(h_2\mathcal{Q}_2)\big{)}, w\big{)}_{L^2}
		\\
		&
		\hspace{0.5cm}
		-
		\frac{\mu^2 \ve  }{2} 
		\big{(}
		\partial_x \mathrm{F}^{\frac{1}{2}} \big{(}\eta (h_1 + h_2) (\beta \partial_x b) \mathscr{T}_2^{-1}(h_2\mathcal{Q}_2)\big{)}, w
		\big{)}_{L^2}
		\\
		& 
		\hspace{0.5cm}
		+
		\frac{\mu^2 \ve  }{2} 
		\big{(}
		\eta (h_1 + h_2) (\beta \partial_x b) \partial_x \mathrm{F}^{\frac{1}{2}}\mathscr{T}_2^{-1}(h_2\mathcal{Q}_2), w
		\big{)}_{L^2}
		-
		\mu^2 \ve 
		\big{(}
		\eta  (\beta \partial_x b)^2 \mathscr{T}_2^{-1}(h_2\mathcal{Q}_2), w
		\big{)}_{L^2}
		\\
		& = :
		\mathcal{V}_3^1 + \mathcal{V}_3^2 + \mathcal{V}_3^3 + \mathcal{V}_3^4.
	\end{align*}
	Each term is treated similarly, but the term with the least margin is $\mathcal{V}_3^1$. In fact, we use integration by parts, Hölder's inequality, the Sobolev embedding $H^{s-1}(\R) \hookrightarrow L^{\infty}(\R)$, \eqref{cond. sol. diff}, and \eqref{Inverse est F T} to get the following estimate
	\begin{align*}
		|\mathcal{V}_3^1|
		& \lesssim 
		\mu^2 \ve   
		|\eta|_{L^2}|h_1^2 + h_1h_2 + h_2^2|_{L^{\infty}}| \mathrm{F}^{\frac{1}{2}}\mathscr{T}_2^{-1}(h_2\mathcal{Q}_2)|_{H^{s}} |\mathrm{F}^{\frac{1}{2}}\partial_x w|_{L^{2}}
		\\
		& 
		\lesssim  \mu \ve  | h_2\mathcal{Q}_2|_{H^{s-1}}	|\eta|_{L^2}|w|_{X^0_{\mu}}.
	\end{align*}
	Then using \eqref{Kill half derivative} and the algebra property of $H^{s}(\R)$ and the boundedness of $\mathrm{F}^{\frac{1}{2}}$, we observe that
	\begin{align*}
		| h_2\mathcal{Q}_2|_{H^{s-1}}& \lesssim  | \mathrm{F}^{\frac{1}{2}} \partial_x (h_2^3(\mathrm{F}^{\frac{1}{2}} \partial_x v_2)^2)|_{H^{s-1}} \lesssim (1+|h_2^3-1|_{H^s})|\mathrm{F}^{\frac{1}{2}} \partial_x v_2|_{H^{s}}^2.
	\end{align*}
	Consequently, we  may gather these estimates to deduce the bound
	\begin{align*}
		|\mathcal{V}_3^1| +...+ |\mathcal{V}_3^4| \lesssim N(s)|(\eta,w)|_{X^0_{\mu}}^2,
	\end{align*}
	where $	|\mathcal{V}_3^2| +...+ |\mathcal{V}_3^4|$ are easier versions of  $\mathcal{V}_3^2$. 
	
	To conclude we must estimate $\mathcal{V}_4$ and $\mathcal{V}_5$. However, since $\mathcal{Q}_b$ contains fewer derivatives than $\mathcal{Q}$, these terms could be considered to be of lower order. In fact, $\mathcal{V}_4$ is estimated by a similar decomposition to the one of $\mathcal{V}_2$, while $\mathcal{V}_5$ is a just a simpler version of $\mathcal{V}_3$. We may therefore conclude that
	\begin{equation*}
			|\mathcal{V}| \lesssim N(s) |(\eta,w)|_{Y^0_{\mu}}^2.
	\end{equation*}
	 Gathering all these estimates, we obtain \eqref{Energy 1}, and the proof of Proposition \ref{Compactness L2} is complete.
 \end{proof}

\begin{remark}\label{divergence est}
    From the proof of the proposition, it is easy to make the rough estimate of the source term in \eqref{Energy 2}:
    \begin{equation*}
        |\big(\mathrm{J}^s \bold{F}, S(\bold{U}_1) \mathrm{J}^s \bold{W}\big)_{L^2}|\lesssim N(s+1)\: |(\eta,w)|_{Y^{s-1}_{\mu}} \Big{(}\tilde{E}_s(\mathbf{W}))^{\frac{1}{2}}
        + N(s) \tilde{E}_s(\mathbf{W}) \Big{)},
    \end{equation*}
    combining the estimates used below (see control of $II$) and using the product estimate for $H^s(\R)$. The estimate \eqref{Energy 2} serves two purposes. One is to prove the full justification of \eqref{W-G-N} as a water waves model, where we allow for a loss of derivatives (see Section \ref{Justification}).

    On the other hand, to get the continuity of the flow, one needs to compensate the norms on the right of \eqref{Energy 2}: 
     \begin{align*}
        \max \limits_{i = 1,2} |(\zeta_i,v_i)|_{Y^{s+1}_{\mu}}  |(\eta,w)|_{Y^{s-1}_{\mu}},
    \end{align*}
    and is done by regularising the initial data and a Bona-Smith argument \cite{BonaSmith75}. 
\end{remark}

    \section{Long time Well-posedness of \eqref{W-G-N}}\label{Proof - WP}

 	For the proof of Theorem \ref{W-P W-G-N} we will use the parabolic regularisation method for the existence of solutions and a Bona–Smith regularisation 
	argument \cite{BonaSmith75} to prove the continuous dependence of the solutions with respect to the initial data. This method is classical in the case of quasilinear equations and we will only outline the steps that are unique to system \eqref{W-G-N} and needed to run the argument. In particular, one can read \cite{DucheneMMWW21} for a similar argument in the case of the classical Green-Naghdi system. Lastly, the reader might also find it useful to read the  detailed proof, using these methods, in the case of the Benjamin-Ono equation in \cite{LinaresPonce15}, and likewise  in the case of  Whitham-Boussinesq systems demonstrated in \cite{Paulsen22}. \\


	\begin{proof}
    \noindent
    \underline{Step 1:} \textit{Existence of solutions for a regularised system.} Let $s>\frac{3}{2}$, $\alpha \in (1,\frac{3}{2}]$ and take $\nu>0$ small. Moreover let $\mathbf{U}_0  = (\zeta_0,v_0)^T \in Y^s_{\mu}(\R)$, $ b \in H^{s+2}(\R)$ satisfying \eqref{nonCavitation} and define $T_\nu >0$ such that
    \begin{align}\label{Time condition 1}
        T_{\nu} \searrow 0\quad \text{as} \quad  \nu \searrow 0, \quad \text{and} \quad  T_{\nu} = T_{\nu}(|(\zeta_0,v_0)|_{Y^{s}_{\mu}}),
    \end{align}
    with the property that  
    \begin{equation}\label{Time condition 2}
        \text{if} \quad a< b \quad \text{then} \quad T_{\nu}(a) > T_{\nu}(b).
    \end{equation}
    Then we claim there is a unique solution $\mathbf{U}^{\nu}  = (\zeta^{\nu},v^{\nu})^T \in C([0,T_{\nu}];Y^s_{\mu}(\R))$ associated to $\mathbf{U}_0$ that satisfy the regularised version of \eqref{Matrix W-G-N} given by,
    \begin{align}\label{system parabolic reg}
        S(\mathbf{U}^{\nu})(\partial_t \mathbf{U}^{\nu} +M_1(\mathbf{U}^{\nu}) \partial_x \mathbf{U}^{\nu})
        +
        M_2(\mathbf{U}^{\nu}) \partial_x\mathbf{U}^{\nu} 
        + 
        Q(\mathbf{U}^{\nu})+Q_b(\mathbf{U}^{\nu})
        =
        -\nu S(\mathbf{U}^{\nu}) \mathrm{J}^{\alpha}\mathbf{U}^{\nu}.
    \end{align}

    To prove the claim, we first suppose the non-cavitation condition for $\mathbf{U}^{\nu}$ and use Proposition \ref{Inverse of T} to apply the inverse of $\mathscr{T}[h,\beta b]$ on the second equation in \eqref{system parabolic reg}.  Then we study the Duhamel formulation:
    \begin{align*}
        \mathbf{U}^{\nu}(t) = \mathrm{e}^{-\nu \langle \mathrm{D} \rangle^{\alpha}t}\mathbf{U}_0 + \int_0^t \mathrm{e}^{-\nu \langle \mathrm{D} \rangle^{\alpha}(t-s)} \mathcal{N}(\mathbf{U}^{\nu})(s)\: \mathrm{d}s,
    \end{align*}
    where $\mathrm{e}^{-\nu \langle \mathrm{D} \rangle^{\alpha}t}$ is the Fourier multiplier defined by
    \begin{align*}
        \mathcal{F}(\mathrm{e}^{-\nu \langle \mathrm{D} \rangle^{\alpha}t}f)(\xi) = e^{-\nu \langle \xi \rangle^{\alpha}t}\hat{f}(\xi),
    \end{align*}
    and with
    \begin{align*}
        \mathcal{N}(\mathbf{U}^{\nu}) =
         (M_1 + S^{-1} M_2)(\mathbf{U}^{\nu}) \partial_x\mathbf{U}^{\nu} 
           + 
            (S^{-1}(Q+Q_b))(\mathbf{U}^{\nu}).
    \end{align*}
    In particular, we prove that the application
    \begin{align}\label{Map fixed pt}
        \Phi : \mathbf{U}^{\nu} \mapsto \mathrm{e}^{-\nu \langle \mathrm{D} \rangle^{\alpha}t}\mathbf{U}_0 + \int_0^t \mathrm{e}^{-\nu \langle \mathrm{D} \rangle^{\alpha}(t-s)} \mathcal{N}(\mathbf{U}^{\nu})(s)\: \mathrm{d}s,
    \end{align}
    is a contraction map on the subspace 
    \begin{align*}
        B(R,h_0)  = \Big{\{}\mathbf{U} = (\zeta, v) \in C([0,T];Y^{s}_{\mu}(\R)) \: : \: |(\zeta,v)|_{Y^{s}_{\mu}}<R,\: \: \inf \limits_{t \in (0,T)}(1+\ve \zeta^{\nu}  -\beta b) \geq h_0\Big{\}}, 
    \end{align*}
    with $R>0$ to be determined. First, observe by Plancherel's identity and then splitting in high and low frequencies that
    \begin{align*}
        |\mathrm{e}^{-\nu \langle \mathrm{D} \rangle^{\alpha}t} \mathbf{U} |_{H^s} 
        & 
        \lesssim  |\mathbf{U}|_{L^2} + (\nu t)^{-\frac{1}{\alpha}}| (\nu t)^{\frac{1}{\alpha}} |\xi |e^{-((\nu t)^{\frac{1}{\alpha}}|\xi |)^{\alpha}} \hat{\mathbf{U}}|_{H^{s-1}}
        \\
        &
        \lesssim (1+(\nu t)^{-\frac{1}{\alpha}} ) |\mathbf{U}|_{H^{s-1}},
    \end{align*}
    and trivially that
    \begin{align*}
        |\mathrm{e}^{-\nu \langle \mathrm{D} \rangle^{\alpha}t} \mathbf{U} |_{H^s} \leq  |\mathbf{U}|_{H^{s}}.
    \end{align*}
    Thus, as a consequence of these estimates and Remark \ref{Quasilinear}  we obtain that
    \begin{align*}
       \sup\limits_{t \in [0,T]} | \Phi(\mathbf{U^{\nu}})(t) |_{H^s} \leq c |\mathbf{U}_0|_{H^s} + cT^{1-\frac{1}{\alpha}}\nu^{-\frac{1}{\alpha}}|\mathbf{U}|_{H^s}.
    \end{align*}
    Now, choose $R$ to be
    \begin{align*}
        R = 2c |(\zeta_0,v_0)|_{Y^s_{\mu}}.
    \end{align*}
    Additionally, since $1-\frac{1}{\alpha}>0$ we may take $T$ positive depending on $\nu$  and $R$ on the form
    \begin{align*}
        T^{1-\frac{1}{\alpha}} \sim  \frac{\nu^{\frac{1}{\alpha}}}{R},
    \end{align*}
    small enough, and such that
    \begin{align*}
        1+\ve \zeta^{\nu}(x,t) - \beta b(x) & = h_0 + \int_{0}^t \partial_t\zeta^{\nu} (x,s) \: \mathrm{d}s \geq 
        h_0 - cT(R+R^2)\geq \frac{h_0}{2},
    \end{align*}
    using the Fundamental theorem of calculus and \eqref{L infty est dt zeta}. Then the map \eqref{Map fixed pt} is  well-defined on $B(R,\frac{h_0}{2})$, and the contraction estimate is obtained similarly after some straightforward algebraic manipulations. We may therefore conclude this step by the Banach fixed point Theorem. \\

    \begin{remark}[The blow-up alternative] 
    If we define the maximal time of existence $T_{\text{Max}}^{\nu}$ to be 
    \begin{align*}
        T_{\text{Max}}^{\nu} = \sup \big{\{}T_{\nu}>0 \: : \: \exists ! \quad \mathbf{U}^{\nu} \quad \text{solution of} \quad \eqref{system parabolic reg} \quad in \quad C([0,T_{\nu}]; Y^s_{\mu}(\R)) \big{\}},
    \end{align*}
    then by a standard contradiction argument, one can deduce that
    \begin{equation}\label{Blow-up alt}
        \text{if} \quad T^{\nu}_{\text{Max}} < \infty, \quad \text{then} \quad \lim\limits_{t \nearrow T^{\nu}_{\text{Max}}} |(\zeta^{\nu}, v^{\nu})|_{Y^{s}_{\mu}} = \infty \quad \text{or} \quad  \lim\limits_{t \nearrow T^{\nu}_{\text{Max}}}\inf\limits_{x \in \R}1+\ve \zeta^{\nu} + \beta b = 0.
    \end{equation}
    This is due to the fact that if \eqref{Blow-up alt} does not hold, one can use Step $1$. and the properties of $T^{\nu}$ given by \eqref{Time condition 1} and \eqref{Time condition 2} to extend the solution beyond the maximal time.\\
    \end{remark}

    \noindent
    \underline{Step 2:} \textit{The existence time is independent of $\nu>0$.}  
    Let $ s > \frac{3}{2}$ and $ (\zeta^{\nu}, v^{\nu}) \in C([0,T_{\text{Max}}^{\nu}); Y_{\mu}^s(\mathbb{R}))$ be a solution of \eqref{system parabolic reg} with initial data $(\zeta_0, v_0) \in Y^s_{\mu}(\mathbb{R})$,  defined on its maximal time of existence and satisfying the blow-up alternative \eqref{Blow-up alt}. Moreover, let $\zeta_0$ satisfy \eqref{nonCavitation}. Then for $\tilde{N} = |(\zeta_0, v_0)|_{Y^s_{\mu}} + | b|_{H^{s+2}}$, there  exist a time 
			\begin{equation}\label{time T0}
				T = \frac{1}{\tilde{N}},
			\end{equation}
			such that $T<T_{\text{Max}}^{\nu}$ and
			\begin{equation}\label{Bound on solution - data}
				\sup\limits_{t \in [0,\frac{T}{\max\{\ve, \beta\}}]} | (\zeta^{\nu}, v^{\nu})(t)|_{Y^s_{\mu}} \lesssim |  (\zeta_0,  v_0) |_{Y^s_{\mu}}.
			\end{equation}

                Indeed, if the solution of \eqref{system parabolic reg} also satisfies estimate \eqref{Energy estimate 3/2}, then  one could combine this estimate with \eqref{Blow-up alt} and a bootstrap argument to get the result. However, to obtain the same estimate for \eqref{system parabolic reg}, one has to take into account an additional term:
                \begin{align*}
                    \frac{d}{dt} E_s(\mathbf{U}^{\nu}) \lesssim  N(s)E_s(\mathbf{U}^{\nu})
                    -
                    \nu \big{(}\mathrm{J}^{s+\alpha}\mathbf{U}^{\nu},S(\mathbf{U}^{\nu})\mathrm{J}^s \mathbf{U}^{\nu}\big{)}_{L^2},
                \end{align*}
                appearing due to the regularisation. To control this additional term, we make the decomposition
                \begin{align*}
                    \big{(}\mathrm{J}^{s+\alpha}\mathbf{U}^{\nu},S(\mathbf{U}^{\nu})\mathrm{J}^s \mathbf{U}^{\nu}\big{)}_{L^2}
                    & = 
                    |\zeta^{\nu}|_{H^{s+\frac{\alpha}{2}}}^2 
                    +
                    \big{(}  \mathscr{T}\mathrm{J}^{s+\frac{\alpha}{2}}v^{\nu}, \mathrm{J}^{s+\frac{\alpha}{2}}v^{\nu} \big{)}_{L^2}
                    +
                    \big{(} [\mathrm{J}^{\frac{\alpha}{2}},\mathscr{T}]\mathrm{J}^sv^{\nu}, \mathrm{J}^{s+\frac{\alpha}{2}}v^{\nu} \big{)}_{L^2}
                    \\
                    & = I_1 + I_2 + I_3.
                \end{align*}
                Then the two first terms will have a positive sign, where
                \begin{align*}
                    I_2\geq c(h_0) |v^{\nu}|_{X^{s+\frac{\alpha}{2}}}^2,
                \end{align*}
                arguing as we did in the proof of Proposition \ref{Inverse of T}, step $2$. On the other hand, $I_3$ is further decomposed by using integration by parts:
                \begin{align*}
                    I_3 
                    & = 
                    -\big{(} [\mathrm{J}^{\frac{\alpha}{2}},h^{\nu}]\mathrm{J}^sv^{\nu}, \mathrm{J}^{s+\frac{\alpha}{2}}v^{\nu} \big{)}_{L^2}
                    -
                    \frac{\mu}{3}
                    \big{(} [\mathrm{J}^{\frac{\alpha}{2}},(h^{\nu})^3]\mathrm{J}^s\mathrm{F}^{\frac{1}{2}}\partial_x v^{\nu}, \mathrm{J}^{s+\frac{\alpha}{2}}\mathrm{F}^{\frac{1}{2}}\partial_x v^{\nu} \big{)}_{L^2}
                    \\
                    & 
                    \hspace{0.5cm}
                    -
                    \frac{\mu}{2}
                    \big{(} [\mathrm{J}^{\frac{\alpha}{2}},(h^{\nu})^2 (\beta \partial_x b)]\mathrm{J}^s v^{\nu}, \mathrm{J}^{s+\frac{\alpha}{2}}\mathrm{F}^{\frac{1}{2}}\partial_x v^{\nu} \big{)}_{L^2}
                    +
                    \frac{\mu}{2}
                    \big{(} [\mathrm{J}^{\frac{\alpha}{2}},(h^{\nu})^2 (\beta \partial_x b)] \mathrm{J}^sv^{\nu}, \mathrm{J}^{s+\frac{\alpha}{2}} v^{\nu} \big{)}_{L^2}
                    \\
                    & 
                    \hspace{0.5cm}
                    +
                    \mu 
                    \big{(} [\mathrm{J}^{\frac{\alpha}{2}},h^{\nu}(\beta \partial_x b)^2] \mathrm{J}^s\mathrm{F}^{\frac{1}{2}}\partial_xv^{\nu}, \mathrm{J}^{s+\frac{\alpha}{2}} v^{\nu} \big{)}_{L^2}.
                \end{align*}
                We recall that $\alpha\in (1,\frac{3}{2}]$. We may therefore estimate each term by Hölder's inequality, \eqref{Commutator estimates}, Sobolev embedding, and then use Young's inequality to deduce that
                \begin{align*}
                    |I_3| 
                    & 
                    \leq 
                    N(s)|v^{\nu}|_{X_{\mu}^s}|v^{\nu}|_{X^{s+\frac{\alpha}{2}}_{\mu}}
                    \\
                    & 
                    \leq 
                    \frac{N(s)}{c_1}|v^{\nu}|_{X_{\mu}^s}^2
                    +
                    c_1N(s)|v^{\nu}|_{X^{s+\frac{\alpha}{2}}_{\mu}}^2,
                \end{align*}
                for $c_1>0$ small enough such that
                \begin{align*}
                    -\nu \big{(}\mathrm{J}^{s+\alpha}\mathbf{U}^{\nu},\mathrm{J}^s \mathbf{U}^{\nu}\big{)}_{L^2} 
                    & =-\nu( I_1 + I_2 + I_3)
                    \lesssim N(s)|v^{\nu}|_{X_{\mu}^s}^2,
                \end{align*}
                and by extension, we obtain that
                \begin{align*}
                    \frac{d}{dt} E_s(\mathbf{U}^{\nu}) \lesssim  N(s)E_s(\mathbf{U}^{\nu}),
                \end{align*}
                allowing us to conclude this step.\\

\begin{remark} \label{Energy for parabolic s=0}
    Since $\frac{\alpha}{2} \in (\frac{1}{2}, \frac{3}{4})$, one can obtain a similar estimate on $|I_3|$ in the case $\mathrm{J}^s = \mathrm{Id}$. Indeed, there holds
    \begin{align*}
        |I_3| 
                    & 
                    \leq 
                    N(r)|v^{\nu}|_{X_{\mu}^0}|v^{\nu}|_{X^{\frac{\alpha}{2}}_{\mu}},
    \end{align*}
    for $r>\frac{3}{2}$.
    \\
\end{remark}

    \noindent
    \underline{Step 3:} \textit{Existence of solutions.} We claim that for all $0\leq s'<s$ there exists a solution $(\zeta,v) \in C([0,\frac{T}{\max\{\ve, \beta\}}]; Y_{\mu}^{s'}(\mathbb{R})) \cap L^{\infty}([0,\frac{T}{\max\{\ve, \beta\}}]; Y_{\mu}^s(\mathbb{R}))$ of \eqref{W-G-N} with $T$ defined by  \eqref{time T0}. \\

    To prove the claim, we let $0<\nu'<\nu<1$ where we take $ (\zeta^{\nu'}, v^{\nu'}),  (\zeta^{\nu}, v^{\nu})$ to be two sets of solutions to system \eqref{system parabolic reg}, obtained in Step $1$, and with the same initial data. Then define the difference to be
    $$\bold{W} = (\eta, w): = (\zeta^{\nu'} - \zeta^{\nu}, v^{\nu'}- v^{\nu}),$$ 
    with $\alpha \in (1,\frac{3}{2}]$. Observe that $(\eta, w)$ satisfies a regularised version of \eqref{lin W}:
    \begin{equation*}
			\partial_t \bold{W}+  (S^{-1}M)(\bold{U}^{\nu'}) \bold{W} = \bold{F} -\nu' \mathrm{J}^{\alpha}\mathbf{W} + (\nu - \nu')\mathrm{J}^{\alpha}\mathbf{U}^{\nu},
    \end{equation*}
    where $\mathbf{F}$ is defined by
     \begin{align*}
            \bold{F} 
            & =
            -  \Big[(M_1+(S^{-1}M_2))(\bold{U}^{\nu'}) - (M_1+S^{-1}M_2)(\bold{U}^{\nu})  \Big] \partial_x\bold{U}^{\nu}
            \\
            & 
            \hspace{0.5cm}
       		- \Big{[}(S^{-1}(Q+Q_b))(\bold{U}^{\nu'}) - (S^{-1}(Q+Q_b))(\bold{U}^{\nu})  \Big]. 
    \end{align*}
    Now, we can easily extend the estimates in Proposition \ref{Compactness L2} and use Remark \ref{Energy for parabolic s=0} to deduce the estimate
    \begin{equation*}
			\frac{d}{dt} \tilde{E}_0(\mathbf{W}) 
			 \lesssim N(s)(	\tilde{E}_0(\mathbf{W}) 
                +
                (\nu -\nu')\big{(}\mathrm{J}^{\alpha}\mathbf{U}^{\nu},S(\mathbf{U}^{\nu})\mathbf{W}\big{)}_{L^2},
    \end{equation*}
    where the last term can be bounded using the definition of $\mathscr{T}$ and the fact that $\alpha\in(1,\frac{3}{2}]$. In particular, we obtain that
    \begin{equation}\label{Energy for difference}
			\frac{d}{dt} \tilde{E}_0(\mathbf{W}) 
			 \lesssim N(s)(	\tilde{E}_0(\mathbf{W}) 
                +
    (\nu - \nu')(\tilde{E}_0(\mathbf{W}) )^{\frac{1}{2}}).
    \end{equation}
    By \eqref{Bound on solution - data} and definition of $N(s)$, we have that $N(s)\lesssim 1$. Moreover, using Grönwall's inequality on \eqref{Energy for difference} and \eqref{equiv 1} yields,
    \begin{equation*}
        \sup\limits_{t\in[0,\frac{T}{\max\{\ve, \beta\}}]}
		|(\eta,w)(t)|_{Y^{0}_{\mu}} \lesssim \nu - \nu'.
    \end{equation*}
    Then using this estimate combined with interpolation we get that
    \begin{align}
         \sup\limits_{t\in[0,\frac{T}{\max\{\ve, \beta\}}]}| (\eta,w) |_{Y^{s'}_{\mu}} 
	\lesssim	
        (\nu - \nu')^{1-\frac{s'}{s}} 
	\underset{\nu \rightarrow 0}{\longrightarrow}       0,
    \end{align}
    from which we deduce that $\{(\zeta^{\nu}, v^{\nu})\}_{0<\nu\leq 1}$ defines a Cauchy sequence  in $C([0,\frac{T}{\max\{\ve, \beta\}}]; Y^{s'}_{\mu}(\mathbb{R}))\cap L^{\infty}([0,\frac{T}{\max\{\ve, \beta\}}]; Y^s_{\mu}(\mathbb{R}))$ for  $s'\in [0,s)$. Thus, we conclude that there exists a limit by completeness. \\

    \noindent
    \underline{Step 4:} \textit{The solution is bounded by the initial data.} We claim that the solution obtained in Step $3$ satisfies \eqref{bound on solution thm}. 
    
    Indeed, using the notation from the previous step, we deduce by \eqref{Bound on solution - data} that 
    $$\{(\zeta^{\nu},u^{\nu})\}_{0<\nu\leq 1} \subset C([0,\frac{T}{\max\{\ve, \beta\}}];Y^s_{\mu}(\mathbb{R})),$$
    is a bounded sequence in a reflexive Banach space. As a result, we have by Eberlein-\u{S}mulian's Theorem that $(\zeta^{\nu},v^{\nu}) \underset{\nu \rightarrow 0}{\rightharpoonup}  (\zeta,v) $ weakly in $Y^s_{\mu}(\mathbb{R})$ for a.e. $t \in [0,\frac{T}{\max\{\ve, \beta\}}]$. In particular, we have that 
    \begin{equation}\label{bound on any initial data}
    	\sup\limits_{t \in [0,\frac{T}{\max\{\ve, \beta\}}]} |(\zeta,v)|_{V^s_{\mu}} \leq\liminf \limits_{\nu \searrow 0}\sup\limits_{t \in [0,\frac{T}{\max\{\ve, \beta\}}]} |(\zeta^{\nu},v^{\nu})|_{Y^s_{\mu}} \lesssim   | (\zeta_0,v_0) |_{Y^s_{\mu}}.
    \end{equation}

    \noindent
    \underline{Step 5:} \textit{Persistence and continuity of the flow.}  There is a solution  $(\zeta, v) \in C([0,\frac{T}{\max\{\ve, \beta\}}]; V^{s}_{\mu}(\mathbb{R}))$ of \eqref{W-G-N} that depends continuously on the initial data.

    For the proof of this step, we define a new sequence of functions $(\zeta^{\delta}, v^{\delta})$ solving \eqref{W-G-N},  with mollified initial data, i.e.
    \begin{equation*}
    	(\zeta^{\delta}_0, v^{\delta}_0) = (\chi_{\delta}(\mathrm{D})\zeta_0, \chi_{\delta}(\mathrm{D})v_0) \in H^{\infty}(\mathbb{R}): = \cap_{s >0}  H^{s}(\R).
    \end{equation*}
	Reapplying the arguments of Step 1 and Step 2, combined with Proposition \ref{Rate of Decay in norm}, one can deduce that 
	$$(\zeta^{\delta}, v^{\delta}) \in C([0,\frac{T}{\max\{\ve, \beta\}}];H^{\infty}(\mathbb{R})),$$
	 satisfying \eqref{bound on any initial data}. Now that the sequence is well-defined one can again define the difference between two solutions and use Proposition  \ref{Rate of Decay in norm}, together with Proposition \ref{Compactness L2} and Remark \ref{divergence est} to deduce the result. As mentioned above, at this stage in the proof,  the argument is classical and the details can be found in e.g. \cite{BonaSmith75,LinaresPonce15,Paulsen22}.

    \end{proof}

\section{Justification of \eqref{W-G-N} as a water waves model}\label{Justification}

We now give the proof of Theorem \ref{thm Justification}.

\begin{proof}
    First, we let $s\geq 4$ and take initial data $(\zeta_0 ,  \psi_0) \in  H^{s}(\R)\times \dot{H}^{s}(\R)$ and $b\in H^{s+2}(\R)$. Then the solutions of the water waves equations \eqref{WW}:
    $$ (\zeta, \psi) \in C([0,\frac{\tilde{T}}{\max\{\ve, \beta\}}] ; H^{s}(\R)\times \dot{H}^s(\R)),$$
    are given by Theorem $4.16$ in \cite{Alvarez-SamaniegoLannes08a}. Moreover, we can define $\overline{V} \in C([0,\frac{\tilde{T}}{\max\{\ve, \beta\}}];X^s_{\mu}(\R))$.  
    Now, use Proposition \ref{Concictency of new model} and formulation \eqref{Matrix W-G-N} to say that for some $\tilde{T}>0$ the functions $\mathbf{U}= (\zeta , \overline{V})^T$ solves
    \begin{equation*}
        \partial_t \mathbf{U} + (M_1   + (S^{-1}M_2))(\mathbf{U}) \partial_x \mathbf{U} + (S^{-1}Q)(\mathbf{U}) + (S^{-1}Q_b)(\mathbf{U})= \mu^2( \ve + \beta) \mathbf{R},
    \end{equation*}
    for any $t \in [0,\frac{\tilde T}{\max\{\ve, \beta\}}]$ and with $S, M_1,M_2,Q,Q_b$ defined as in \eqref{Matrix W-G-N} and $\mathbf{R}  = (0,R) \in L^{\infty}([0,\frac{\tilde T}{\max\{\ve, \beta\}}] ; X_{\mu}^{r}(\R)) $  for some $r\in \N$.

    The next step is to let $v^{\text{\tiny   WGN} }_0 = \overline{V}|_{t=0} \in X^s_{\mu}(\R)$ and then use Theorem \ref{W-P W-G-N} deduce the existence of $T>0$ such that 
    $$\mathbf{U}^{\text{\tiny   WGN} } = (\zeta^{\text{\tiny   WGN} }, v^{\text{\tiny   WGN} } ) \in C([0,\frac{T}{\max\{\ve, \beta\}}] ; Y_{\mu}^s(\R)),$$
    solves  system \eqref{Matrix W-G-N}:
    \begin{equation*} 
    \partial_t \mathbf{U}^{\text{\tiny   WGN} } + (M_1   + (S^{-1}M_2))(\mathbf{U}^{\text{\tiny   WGN} }) \partial_x \mathbf{U}^{\text{\tiny   WGN} }+ (S^{-1}Q)(\mathbf{U}^{\text{\tiny   WGN} }) + (S^{-1}Q_b)(\mathbf{U}^{\text{\tiny   WGN} })=\mathbf{0},
    \end{equation*}
    for any $t \in [0,\frac{T}{\max\{\ve, \beta\}}]$. Consequently, taking the difference between the two solutions 
    $$\mathbf{W} = (\eta, w)^T = \mathbf{U} - \mathbf{U}^{\text{\tiny   WGN} },$$
    we obtain the following system

\begin{equation}\label{lin W}
		\partial_t \bold{W} +  (M_1 + S^{-1}M_2)(\bold{U})\partial_x \bold{W}= \tilde{\bold{F}}, 
	\end{equation}
	similar to \eqref{lin W}  and with
        \begin{align*}
            \tilde{\bold{F}} 
            & =
            -  \Big[(M_1+(S^{-1}M_2))(\bold{U}) - (M_1+S^{-1}M_2)(\bold{U}^{\text{\tiny   WGN} })  \Big] \partial_x\bold{U}^{\text{\tiny   WGN} }
            \\
            & 
            \hspace{0.5cm}
       		- \Big{[}(S^{-1}(Q+Q_b))(\bold{U}) - (S^{-1}(Q+Q_b))(\bold{U}^{\text{\tiny   WGN} })  \Big]  
           +
           \mu^2( \ve + \beta) \mathbf{R}
           \\
           & = 
           \mathbf{F}+\mu^2( \ve + \beta) \mathbf{R},
        \end{align*}
        for any $t \in [0,\frac{\min\{\tilde{T},T\}}{\max\{\ve, \beta\}}]$. Then using the estimates \eqref{Energy 2},\eqref{equiv 2}, and Remark \ref{divergence est} we deduce  for $r>\frac{3}{2}$ that
        \begin{align*}
            \frac{d}{dt} \tilde{E}_r(\bold{W}) 
            & \lesssim  
            |\big(\mathrm{J}^r\tilde{\bold{F}}, S(\mathbf{U} ) \mathrm{J}^r \bold{W}\big)_{L^2}| +  N(r)  \tilde{E}_r(\mathbf{W})
            \\
            & 
            \lesssim \mu^2( \ve + \beta)|\big(\mathrm{J}^r R, \mathscr{T}[h,\beta b] \mathrm{J}^r w\big)_{L^2}|
            +
            N(r+1)  \tilde{E}_r(\mathbf{W}).
        \end{align*}
        However, by definition of $\mathscr{T}[h,\beta b] $  and using integration by parts, Hölder's inequality and the Sobolev embedding we easily obtain the estimate
        \begin{align*}
            |\big(\mathrm{J}^r R, \mathscr{T}[h,\beta b] \mathrm{J}^r w\big)_{L^2}|
            \lesssim   N(r)|R|_{X^r_{\mu}} |w|_{X^r_{\mu}}.
        \end{align*}
        Gathering these estimates, together with \eqref{equiv 2}, we observe
        \begin{equation*}
             \frac{d}{dt} \tilde{E}_r(\bold{W}) 
            \lesssim \mu^2( \ve + \beta)|R|_{X^r_{\mu}}(\tilde{E}_r(\mathbf{W}))^{\frac{1}{2}} +
            N(r+1)  \tilde{E}_r(\mathbf{W}).
        \end{equation*}
        Now, a simple application of Grönwall's inequality and  \eqref{equiv 2} yields
        \begin{equation} \label{convergence est in proof}
            |(\eta, w)|_{Y^{r}_{\mu}} \lesssim 
            \mu^2( \ve + \beta)t \: |R|_{X^r_{\mu}} e^{N(r+1)t}.
        \end{equation}
        Finally, to conclude we use that  $Y^{r}_{\mu}(\R)\subset H^{r}(\R)  \hookrightarrow L^{\infty}(\R)$ for $r > \frac{3}{2}$, and \eqref{convergence est in proof} to get 
        \begin{align*}
            | \mathbf{U}  - \mathbf{U}^{\text{\tiny   WGN} } |_{L^{\infty}([0,t]; \R)} 
            & \lesssim |(\eta, w)|_{L^{\infty}([0,t];Y^{r}_{\mu}(\R))} 
            \\
            & \lesssim   \mu^2( \ve + \beta)t \:  |R|_{X_{\mu}^{r}} e^{N(r+1)t}.
        \end{align*}
        To conclude, we let $s$ be large enough such that $r+1<s$  to get that
        \begin{align*}
             | \mathbf{U}  - \mathbf{U}^{\text{\tiny   WGN} } |_{L^{\infty}([0,t]; \R)} 
            \lesssim \mu^2( \ve + \beta)t,
        \end{align*}
        for all $t \in [0,\frac{\min\{\tilde{T},T\}}{\max\{\ve, \beta\}}]$.

\end{proof}

	\section*{Acknowledgements}
	 
	This research was supported by a Trond Mohn Foundation grant. It was also supported by the Faculty Development Competitive Research Grants Program 2022-2024 of Nazarbayev University: Nonlinear Partial Differential Equations in Material Science, Ref. 11022021FD2929. 
 

    \bibliographystyle{plain}
    \bibliography{Biblio}

\end{document}